\newif\ifpersonal
\personaltrue 
\documentclass[11pt, twoside]{article} 

\usepackage{relsize,amsmath,amsthm,mathtools,bm,tensor} 
\usepackage{microtype,lmodern} 
\usepackage[utf8]{inputenc} 
\usepackage[T1]{fontenc} 
\usepackage{enumerate,comment,braket,xspace,tikz-cd,csquotes} 
\usepackage[centering, vscale=0.8, hscale=0.75]{geometry}
\usepackage[textwidth=20mm, textsize=tiny]{todonotes}
\usepackage{xr}
\usepackage[hidelinks,linktoc=all]{hyperref}

\usepackage{epigraph}
\usepackage[garamond,cal=cmcal]{mathdesign}
\usepackage{eucal}

\usepackage[capitalize]{cleveref}
\usepackage{fancyhdr}
\usepackage{lscape}
\usepackage{adjustbox}
\usepackage{faktor}
\numberwithin{equation}{section}
\usepackage{appendix}
\usepackage{stmaryrd}
\usepackage{subfiles}
\theoremstyle{plain}
\newtheorem{thm-intro}{Theorem}[section]
\newtheorem{thm}{Theorem}[section]
\newtheorem{lem}[thm]{Lemma}

\newtheorem{prop}[thm]{Proposition}

\newtheorem{cor}[thm]{Corollary}

\theoremstyle{definition}
\newtheorem{defin}[thm]{Definition}
\newtheorem{recall}[thm]{Recall}
\newtheorem{notation}[thm]{Notation}
\newtheorem{eg}[thm]{Example}
\newtheorem{rem}[thm]{Remark}
\newtheorem{construction}[thm]{Construction}
\newtheorem{note}[thm]{Note}
\newtheorem{warning}[thm]{Warning}
\ifpersonal
\newcommand*{\personal}[1]{\textcolor[rgb]{0,0,1}{#1}}
\else
\newcommand*{\personal}[1]{\ignorespaces}
\renewcommand*{\todo}[1]{\ignorespaces}
\fi

\DeclareFontFamily{U}{mathx}{}
\DeclareFontShape{U}{mathx}{m}{n}{<-> mathx10}{}
\DeclareSymbolFont{mathx}{U}{mathx}{m}{n}
\DeclareMathAccent{\widehat}{0}{mathx}{"70}
\DeclareMathAccent{\widecheck}{0}{mathx}{"71}

\makeatletter
\newcommand{\subjclass}[2][2020]{%
	\let\@oldtitle\@title%
	\gdef\@title{\@oldtitle\footnotetext{#1 \emph{Mathematics subject classification.} #2}}%
}
\newcommand{\keywords}[1]{%
	\let\@@oldtitle\@title%
	\gdef\@title{\@@oldtitle\footnotetext{\emph{Key words and phrases.} #1.}}%
}
\makeatother

\newcommand{\rr}{\mathbb R}

\renewcommand{\epsilon}{\varepsilon}

\newcommand{\C}{\mathbb C}

\newcommand{\E}{\mathbb E}
\newcommand{\Q}{\mathbb Q}

\newcommand{\cA}{\mathcal A}
\newcommand{\cB}{\mathcal B}
\newcommand{\cc}{\mathcal C}
\newcommand{\cD}{\mathcal D}
\newcommand{\cE}{\mathcal E}
\newcommand{\cF}{\mathcal F}
\newcommand{\cH}{\mathcal H}
\newcommand{\cG}{\mathcal G}
\newcommand{\cI}{\mathcal I}
\newcommand{\cJ}{\mathcal J}
\newcommand{\cK}{\mathcal K}

\newcommand{\cO}{\mathcal O}
\newcommand{\cP}{\mathcal P}

\newcommand{\cS}{\mathcal S}
\newcommand{\cT}{{\mathcal T}}
\newcommand{\cU}{\mathcal U}

\newcommand{\cX}{\mathcal X}
\newcommand{\cY}{\mathcal Y}
\newcommand{\cZ}{\mathcal Z}
\DeclareFontFamily{U}{BOONDOX-calo}{\skewchar\font=45 }
\DeclareFontShape{U}{BOONDOX-calo}{m}{n}{<-> s*[1.05] BOONDOX-r-calo}{}
\DeclareFontShape{U}{BOONDOX-calo}{b}{n}{<-> s*[1.05] BOONDOX-b-calo}{}
\DeclareMathAlphabet{\mathcalboondox}{U}{BOONDOX-calo}{m}{n}

\newcommand{\bDelta}{\boldsymbol\Delta}

\newcommand{\Gm}{{\mathbb G_{\textup{m}}}}

\newcommand{\Aut}{\textup{Aut}}
\newcommand{\Sch}{\textup{Sch}}

\newcommand{\Mod}{\textup{Mod}}
\newcommand{\fd}{\textup{fd}}

\newcommand{\op}{\textup{op}}
\newcommand{\Rep}{\textup{Rep}}

\newcommand{\Bun}{\textup{Bun}_G}

\newcommand{\PSh}{\textup{PSh}}
\newcommand{\Sh}{\textup{Sh}}
\newcommand{\Shv}{\textup{Shv}}

\renewcommand{\Pr}{\textup{Pr}}

\newcommand{\Grpd}{\textup{Grpd}}

\newcommand{\Coh}{\textup{Coh}}

\newcommand{\Aff}{\textup{Aff}}

\newcommand{\Top}{{\textup{Top}}}

\newcommand{\Ran}{\textup{Ran}}

\newcommand{\Alg}{\textup{Alg}}
\newcommand{\CAlg}{\textup{\Alg}_\C}

\newcommand{\Fact}{\textup{Fact}(M)^\otimes}

\newcommand{\Corr}{\textup{Corr}}

\newcommand{\Cat}{\mathcal C\textup{at}_\infty}
\newcommand{\cat}{\mathcal C\textup{at}}

\newcommand{\sets}{\textup{Set}}

\newcommand{\Disk}{\textup{Disk}}
\newcommand{\loc}{\textsf{loc}}
\newcommand{\Prl}{\mathcal P\textup{r}^{\textup L}}
\newcommand{\Prlo}{\mathcal P\textup{r}^{\textup L,\otimes}}
\newcommand{\Prr}{\mathcal P\textup{r}^{\textup{R}}}
\newcommand{\Prro}{\mathcal P\textup{r}^{\textup{R},\otimes}}
\newcommand{\Prlop}{\mathcal P \textup{r}^{\textup{L}, \otimes\textup{-}\op}}

\newcommand{\ett}{\textsf{\'et}}
\newcommand{\Gr}{\textup{Gr}}
\newcommand{\QCoh}{{\textup{QCoh}}}

\newcommand{\an}{\textup{an}}
\newcommand{\const}{\textup{const}}
\newcommand{\Open}{\textup{Open}}

\newcommand{\inj}{\textup{inj}}

\newcommand{\dinj}{\bDelta_\inj^\op}

\newcommand{\nun}{\textup{nu}}
\newcommand{\glob}{\textup{glob}}
\renewcommand{\vert}{\textsf{vert}}
\newcommand{\horiz}{\textsf{horiz}}
\newcommand{\isom}{\textsf{isom}}
\newcommand{\adm}{\textsf{adm}}
\newcommand{\all}{\textsf{all}}

\newcommand{\coadm}{\textsf{co}\textup{-}\textsf{adm}}

\newcommand{\id}{\textup{id}}
\newcommand{\union}{\textup{union}}
\newcommand{\disj}{\textup{disj}}

\newcommand{\hdis}{(\hck_{\Ran,k}\times \hck_{\Ran,k})_{\disj}}

\newcommand{\Sing}{\textup{Sing}}
\newcommand{\dc}{\cD_\textup{c}}

\newcommand{\Perv}{{\mathcal P}\textup{erv}}

\newcommand{\triv}{|_{X_R\setminus \Gamma_S}\xrightarrow{\sim}\cT|_{X_R\setminus \Gamma_S}}

\newcommand{\Stk}{\textup{Stk}}
\newcommand{\topp}{\textup{top}}

\newcommand{\Conv}{\textup{Conv}}
\newcommand{\hck}{{\textup{Hck}}}
\newcommand{\Hck}{\cH{\textup{ck}}}

\newcommand{\red}{{\textup{red}}}

\newcommand{\con}{\textup{con}}

\newcommand{\Sph}{{\textup{Sph}}}
\newcommand{\ren}{{\textup{ren}}}
\newcommand{\locc}{{\textup{loc.c}}}

\newcommand{\str}{\textsf{str}}

\newcommand{\Str}{\textup{Str}}
\newcommand{\Cw}{\textup{Cw}}
\newcommand{\Poset}{\textup{Poset}}
\newcommand{\Cons}{\textup{Cons}}

\newcommand{\Grp}{\textup{Grp}}

\newcommand{\she}{\mathsf{she}}

\newcommand{\Alex}{\textup{Alex}}

\newcommand{\Exit}{\textup{Exit}}
\newcommand{\corr}{\textup{corr}}

\newcommand{\sS}{\mathscr S}

\newcommand{\Lan}{\textup{Lan}}
\newcommand{\Fun}{\textup{Fun}}

\usetikzlibrary{decorations.markings} 
\tikzset{
	closed/.style = {decoration = {markings, mark = at position 0.5 with { \node[transform shape, xscale = .8, yscale=.4] {/}; } }, postaction = {decorate} },
	open/.style = {decoration = {markings, mark = at position 0.5 with { \node[transform shape, scale = .7] {$\circ$}; } }, postaction = {decorate} }
}

\DeclareMathOperator{\Hom}{Hom}

\DeclareMathOperator{\Map}{Map}
\DeclareMathOperator{\Mor}{Mor}

\DeclareMathOperator{\Spec}{\textup{Spec}}

\DeclareMathOperator{\Sym}{Sym}

\DeclareMathOperator*{\colim}{colim}

\renewcommand{\Pr}{\mathcal P\textup{r}}

\newcommand{\GL}{{\textup{GL}}}

\newcommand{\xt}{{\mathbb X_\bullet(T)^+}}

\newcommand{\Inv}{\textup{Inv}}

\newcommand{\hrk}{\hck_{\Ran,k}}

\newcommand{\fact}{{\textup{fact}}}
\newcommand{\hf}{\hck^{\fact}}
\newcommand{\hfb}{{\hck^\fact_{\bullet}}}
\newcommand{\Ind}{{\textup{Ind}}}
\newcommand{\oD}{{\mathring D}}

\newcommand{\pb}{\textsf{hcb}}
\newcommand{\lft}{\textup{lft}}

\renewcommand{\triv}{\xrightarrow{\sim}}

\newcommand{\pushcons}{\dashv}
\newcommand{\Prlr}{\mathcal{P}\textup{r}^{\textup{LR}}}
\newcommand{\Prlro}{\mathcal{P}\textup{r}^{\textup{LR}, \otimes}}

\newcommand{\oo}{\mathcal O}

\newcommand{\ex}{\textup{ex}}

\newcommand{\nn}{\mathbb N}

\newcommand{\taylor}{\llbracket t \rrbracket}
\newcommand{\laurent}{(\!(\!t\!)\!)}

\newcommand{\fs}{\textup{Fin}_{\geq 1,\textup{surj}}}
\newcommand{\smooth}{\textsf{smooth}}

\tikzset{
	closed/.style = {decoration = {markings, mark = at position 0.5 with { \node[transform shape, xscale = .8, yscale=.4] {$\backslash$}; } }, postaction = {decorate} },
	open/.style = {decoration = {markings, mark = at position 0.5 with { \node[transform shape, scale = .7] {$\circ$}; } }, postaction = {decorate} }
}

\newcommand{\A}{\mathbb A}

\newcommand{\tC}{\textup{C}}
\newcommand{\tloc}{\textup{loc}}

\newcommand{\TStk}{\textup{TStk}}
\newcommand{\Strat}{\textup{StrTop}}

\newcommand{\PStrTStk}{\widehat{\Str\TStk}}
\newcommand{\PStrat}{\widehat{\Str\TStk}}

\newcommand{\PConproeq}{\widehat{\Str\TStk_{\con,\textup{proeq}}}}

\newcommand{\StrStklft}{{\Str\Stk^\lft_\C}}
\newcommand{\StrTStk}{{\Str\TStk}}

\newcommand{\PStrStk}{\widehat{\Str\Stk_\C^\lft}}
\newcommand{\ConStk}{\Str\TStk_\con}
\newcommand{\PCon}{\widehat{\Str\TStk_\con}}
\newcommand{\Con}{\Str\Top_\con}
\newcommand{\subm}{\textsf{subm}}
\newcommand{\psubm}{{\widehat{\textsf{subm}}}}

\newcommand{\Pro}{\textup{Pro}}
\newcommand{\equ}{\textup{eq}}

\newcommand{\lCatE}{\widehat{\cat}_{\infty,\cE}}

\newcommand{\Perf}{\textup{Perf}}
\newcommand{\cont}{\textup{cont}}
\newcommand{\st}{\textup{st}}
\newcommand{\Dmod}{\textup{DMod}}
\newcommand{\PrLst}{\Prl_\st}
\newcommand{\Catex}{\Cat^\ex}

\newcommand{\tors}{\textup{tors}}
\newcommand{\Ac}{\mathbb A_{\mathbb C}^1}
\newcommand{\MDisk}{\textup{MDisk}}
\newcommand{\tri}{\mathsf{tri}}
\newcommand{\uni}{\mathsf{uni}}

\newcommand{\pshe}{\widehat{\she}}
\newcommand{\can}{\textup{can}}

\newcommand{\Loc}{\textup{Loc}}

\title{A model for the $\E_3$ fusion-convolution product of constructible sheaves on the affine Grassmannian}

\date{\today}
\author{Guglielmo Nocera\footnote{Institut des Hautes \'Etudes Scientifiques, 35 Rte de Chartres, 91440 Bures-sur-Yvette, France. \texttt{nocera-at-ihes.fr}.}}

\begin{document}
	
	\maketitle
	\begin{abstract}
		Let $G$ be a complex reductive group. The spherical Hecke category of $G$ can be presented as the category of $G_\cO$-equivariant constructible sheaves on the affine Grassmannian $\Gr_G$. This category admits a convolution product, extending the convolution product of equivariant perverse sheaves. 
		In this paper, we upgrade the mentioned convolution product to a left t-exact $\E_3$-monoidal structure in $\infty$-categories. 
		The construction is intrinsic to the automorphic side.
		Our main tools are the Beilinson--Drinfeld Grassmannian, Lurie's characterization of $\E_k$-algebras via the topological Ran space, 
		the homotopy theory of stratified spaces and the formalism of correspondences.
	\end{abstract}

	\tableofcontents

	\epigraph{[...] Non c'eravamo accorti \\ di un buco tra i rampicanti.}{\textsc{E. Montale}, A pianterreno, \textit{Satura II}}
	
	\section{Introduction}
	\subsection{Main results}
	Let $G$ be a complex reductive group, and $R$ a commutative ring of coefficients. The aim of this paper is to provide an extension of the convolution product in the Satake category of equivariant perverse sheaves on the affine Grassmannian (\cite[§4]{MV}) to the spherical Hecke category (\cite[§12.2.1]{Arinkin-Gaitsgory}), and endow this extension with an $\E_3$-algebra structure in $\infty$-categories. This upgrade is the derived avatar of Mirkovic and Vilonen's commutativity constraint \cite[§5]{MV}. We briefly illustrate the main results of the paper below.
	
	\begin{defin}[Affine Grassmannian]\label{affine-Grassmannian}Let $G$ be a reductive group. The \textit{arc group} $G_\cO$ (also denoted by $G\taylor$ or $\textup{L}^+G$) is defined as the functor 
		\begin{gather*}{\Aff_\C^\op}\to \Grp \\
			R\mapsto G(R\taylor)=\Hom(R\llbracket t \rrbracket,G).
		\end{gather*} The \textit{loop group} $G_\cK$ (also denoted by $G\laurent$ or $\textup{L}G$) is defined as the ind-representable functor 
		\begin{gather*}\Aff_\C^\op\to \Grp\\
			R\mapsto G(R\laurent)=\Hom(R\laurent,G).
		\end{gather*} 
		The \textit{affine Grassmannian} $\Gr_G$ is the fpqc quotient stack $$\Gr_G=\big [G_\cK/G_\cO\big ].$$\end{defin}

	When there is no ambiguity, we usually denote $\Gr_G$ by just $\Gr$.
	\begin{rem}The stack $\Gr$ is actually ind-representable. As such, it admits an underlying complex-analytic space $\Gr^\an$.\end{rem}

	\begin{notation}
		Let $\cE$ be a symmetric monoidal presentable stable $\infty$-category, and $\cE^\omega$ its stable subcategory of compact objects. Let $\Prlo_\cE$ be the symmetric monoidal $\infty$-category of $\cE$-linear presentable $\infty$-categories and left adjoint functors between them; 
	 Let also $\cat_{\infty, \cE^\omega}^\times$ be the $\infty$-category of $\cE^\omega$-linear small $\infty$-categories (see \cref{small-linear-categories}) with its cartesian symmetric monoidal structure. 

	When $R$ is a commutative ring and $\cE=\Mod_R$ is the $\infty$-category of $R$-modules, we use the notations $\Prlo_R,
	\cat_{\infty,R}$.
		\end{notation}

		\begin{thm}[Main result, {\cref{final-theorem}}]\label{final-theorem-in-introduction}Let $G$ be a complex reductive group and $\cE$ a symmetric monoidal presentable stable $\infty$-category. Then there exists an object $$\Sph(G;\cE)^{\otimes}\in \Alg_{\E_3}(\Prlo_\cE)$$ whose underlying $\infty$-category is $$\Cons_{G_\cO^\an}(\Gr^\an;\cE),$$ the unbounded derived $\infty$-category of topological $G_\cO^\an$-equivariant constructible sheaves over $\Gr^\an$, with coefficients in $\cE$.\end{thm}
	
	\begin{cor}[{\cref{corollary-small-subcategory}, \cref{t-exactness-of-convolution}}]\label{corollary-small-subcategory-introduction}In the same setting as \cref{final-theorem-in-introduction}, there exists an object $$\Sph(G;\cE)^{\textup{loc.c},\otimes}\in \Alg_{\E_3}(\cat_{\infty,\cE^\omega}^\times)$$ whose underlying $\infty$-category is the small spherical Hecke category of $G$ (see \cref{spherical-categories}).
		
	Let $R$ be a discrete ring. For $\cE=\Mod_R$, this $\E_3$-structure is left t-exact for the perverse t-structure (and exact if $R$ is a field). It canonically induces a symmetric monoidal structure on the abelian subcategory of equivariant perverse sheaves, coinciding with the classical convolution product of \cite[§4]{MV}.\end{cor}


\begin{rem}\label{remark-coefficients-intro}Let $R$ be a ring. As explained in \cref{Section-models-Sph}, the category $\Sph(G;R)^\locc=\Sph(G;\Mod_R^\omega)$ has several presentations. 

One is the one that we use as definition, namely the $\infty$-category of $G_\cO^\an$-equivariant constructible sheaves over $\Gr^\an_G$ with values in $R$, with bounded finitely presented stalks. 

When $R=\C$, then the Riemann-Hilbert correspondence implies that $\Sph(G;R)^\locc$ can be presented as the subcategory of $\Dmod_{G_\cO}(\Gr_G)$ spanned by objects whose underlying D-module is compact (i.e. which become compact after forgetting the equivariant structure), which agrees with \cite[12.2.3]{Arinkin-Gaitsgory}, see \cref{comparison-dmodules}. 

If $R$ is finite, profinite or $\ell$-adic (i.e. an algebraic extension of $\Q_\ell$), then $\Sph(G;R)^\locc$ can be presented as a category of equivariant \textit{\'etale} sheaves over the ind-scheme $\Gr_G$, see \cref{comparison-with-algebraic-l-adic}.\end{rem}

	Another direct corollary of \cref{final-theorem-in-introduction} regards the renormalized spherical Hecke category:
	\begin{cor}[\cref{corollary-Ind-completion}]\label{corollary-Ind-completion-introduction}
		In the same setting as \cref{corollary-small-subcategory-introduction}, there is also an object $$\Sph(G;\cE)^{\ren,\otimes}\in \Alg_{\E_3}(\Prlo_\cE)$$ whose underlying $\infty$-category is the renormalized spherical $\infty$-category $$\Sph(G;\cE)^\ren=\Ind(\Sph(G;\cE)^\textup{loc.c})$$ appearing e.g. in \cite[§12]{Arinkin-Gaitsgory}.
	\end{cor}
	
When $\cE=\Mod_k$ for $k$ a ring, we denote $\Sph(G;\cE)^\ren$ by $\Sph(G;k)^\ren$.
	We are now going to provide some context and motivation for these results (\cref{Introduction-Geometric-Langlands}) and an outline of the paper (\cref{Introduction-sketch}).

	\subsection{Motivation: the Derived Satake Theorem}\label{Introduction-Geometric-Langlands}
	A classical problem in representation theory is the study of a reductive group $G$ over a local field (e.g. $\textup{GL}_n$, $\textup{SL}_n$, $\mathbb P\textup{GL}_n$) and its Langlands dual $\check G$ (e.g. $\check{\textup{GL}}_n=\textup{GL}_n,\check {\textup{SL}}_n=\mathbb P\textup{GL}_n$). 
	
	A celebrated result in the study of the Langlands duality is the Satake theorem \cite{Satake} which, given a reductive group $G$ over $\mathbb F_p$, establishes an isomorphism between the $\C$-algebra of complex compactly supported $G(\mathbb Z_p)$-biinvariant functions on $G(\mathbb Q_p)$, called the \textit{(spherical) Hecke algebra} of $G$, and the (complexified) Grothendieck ring of finite-dimensional representations of $\check G$. Ginzburg \cite{Ginzburg} and later Mirkovic and Vilonen \cite[(13.1)]{MV} provided a ``sheaf theoretic'' analogue (actually a categorification) of this theorem, called the \textit{Geometric Satake Equivalence}: here $G$ is a \textit{complex} reductive group, and the statement has the form of an equivalence of symmetric monoidal abelian categories between the category of \textit{equivariant perverse sheaves} $\Perv_{G_\cO}(\Gr_G)$ and the category of finite dimensional representations of $\check G$. The key new object here is the affine Grassmannian $\Gr_G$, whose definition we recalled above (\cref{affine-Grassmannian}). This is an infinite dimensional algebro-geometric object with the property that $\Gr_G(\C)=G(\C(\!(t)\!))/G(\C\llbracket t\rrbracket)$. The Grothendieck group of $\Perv_{G_\cO}(\Gr_G)$ is the analogue of the Hecke algebra appearing in the original Satake theorem, but now for a reductive group over $\C$.\footnote{A closer analogue to the original Satake isomorphism is given by the geometric Satake theorem in mixed characteristic, see \cite{Zhu-mixed}.}

	\begin{thm}[Geometric Satake Equivalence, {\cite[(1.1)]{MV}}]\label{Geometric-Satake}Fix a reductive algebraic group $G$ over $\C$, and a discrete commutative ring $R$, noetherian and of finite global dimension. There exists a symmetric monoidal structure $\star$ on $\Perv_{G_\cO}(\Gr_G;R)$, called \textup{convolution}, and an equivalence of symmetric monoidal abelian categories \begin{equation}\label{geometric-Satake-introduction}({\Perv}_{G_\mathcal O}(\Gr_G;R),\star)\simeq (\Rep^{\mathrm{fd}}(\check G_R,R),\otimes)\end{equation} where $\check G_R$ is the Langlands dual of $G$ over $R$ \cite[Beginning of Sec.\ 12]{MV} and $\otimes$  denotes the standard tensor product of finite-dimensional $\check G_R$-representations with coefficients in $R$.
	\end{thm}
	We recall the meaning of this statement, together with various theoretical recollections necessary for this paper, in Appendix \ref{appendix-Geometric-Satake}. We refer the reader seeking for a complete survey to \cite{Zhu} and \cite{Baumann-Riche}.

	\begin{thm}[Derived Satake Theorem, {\cite[Theorem 5]{Bezrukavnikov-Finkelberg}}]\label{Derived-Satake}Let $G$ be a complex reductive group and $k$ a field of coefficients of characteristic zero. There is a \textup{monoidal} equivalence of triangulated categories\footnote{For the sake of coherence with the rest of the work, we adopt the notation $\textup{h}-$ in order to refer to ``the homotopy category of a stable $\infty$-category''. Of course, in the original paper both sides are defined directly as triangulated categories.} \begin{equation}\label{equation-derived-Satake}\textup{h}\Cons_{G_\cO}^{\textup{fd}}(\Gr_G;k)\simeq \textup{h}\Perf_{\check G_k}(\textup{Sym}(\check{ \mathfrak g}_k[-2]))\end{equation} where $\check G_k$ is the Langlands dual of $G$ over $k$ and $\check{\mathfrak g}_k$ is its Lie algebra.\end{thm}

	\begin{rem}\label{heart-of-derived-Satake}	Here $\Cons^\fd_{G_\cO}(\Gr;k)$ is the bounded derived $\infty$-category of $G_\cO$-equivariant constructible sheaves over $\Gr$ with coefficients in $k$ and finitely presentable stalks. The category $\Perv_{G_\cO}(\Gr_G,k)$ is the heart of a t-structure on $\Cons_{G_\cO}^{\textup{fd}}(\Gr_G,k)$ (and hence on the homotopy category). Indeed, this t-stucture is inherited from the presentation of the equivariant constructible category \`a la Bernstein-Lunts, see \cref{t-exactness-of-convolution}. As explained in \cref{t-exactness-of-convolution}, the Geometric Satake Theorem can be formally recovered from the Derived Satake Theorem by passing to the heart, up to a detail: a priori, the induced statement will only be a monoidal equivalence of monoidal abelian categories, and not a symmetric monoidal equivalence. 
		
		Both the left and right-hand side, as triangulated categories, carry a symmetric monoidal structure (for the left-hand-side, see \cite[Section 3.3]{Achar-Riche}; on the right-hand-side, it is the tensor product described in \cite[2.7]{Bezrukavnikov-Finkelberg}). However, the equivalence is \textit{not} symmetric (or even braided) monoidal (cf. \cite[Remark 12.4.3]{Arinkin-Gaitsgory}).\end{rem}
	
	In the following, the notion of $\E_n$-center of an $\E_n$-$\infty$-category we are referring to is \cite[Definition 5.3.1.6, Example 5.3.1.13]{HA}, and generalizes the notion of Drinfeld center.
	
	\begin{prop}\label{center-of-rep} Under the hypotheses of \cref{Derived-Satake}, there is a monoidal equivalence of $\infty$-categories  \begin{equation}\label{equation:center-of-rep}\QCoh(\Spec(\textup{Sym}(\check{ \mathfrak g}_k[1]))/\check G_k)\simeq\textup{Z}_{\E_2}(\textup{DRep}(\check G; k))\end{equation} where $\textup{Z}_{\E_2}$ stays for ``$\E_2$-center'' and $\textup{DRep}(\check G; k)$ is the derived $\infty$-category of representations seen as an $\E_2$-$\infty$-category by forgetting its $\E_\infty$-monoidal structure $\otimes$ along the map of operads $\E_2\to\E_\infty$.
	\end{prop}

	One can make this result follow from work of Ben-Zvi, Francis, Nadler and Preygel on centers \cite{BZFN}, \cite{BenZvi-Nadler-Preygel}: see \cite[§1.3]{Campbell-Raskin-paper} and \cite[(17.1.2)]{Relative-Langlands}.

	\begin{recall}Note that the left-hand-side of \eqref{equation:center-of-rep} is related to $\Perf_{\check G_k}(\textup{Sym}(\check{ \mathfrak g}_k[-2]))$, i.e. the $\infty$-category whose homotopy category appears in the right-hand-side of \cref{Derived-Satake}. More precisely, by {\cite[Proposition 12.4.2]{Arinkin-Gaitsgory}}, there is an equivalence $$\Ind\Coh(\Spec(\textup{Sym}(\check{ \mathfrak g}_k[1]))/\check G_k)\simeq \Mod_{\check G_k}(\textup{Sym}(\check{ \mathfrak g}_k[-2])).$$
	The right-hand-side of this equivalence is the Ind-completion of the right-hand-side of \eqref{equation-derived-Satake}, whereas the left-hand-side can be thought of as a renormalization of the left-hand-side of \eqref{equation:center-of-rep}. 

	In recent work appeared during the revision of the present paper, Campbell and Raskin proved the following result \cite[Theorem 6.6.1]{Campbell-Raskin-paper}. Assume $k=\C$ (or more generally, that $G$ is a reductive group over a field $k$ of characteristic zero). Then, with the notation of \cref{corollary-Ind-completion-introduction}, then there is an equivalence of $\infty$-categories \begin{equation}\label{CR-Derived-Satake}\Sph(G;k)^\ren\simeq\Ind\Coh(\Spec(\Sym(\check{\mathfrak g}_k[1]))/\check G_k)\end{equation}
		which lifts to an equivalence of \textit{factorizable monoidal categories}.
		 
		  This is the correct statement of a result announced by Gaitsgory and Lurie several years ago, and originally conjectured by Drinfeld (see e.g. \cite[footnote 19]{Arinkin-Gaitsgory}). In particular, it is an $\infty$-categorical enhancement of \cref{Derived-Satake}, which can be deduced from \eqref{CR-Derived-Satake} by passing to compact objects and taking the homotopy category.
\end{recall}

\begin{rem}[Role of the present paper]\label{rem-what-we-do-intro}
	The right-hand side of \eqref{CR-Derived-Satake} has a natural $\E_3$-monoidal structure coming from the fact that it is a renormalization of the $\E_2$-center of $\textup{DRep}(\check G_k;k)$, as per \cref{center-of-rep}.
	
	What we do in this paper is rather to endow the left-hand-side of \eqref{CR-Derived-Satake} with an $\E_3$-monoidal structure: see \cref{corollary-Ind-completion-introduction}. Our construction is intrinsic to the automorphic side, i.e. it does not use \eqref{CR-Derived-Satake}. In contrast to \cite{Campbell-Raskin-paper}, we work with a $G$ is defined over $\C$ (and not over an arbitrary field of characteristic zero): the reason for this will be evident from \cref{remark-E3-factmon-intro}. However, this also gives us freedom in the choice of coefficients, see \cref{remark-coefficients-intro}.
	\end{rem}
	
	\begin{rem}\label{remark-E3-factmon-intro}Note that an $\E_3$-monoidal structure is a slightly stronger notion than being factorizable monoidal. 

	More precisely, by the Dunn--Lurie Additivity Theorem \cite[Theorem 5.1.2.2]{HA} an $\E_3$-monoidal structure decomposes into an $\E_1$-monoidal and an $\E_2$-monoidal structure on the same $\infty$-category, which distribute with one another (a higher avatar of the Eckmann-Hilton principle). An $\E_1$-monoidal structure is just a monoidal structure. An $\E_2$-structure is the same as a braided monoidal structure. 

	In the expression ``factorizable monoidal'', the ``monoidal'' part corresponds to the $\E_1$-monoidal structure mentioned above. The ``factorizable'' part is related to the mentioned $\E_2$-monoidal structure as follows: the existence of an $\E_2$-monoidal structure implies the existence of a structure of factorizable category, but not vice-versa: the gap lies precisely in a notion of ``local constancy'': a ``locally constant'' factorizable structure induces an $\E_2$-monoidal structure. Formally, this is exactly the constructibility property appearing in \cref{recall-factalg-intro} below.
	\end{rem}

\begin{rem}\label{Tannaka}Our result is somehow in the same spirit of the Tannakian reconstruction principle used in the proof of the Geometric Satake Theorem (\cref{Geometric-Satake}), where the existence of a symmetric monoidal structure on $\Perv_{G_\cO}(\Gr;R)$ is a part of the structure needed to apply the reconstruction machinery, and only a posteriori it is interpreted as corresponding to the tensor product in $\Rep^\fd(\check G_R;R)$. 
\end{rem}
	
\begin{rem}
	In light of \cref{rem-what-we-do-intro}, it is natural to expect an $\E_3$-monoidal equivalence between the two sides of \eqref{CR-Derived-Satake}, refining the factorizable monoidal equivalence proved by Campbell and Raskin.
\end{rem}

\begin{note}A reasonable question is whether the $\E_2$-component of our $\E_3$-structure is related to some instance of nearby cycles: this question can be thought of as a lift of the discussions in \cite[§3, in particular Proposition 3.3.7]{Achar-Riche} from perverse to constructible sheaves. This is the subject of work in progress with Marius Kj\ae rsgaard, Dmitry Kubrak and Qixiang Wang.\end{note}
	
	\subsection{Outline of the work}\label{Introduction-sketch}
	Let $G$ be complex reductive group. Recall \cref{affine-Grassmannian}.

	In \textit{\cref{Section-2}} we recall that, for any choice of a smooth complex curve $X$, there exists a presheaf $\Ran(X)$, defined as the colimit in $\PSh(\Sch_\C)$ of the diagram \begin{gather*}\fs^\op\to \Ind\Sch_\C\\
	I\mapsto X^I\end{gather*}
where $\fs$ is the category of nonempty finite sets with surjections between them, and the diagram sends a surjection $I\twoheadrightarrow J$ to the corresponding diagonal $X^J\to X^I$.

We also recall that there exists a presheaf $\Gr_\Ran$ called the Ran Grassmannian, living over $\Ran(X)$ and such that for any choice of $x_0\in X(\C)$ the singleton map $\{x_0\}:\Spec \C\to \Ran(X)$ induces a pullback square \begin{equation}\label{pullback-of-GrBD-to-point-intro}\begin{tikzcd}\Gr\arrow[r]\arrow[d]&\Gr_\Ran\arrow[d]\\
\Spec \C\arrow[r, "\{x_0\}"]&\Ran(X)\end{tikzcd}.\end{equation}
Such a presheaf arises as the colimit in $I\in \fs^\op$ of the \textit{Beilinson-Drinfeld Grassmannians}
 \begin{equation}\label{BD-Grassmannian-intro}
	\Gr_I=\{x_I\in X^I,\cF\in \Bun(X),\alpha:\cF|_{X\setminus x_I}\triv\cT|_{X\setminus x_I}\}.
\end{equation} Here $\cT$ is the trivial $G$-bundle, see \cref{defin-GrBD}. The so-called ``moduli interpretation'' of the affine Grassmannian (\cref{moduli-interpretation}) implies the existence of the pullback diagram \eqref{pullback-of-GrBD-to-point-intro}.

The point of view involving $\Ran(X)$ is already used for instance in \cite{Zhu}, \cite{Gaitsgory-Tamagawa} and \cite{James-GrRan}. We consider an equivariant version of this phenomenon: first of all, recall that $\Gr$ admits an action of $G_\cO$ by left multiplication (see \cref{perverse-and-stratification}). We define $$\hck$$ as an ind-pro-stack whose realization is the fpqc quotient $$[G_\cO\backslash \Gr],$$ see \cref{Hecke-stack}, \cref{Hck-as-ind-pro}. Just like in the case of $\Gr_\Ran$ in \eqref{pullback-of-GrBD-to-point-intro}, there is an object $\hck_\Ran$ fitting in a pullback square 
\begin{equation}\label{pullback-of-hck-to-point-intro}\begin{tikzcd}\hck\arrow[d]\arrow[r]&\hck_\Ran\arrow[d]\\
\Spec\C\arrow[r,"\{x_0\}"]&\Ran(X)\end{tikzcd}.\end{equation}
Here $\hck_\Ran$ is an object of the category $\PStrStk$, a suitable pro-completion and free cocompletion of the category of stratified stacks (see \cref{stratified-prestacks}): this is the right environment for our constructions since we want to keep track of the fact that our objects can be approximated by finite-dimensional objects at various levels. In particular, this will allow later to define categories of constructible sheaves in the right way (e.g. as a colimit along the coweight filtration on the affine Grassmannian). 

All objects in sight have natural stratifications, ultimately coming from the fact that the classical stratification in Schubert cells of the affine Grassmannian (\cref{perverse-and-stratification}) can be extended to the Beilinson-Drinfeld Grassmannian. Stratifications are crucial in this work because of the stratified-homotopy-invariance features enjoyed by the procedure of taking constructible sheaves with respect to a given stratification (as opposed to \textit{some} stratification, cf. \cref{varying-strat}, \cref{topological-dc} for the distinction). For this reason, we always work with \textit{stratified} stacks and variations thereof.

There exists a span \begin{equation}\label{convolution-intro}\begin{tikzcd}
	&\hck_2\arrow[ld,"\overline p"']\arrow[rd,"\overline m"]&\\
	\hck\times \hck&&\hck
\end{tikzcd}\end{equation} which we call ``convolution diagram''. 
This span admits a ``Ran version'' of the form
 \begin{equation}\label{Ran-convolution-intro}\begin{tikzcd}
 		&\hck_{\Ran,2}\arrow[ld,"\overline p_\Ran"']\arrow[rd,"\overline m_\Ran"]&\\
 		\hck_\Ran\times \hck_\Ran&&\hck_\Ran
 	\end{tikzcd}.\end{equation}
 The main reason for the existence of this Ran version of the convolution diagram is the fact that the so-called \textit{convolution Grassmannian} admits a Ran version (see \cref{convolution-grassmannians}), allowing to define the upper vertex of this diagram. 
From this, we prove that $\hck_\Ran$ carries a nonunital $\E_1$-algebra structure in correspondences, i.e. there exists an object \begin{equation}\label{algebra-structure-in correspondences-intro}\hck^\otimes_\Ran\in\Alg_{\E_1}^\nun(\Corr(\PStrStk)^\times)\end{equation} whose underlying object is $\hck_\Ran$. The target is the $1$-category of correspondences on $\PStrStk$ in the sense of \cite{GRI}, \cite{Mann}, together with the monoidal structure induced by the Cartesian monoidal structure on $\PStrStk$. Associativity ($\E_1$) here is a consequence of the existence of $n$-fold convolution Grassmannians, see \cref{convolution-grassmannians}.

The reason we are interested in extending \eqref{convolution-intro} to \eqref{Ran-convolution-intro} is the following. Informally, push-and-pull of perverse (or, in our case, constructible) sheaves along \eqref{convolution-intro} induces the convolution product of \cite[§4]{MV}, which corresponds to an $\E_1$-algebra structure on the chosen category of sheaves over $\hck$. As we will see, the existence of the Ran version \eqref{Ran-convolution-intro} allows to add an additional ``$\E_2$-direction'', corresponding to the commutativity constraint appearing in \cite[§5]{MV} (which itself uses the existence of the Beilinson-Drinfeld Grassmannian). However, in order to carry out the latter step, we choose to pass to the complex-analytic world: this allows to use Lurie's characterization of factorization algebras \cite[Theorem 5.5.4.10]{HA}. Note that this latter ingredient is absent in \cite{MV}, where the properties of perverse sheaves allow to establish the commutativity constraint ``on the nose''.

More precisely, in §\ref{Section-3.1} we apply the stratified analytification functor $(-)^\an$ of \cref{big-stratified-analytification} to the objects constructed in the previous section, with the goal of deducing the existence of an $\E_3$-algebra structure on the category of \textit{topological} constructible sheaves over the resulting complex-analytic objects.\footnote{``Topological'' here is to be understood as opposed to ``algebraic''. It is actually an interesting question whether an $\E_3$-structure can be established on a category of constructible \'etale sheaves over $\hck$ without using the theory of topological factorization algebras. One should however keep in mind that this would only make sense for finite or $\ell$-adic coefficients, since we would be looking at \'etale sheaves. In this case, the category of algebraic constructible sheaves on $\hck$ and the category of topological constructible sheaves on $\hck^\an$ coincide (see \eqref{Grassmannian-GAGA}), so it is really a matter of techniques used, not of the result. For other coefficients, the topological model is less replaceable: in the case of complex coefficients, for instance, one looks at $\Cons(\hck^\an;\C)$, which corresponds to $\Dmod(\hck)$, see \cref{remark-coefficients-intro}. 
	
	This point of view is also underlined in \cite[end of page 2]{MV}.} The functor appearing in \cref{big-stratified-analytification} is an upgraded version of Raynaud's original analytification functor \cite{SGA1-XII}, which takes into account stratifications and the formation of pro-objects and free colimits. In particular, the analytification of the object $$\hck_{\Ran}\in \PStrStk$$ belongs to the category $\PStrat$, which arises in a totally similar way to $\PStrStk$ as a pro-completion and free cocompletion of the category of \textit{topological} stratified stacks, see \cref{stratified-topological-prestacks}.

	The algebra structure in correspondences \eqref{algebra-structure-in correspondences-intro} is transferred via this procedure to an object \begin{equation}\label{topological-algebra-structure-in-correspondences-intro}\hck^{\an, \otimes}_\Ran\in \Alg_{\E_1}^\nun(\Corr(\PStrat)).
	\end{equation}
	Additionally, we are able to build a factorization algebra structure on $\hck^\an_\Ran$, in the sense that there is a  map of operads (\cref{hck-fact})
	
	\begin{equation}\label{fact-alg-intro}\hck^\fact:\textup{Fact}(\rr^2)\to\Alg_{\E_1}^\nun(\Corr(\PStrat))^\times.\end{equation}
	 \begin{recall}\label{recall-factalg-intro}Here $\textup{Fact}(\rr^2)$ is a certain operad whose algebras correspond include nonunital $\E_2$-algebras under Lurie's criterion \cite[Theorem 5.5.4.10]{HA}; the needed conditions in order to obtain a nonunital $\E_2$-algebra are essentially three (see \cite[Theorem 5.5.4.10]{HA} for the meaning of the words in italic):
	\begin{itemize}
		\item \textit{factorizability}, corresponding to the factorization property of the Beilinson-Drinfeld Grassmannian \cref{factorization-property-BD};
		\item \textit{constructibility} up to stratified homotopy, corresponding to the fact that the analytification of $\Gr_{\Ran}$ is homotopy invariant under dilation of coordinates of $\Ac$ (\cite{WM});
		\item codescent with respect to the euclidean topology of $\rr^2$ (i.e. wrt the complex-analytic topology on $\Ac$).  This condition is ``almost'' satisfied: the defect is due to the presence of pro-objects in the story, an issue which is completely solved after taking constructible sheaves: see \cref{not-cosheaf}.
	\end{itemize}
\end{recall}
	
	Let now $\cE$ be a presentable stable stable $\infty$-category, which will be our category of coefficients. We want to give a meaning to the expression $$\Cons(\hck_\Ran^\an;\cE).$$ The idea is to define this by colimits and limits from the finite-dimensional terms involved in the construction of $\hck_\Ran$. To this end, we ideally would like to build a symmetric monoidal functor $$\Cons(-;\cE):\Corr(\PStrat)\to \widehat\Cat,$$ the target being the category of large categories. Such a functor would transfer all the desired properties in one go. However, the functorialities needed to build such a functor are somehow only understood for a strict subcategory of $\PStrat$, built out of what are known as \textit{conically} stratified spaces (\cref{defin-con}): these are spaces with some mild topological conditions and the crucial requirement that the stratification satisfies a certain equisingularity condition (called the conical stratification condition). Whitney stratifications are a standard example of such, and indeed we use the fact that the analytification of the Beilinson-Drinfeld Grassmannian  \eqref{BD-Grassmannian-intro} is Whitney (due to David Nadler in his PhD thesis) to prove that it is conical.
	
	This construction is performed in Appendix \ref{appendix-stratified-spaces}. More precisely, there is a functor \begin{equation}\label{Cons-in-intro}\begin{gathered}\Str\Top_\con\to \Prl_\cE\\(X,s)\mapsto \Cons(X,s;\cE)\end{gathered}\end{equation}
	 where $\Str\Top_\con$ is the category of conically stratified spaces, $\Prl_\cE$ is the category of $\cE$-linear presentable stable categories, and $\Cons(X,s;\cE)$ is the category of $\cE$-valued constructible sheaves on $(X,s)$. This functor also carries a symmetric monoidal structure, see \cref{cons-monoidal}. The existence of \eqref{Cons-in-intro} relies on the formalism of exit paths as developed in \cite[Appendix A]{HA} and later in \cite{Exodromy-PT}.

	This functor can then be extended to $\PCon$, which is the subcategory of $\PStrat$ built out of $\Str\Top_\con$ instead of $\Str\Top$ (\cref{stratified-topological-prestacks}). The resulting functor further extends to a category of correspondences via the formalism of \cite[Part III]{GRI} and \cite{Mann}. Again, the category of correspondences appearing in the result is a strict subcategory of $\Corr(\PCon)$, in that it has less morphisms: we are looking at $$\Corr(\PCon)_{\all, \subm},$$ whose morphisms are spans $$\begin{tikzcd}&\cY\arrow[ld,"h"']\arrow[rd,"v"]&\\ \cX&&\cZ\end{tikzcd}$$
	 of morphisms in $\PCon$ where the arrow $h$ belongs to a certain class of ``smooth submersions'' (\cref{defin-subm-prestacks}).
	 This restriction is necessary in order to have the necessary base change properties (also known as Beck-Chevalley conditions) for the extension to correspondences. The final output is a symmetric monoidal functor $$\Corr(\PCon)_{\all,\subm}^\times\to \Prro_\cE,$$ see \cref{take-constructibles-final}. Here $\Prro_\cE$ is the symmetric monoidal $\infty$-category of $\cE$-linear presentable $\infty$-categories with right adjoint functors (\cref{Prro}).

	 In \cref{subsection-setup-constructibles} we prove that the analytification of $\hck_\Ran$ and its variations do belong to $\PCon$, and that the analytification of the left leg in \eqref{Ran-convolution-intro} belongs to $\subm$. This implies that the functor $\hck^\fact$ from \eqref{fact-alg-intro} factors via the subcategory $\Alg_{\E_1}^\nun(\Corr(\PCon)_{\all, \subm}^\times)$, as desired.

	In \textit{Section 4} we study the categories of sheaves $\Cons(\hck_\Ran;\cE)$ and $\Cons(\hck;\cE)$. By composing $\hck^\fact$ with $\Cons(-;\cE)$, we can apply Lurie's criterion \cite[Theorem 5.5.4.10]{HA} mentioned above (whose conditions are now completely satisfied) and obtain a map of operads $\E_2^\nun\to (\Alg_{\E_1}^\nun(\Prro_\cE))$ whose underlying category is $\Cons(\hck_\Ran^\an;\cE):$ see \cref{algebra-structure-on-Sph-Ran}.
	In other words, there are two nonunital monoidal structures on $\Cons(\hck_\Ran^\an;\cE)$, one of which is braided, which distribute one with the other.
	
	\cref{subsection-specialization} is devoted to transfer these two structures from $\Cons(\hck_\Ran^\an;\cE)$ to $\Cons(\hck^\an;\cE),$ which is done by specializing to any chosen point $x_0\in \A^1_\C$ (cf. \eqref{pullback-of-hck-to-point-intro}). This procedure amounts to some base change verifications, enabled by the fact that the right leg of the convolution diagram \eqref{Ran-convolution-intro} is ind-proper and the left leg is pro-smooth.

	We prove that, after this specialization, both the $\E_1^\nun$- and $\E_2^\nun$-structures gain units, and therefore the Dunn-Lurie Additivity Theorem \cite[Theorem 5.1.2.2]{HA} can be applied, thus combining the two algebra structures into an $\E_3$-structure on $\Cons(\hck^\an;\cE)$.

	The latter category is precisely $\Sph(G;\cE)$, i.e. the $\infty$-category of $G_\cO^\an$-equivariant constructible sheaves on $\Gr^\an$ with coefficients in $\cE$. We thus obtain \cref{final-theorem-in-introduction}. When $\cE=\Mod_R$ for a discrete ring $R$, this structure is left t-exact (t-exact if $R$ is a field) and therefore restricts canonically to a symmetric monoidal structure on perverse sheaves, which is the classical one used in \cite{MV}: see \cref{t-exactness-of-convolution}.

	\begin{rem}\label{complex-topology}
		
	Note that $\Gr^\an$ is homotopy equivalent to the loop space $\Omega G^\an\simeq\Omega^2\textup{B} (G^\an)$, which carries a standard $\E_2$-algebra structure in spaces. This is not sufficient to derive the $\E_2$-algebra structure on the spherical Hecke category though (at least not with our techniques), because the functor taking constructible sheaves is stratified homotopy invariant but not homotopy invariant (for example, constructible sheaves on $\mathbb R$ and on the point are not the same). 
	For recent developments on the loop space perspective, see \cite{Chen-Nadler-Quasi-Maps, Chen-Nadler-Real-Groups}, which also take stratifications into account.
	
		The application of Lurie's characterization of $\E_k$-algebras to the affine Grassmannian also appears in \cite{HY}, though in that paper the authors are interested in a purely topological problem and do not take constructible sheaves. To our knowledge, the formalism of constructible sheaves via exit paths and exodromy has never been applied to the study of the affine Grassmannian and the spherical Hecke category. 
	
	\end{rem}

	\subsection*{Acknowledgments}This paper constitutes a chapter of my thesis as a graduate student at Scuola Normale Superiore di Pisa and Universit\'e de Strasbourg (2018-2022) under the supervision of Mauro Porta and Gabriele Vezzosi.
	
	During the last phase of revision I was supported by the ERC Starting Grant \textit{Foundations of motivic real K-theory} (2020-2025) held by Yonatan Harpaz.
	
	I am indebted to Pramod Achar, Pierre Baumann, Katsuyuki Bando, Dario Beraldo, Justin Campbell, Robert Cass, Dustin Clausen, Ivan Di Liberti, Andrea Gagna, Dennis Gaitsgory, Jeremy Hahn, Yonatan Harpaz, Andreas Hayash, Hiroki Kato, Marius Kj\ae rsgaard, Vasily Krylov, Dmitry Kubrak, Jacob Lurie, Andrea Maffei, Lucas Mann, David Nadler, C\'edric P\'epin, Michele Pernice, Sam Raskin, Simon Riche, James Tao, Angelo Vistoli, Qixiang Wang and Allen Yuan for their suggestions and explanations. 
	
	I especially thank Peter Haine, Mark Macerato, Emanuele Pavia, Morena Porzio and Marco Volpe for the extensive and fruitful discussions carried out with them during various phases of the work, and Roman Bezrukavnikov for hosting me at MIT during April-May 2022.

	
	\section{Convolution over the Ran space}\label{Section-2}

	Throughout this whole work, $G$ will be a complex reductive group and $X$ a complex smooth curve.

	\begin{notation}\label{notation-no-R}
		When defining a presheaf over the category of complex affine schemes, we will sometimes drop the dependance on $\Spec R$ when it does not cause confusion. A point $x\in X(R)$ will just be denoted by $x\in X$, and its graph in $X\times \Spec R$ by $\Gamma_x$.
		
		Two $R$-points of $X$ will be declared ``equal'' if they coincide as maps $\Spec R\to X$, ``distinct'' if they do not coincide (but their graphs may intersect nontrivially inside $X_R$), and ``disjoint'' if their graphs do not intersect. 
		
	 Let $I\in \fs$ and $x_I\in X^I(R)$. Let $\textup{pr}_i:X^I\rightarrow X$ be the projection on the $i$-th coordinate and denote by $x_i$ the composite $\textup{pr}_i\circ x_I$. We denote by $\Gamma_{x_i}\subset X_R$ the closed subscheme given by the graph of $x_i$.
	We denote by $\Gamma_{x_I}$ the closed (possibly nonreduced) subscheme of $X_R$ corresponding to the composition 
	$$\Spec R\to X^I\to \textup{Sym}^{|I|}_X\simeq \textup{Hilb}^{|I|}_X$$ where the last isomorphisms comes from the fact that $X$ is a curve. This subscheme is supported at the union of the graphs $\Gamma_{x_i}$. For instance, if $R=\C$, $I=\{1,2\}$ and $x_1=x_2$ is a closed point of $X$, then $\Gamma_{x_I}$ is the only closed subscheme supported at the point and of length $2$.

		The definition of the affine and punctured formal neighbours of a closed subscheme $\Gamma$ of a scheme $S$, denoted by $\widetilde S_\Gamma$ and $\mathring{S}_\Gamma$ respectively, is recalled in \cref{affine-formal-completions}. When there is no risk of confusion about the ambient scheme, we will also denote them by $\widetilde \Gamma$ and $\mathring \Gamma$ respectively.

		A $G$-torsor $\cF\in \Bun(X_R)$ will just be denoted by $\cF\in \Bun(X)$, and the trivial $G$-torsor over a scheme $S$ will be denoted by $\cT_S$.
		
		Finally, the symbol $\PSh(-)$ denotes groupoid-valued presheaves, whereas $\cP(-)$ denotes space-valued presheaves.
	\end{notation}

	\subsection{The Beilinson--Drinfeld setting}

	The following definitions also appear, in various forms, in \cite{Reich}, \cite{Richarz} and \cite{Cass-BDGrassmannian}, and are natural generalizations of the characterizations recalled in \cref{moduli-interpretation}.
	
	\begin{defin}\label{first-definitions}Let $I,I_1,I_2$ be nonempty finite sets. We recall the following definitions. \begin{itemize}\item The Beilinson--Drinfeld arc group $$G_{\cO,I}=\{x_I\in X^I,g \in G(\widetilde X_{x_I})\}.$$ Note that $G(\widetilde X_{x_I})\simeq \Aut(\cT_{\widetilde X_{x_I}})$.
			
			\item The Beilinson--Drinfeld loop group
			\begin{gather*}G_{\cK, I}=\{x_I\in X^I,\cF\in \Bun(X), \alpha\textup{ trivialization of $\cF$ on }X\setminus x_{I},\mu\textup{ trivialization of $\cF$ on }\widetilde X_{x_I}\}\end{gather*}
			with its ``decoupled version''
			\begin{gather*}G_{\cK, I_1,I_2}=\{x_{I_1}\in X^{I_1},x_{I_2}\in X^{I_2}, \cF\in \Bun(X), \\ \alpha\textup{ trivialization of $\cF$ on }X\setminus x_{I_1},\mu\textup{ trivialization of $\cF$ on }\widetilde X_{x_{I_2}}\},\end{gather*}\end{itemize}\end{defin}

\begin{defin}\label{defin-GrBD}	The Beilinson--Drinfeld Grassmannian is defined as

			$$\Gr_{I}=\{x_I\in X^I, \cF\in \Bun(X), \alpha\textup{ trivialization of $\cF$ on }X\setminus x_I\}.$$	\end{defin}

	\begin{rem}\label{action-from-the-right}The objects $G_{\cK, I}, G_{\cK,I_1,I_2},\Gr_{I}$ are ind-schemes over $X^I$ by \cite[Theorem 3.1.3, Proposition 3.1.9 and variations thereof]{Zhu}. The object $G_{\cO,I}$ is representable \cite[Proposition 3.1.6]{Zhu}, and has the structure of an infinite-dimensional group scheme relative to $X^I$. The object $\Gr_I$ is often denoted by $\Gr_{G,X^I}, \Gr_{G,I},\Gr_{X^I}$. The notations $G_{\cO,X}, G_{\cK,X}, \Gr_{X}$, respectively for $G_{\cO,\{1\}}, G_{\cK,\{1\}}, \Gr_{\{1\}}$, are also common and we will often use them.

	The group scheme $G_{\cO,I}$ acts on $G_{\cK,I}$ relatively to $X^I$ by modification of $\mu$, and there is an equivalence 
	$$\Gr_I\simeq G_{\cK,I}/G_{\cO,I}$$ where the right-hand-side is the fpqc quotient relative to $X^I$.
	Analogously, there is an action of $G_{\cO,I_2}$ on $G_{\cK,I_1,I_2}$ relative to $X^{I_2}$, and the quotient is $\Gr_{I_1}\times X^{I_2}$.

		
		
	\end{rem}

	\begin{notation}\label{notation-fibers-at-point}
	We denote Let $x:\Spec \C\to X$ be a closed point. We adopt the notation $$\Gr_x=\Gr_X\times_{X,x}\Spec \C$$ and similarly $G_{\cK,x},G_{\cO,x}$.
	\end{notation}
	
	\begin{prop}[Translational invariance]\label{translational-invariance}Let $X=\A^1_\C$. Then any choice of a closed point $x\in \A^1_\C$ induces splittings $$\Gr_{\Ac}\simeq \Gr_x\times \A^1_\C$$
		$$G_{\cO,\Ac}\simeq G_{\cO,x}\times \A^1_\C.$$
	\end{prop}
	\begin{proof}
		The case of $\Gr$ is proven as follows. The definition of $\Gr_{X}$ is functorial in $X$, and hence the translation action on $\A^1_\C$ lifts to $\Gr_{\A^1}$ as a map $$\Gr_{\Ac}\times_\C\Ac\to \Gr_{\Ac}.$$ The choice of any point $x$ induces a map $$\Gr_x\times_\C \Ac\to \Gr_{\Ac}\times_\C\Ac\to\Gr_{\Ac}$$ which provides the splitting.

		The case of $G_\cO$ is straightforward from the definition.
	\end{proof}
	
	\begin{warning}
		Such splittings only hold for $I=\{1\}$.
	\end{warning}
	
	\begin{defin}We denote by $\fs$ the category of nonempty finite sets and surjections between them.\end{defin}
	
	\begin{rem}\label{action-L}Fix $I\in \fs$. 
	The group $G_{\cO,I}$ acts on $\Gr_{I}$ over $X^I$ as follows: $$(x_I,g).(x_I,\cF,\alpha):=(x_I,\cF,g|_{\mathring{X}_{x_I}}\circ\alpha|_{\mathring X_{x_I}}).$$ 
	This definition is well-posed thanks to the Beuville-Laszlo theorem (cf. \cref{rem:Grglob}), which implies that the datum of a trivialization on the punctured affine formal neighbourhood of a point is equivalent to one on the complement of the point. 
	We will implicitly use this argument in the rest of the paper while writing similar expressions.
		
	Suppose $I=I_1\sqcup I_2$. The relative group scheme $G_{\cO,I}$ acts on $G_{\cK,I_1,I_2}$ relatively over $X^I$ again by modification of $\alpha$. Note that $\alpha$ is a trivialization away from $x_{I_1}$, and we are modifying it at all points of $x_I$ (not just those of $x_{I_1}$).
	
	Let now $I_1,I_2,I_3$ be nonempty finite sets. The relative group scheme $G_{\cO,I_2}$ acts on $G_{\cK,I_1,I_2}\times_{X^{I_2}}G_{\cK, I_2,I_3}$ relative over $X^{I_2}$ by simultaneous modification of $\mu$ in the first component and $\alpha$ in the second one, respectively over $\widetilde{X}_{\Gamma_{x_{I_2}}}$ and $\mathring{X}_{\Gamma_{x_{I_2}}}$. The same relative group scheme acts on $G_{\cK,I_1,I_2}\times_{X^{I_2}}\Gr_{I_2}$ relatively over $X^{I_2}$ in a similar way.
	\end{rem}

	\begin{construction}\label{notation-partitions}Let $I,J\in\fs$ and $[\phi:I\twoheadrightarrow J]$ a $J$-partition of $I$, i.e. the equivalence class of a surjection $\phi:I\twoheadrightarrow J$ modulo autobijections of $J$. Following \cite[§4.2]{Nadler-Perverse-real} and \cite[(4.2)]{Cass-BDGrassmannian}, let $$X^\phi=\{x_I=(x_1,\dots,x_{|I|})\in X^I\mid \phi(i)=\phi(i')\Rightarrow x_i=x_{i'}, $$$$\phi(i)\neq \phi(i')\Rightarrow \Gamma_{x_i}\cap \Gamma_{x_{i'}}=\varnothing, i,i' \in I\}\subset X^I.$$
		
		This partition of $X^I$ forms a stratification which is called the \textit{incidence stratification} of $X^I$ (\cite[§4.2]{Nadler-Perverse-real}).
	\end{construction}

	\begin{prop}[Factorization property]\label{factorization-property-BD} With the above notation, there is an isomorphism 
	$$\Gr_{I}\times_{X^I}{X^\phi}\simeq \left(\prod_J\Gr_X\right)\times_{X^I} X^\phi$$ 
	where the map $\prod_J \Gr_X\to X^I$ is induced by the diagonal $X^J\to X^I$ associated to $\phi$. Note that the right-hand-side is also isomorphic to $\left(\prod_J\Gr_X\right)\times_{X^J} X^\id$, $\id$ being the partition induced by the identity of $J$ and $X^\id\subset X^J$ being the associated stratum, i.e. the open stratum of pairwise distinct coordinates in $X^J$.
	\end{prop}
	\begin{proof}
		See \cite[Proposition 4.2.1]{Nadler-Perverse-real} or \cite[Proposition 4.6]{Cass-BDGrassmannian} (which refers directly to \cite[Proposition 3.1.13]{Zhu}). To be precise, the proof in \cite{Zhu} is performed for $X=\A^1_\C$ (see \cref{factorization-for-A1}), but it is literally the same in the general case.
	\end{proof}
	
	\begin{prop}\label{factorization-GO}
		With the above notations, there is an isomorphism $$G_{\cO,I}\times_{X^I}X^\phi\simeq \left(\prod_JG_{\cO,X}\right)\times_{X^I}X^\phi$$ and the right-hand side is in turn isomorphic to $\prod_JG_{\cO,X}\times_{X^J}{X^\id}$ as above.
	\end{prop}
	\begin{proof}
		Straightforward from the definition.
	\end{proof}
	
	\begin{rem}
		Under the identifications of \cref{factorization-property-BD}, we note the following. Let $x\in X(\C)$. Then we can perform pullbacks along $\Spec \C\xhookrightarrow{(x,\dots,x)} X\hookrightarrow X^I$ and obtain isomorphisms
		\begin{gather*}\Spec \C\times_{X^I}\Gr_I\simeq\Gr\\
		\Spec \C\times_{X^I}G_{\cO,I}\simeq G_\cO\\
		\Spec \C\times_{X^I}G_{\cK,I}\simeq G_\cK\end{gather*}
		and the actions appearing in \cref{action-L} become the ones from \cref{perverse-and-stratification} and \cref{definition-of-twisted-product}.
	\end{rem}
	
	\begin{rem}\label{stratification-of-GrX-HckX}The stratification in Schubert cells of $\Gr$ (\cref{perverse-and-stratification}) naturally induces a stratification on $\Gr_X$ with the same stratifying poset $\xt$, as showed in \cite[(3.1.11)]{Zhu}, \cite[Recall A.14]{WM}. If $(x,\cF,\alpha)\in \Gr_X(\C)$, we denote by $$\Inv_x(\cF, \alpha)\in \xt$$ the associated coweight (we will often abbreviate this by $\Inv_x(\alpha)$). We also denote the stratum with coweight $\mu$ by $$\Gr_{X,\mu},$$ and set $$\Gr_{X,\leq \mu}=\bigcup_{\nu\leq \mu}\Gr_{X,\nu}.$$\end{rem}
	
	\begin{rem}The notion of $\Inv_x(\alpha)$ admits the following generalization. Let $x\in  X(\C),\cF_0,\cF_1\in \Bun(X)(\C)$, and let $\eta:\cF_1|_{X\setminus x}\simeq \cF_0|_{X\setminus x}$. Fix a trivialization $\lambda$ of $\cF_0$ on the formal neighbourhood of $x$. Such a trivialization always exists because all $G$-torsors are trivial on $\Spec \C\taylor$. Then one can compute $$\Inv_x(\lambda|_{\mathring X_x}\circ \eta|_{\mathring X_x})\in \xt$$ and check, by uniqueness of the Cartan decomposition, that this coweight is independent of the choice of $\lambda$. We denote it by $$\Inv_x(\eta).$$\end{rem}

	\begin{recall}There is a well-defined stratification on $\Gr_{I}$ whose explicit description is provided in \cite[§4.2]{Nadler-Perverse-real} or also \cite[Definition 4.18]{Cass-BDGrassmannian}. We recall it here, just to fix notations for the generalization to the convolution Grassmannian. The indexing poset of the stratification is $$\{[\phi:I\twoheadrightarrow J]\textup{ partition of }I, \mu_{J}=(\mu_1,\dots,\mu_{|J|})\in (\xt)^J\}$$ where the order relation is given by: $$(\phi, \mu_{J})\leq (\phi',\mu'_{J'})$$ if and only if there exists a refinement $[\psi:J'\twoheadrightarrow J]$ such that for every $h\in J$ $$\mu_h\leq \sum_{h'\in J',\psi(h')=h}\mu_{h'}'.$$
		The stratification is then defined by setting $$\Gr_{I,\phi,\mu_{J}}=\prod_{h\in J} \Gr_{X,\mu_h}\times_{X^J} X^\phi\hookrightarrow\Gr_{I}$$ where the embedding is induced by \cref{factorization-property-BD}.
		
		For $[\phi:I\twoheadrightarrow J]$ partition of $I$, $\mu_J\in(\xt)^J$, we consider the Zariski closure $$\overline{\Gr_{I,\phi,\mu_J}}.$$ By \cite[Lemma 4.20]{Cass-BDGrassmannian}, this is a union of strata.
		
		Let $(\phi,\mu_J)\leq (\phi',\mu_{J'}')$ as above. Then we have a natural closed embedding $$\overline{\Gr_{I,\phi,\mu_J}}\hookrightarrow \overline{\Gr_{I,\phi',\mu_{J'}'}}.$$	\end{recall}
		
		\begin{recall}\label{GrN}We recall the definition of the standard filtration of the Beilinson-Drinfeld Grassmannian (\cite[Theorem 3.1.2]{Zhu}, \cite[Lemma 3.4]{Richarz}).  Let $I\in \fs, n, N\in \nn$. Define $$\Gr^{(N)}_{\GL_n,I} =
		\{(x_I, \cF, \alpha) \in \Gr_{\GL_n,I} \mid \cO^n_X(-N\Gamma_{x_I}) \subset \cF \subset \cO^n_{X}(N\Gamma_{x_I})\}$$ (here we are implicitly identifying $\GL_n$-torsors with locally free sheaves of rank $n$). Then this is a projective scheme relative to $X^I$, and $\Gr_{\GL_n,I}$ is filtered by the $\Gr^{(N)}_{\GL_n,I}$'s. For the case of a general $G$, one chooses a faithful representation $\rho:G\to\GL_n$ and defines $\Gr^{(N)}_{G,I}$ via the closed embedding between Beilinson-Drinfeld Grassmannians induced by $\rho$ (cf. \cite[Propositions 1.2.5, 1.2.6]{Zhu}).
			\end{recall}
		\begin{rem}Note that $\Gr^{(N)}_{G,I}$ is a union of strata of $\Gr_{G,I}$ by a principle similar to \cref{filtration-and-stratification-of-Gr}.
		
		For instance, let $G=\GL_n, I\in \fs$ and $N\in \nn$. Choose $T=(\Gm)^n$ given by the diagonal matrices in $\Gm$. This induces an embedding of posets $\xt\hookrightarrow \nn^n$ whose image is spanned by those $n$-uples $(\mu_1,\dots,\mu_n)$ where $\mu_1\geq \dots\geq\mu_n$ and $\nu_I\in(\xt)^I$ given by $\nu$ in all components. Via this identification, it makes sense to define $\nu=(N,\dots,N)\in \xt$. Then $$\big(\Gr_{I}^{(N)}\big)_\red=\overline{\Gr_{I,\id:I\to I, \nu_I}},$$ cf. also \cite[Recall A.15]{WM}.

	\end{rem}

\begin{rem}\label{fibers-of-GrI}
	Let $I\in \fs$. Then the pullback $X\times_{X^I}\Gr_{I}^{(N)}$ along the principal diagonal $X\to X^I$ is isomorphic to $\Gr_{X}^{N\cdot |I|}$.
\end{rem}
	
	\begin{recall}Recall now from \cref{perverse-and-stratification} that the action of $G_\cO$ on $\Gr$ restricts to $\Gr_{\leq \mu}$ for each $\mu\in \xt$ and that this restriction factors through a quotient $$G_\cO\twoheadrightarrow G_{\cO}^{(j_\mu)}$$ for a sufficiently large natural number $j_\mu$, where $$G_{\cO}^{(j_\mu)}:=G(\C\taylor/t^{j_\mu})\simeq G(\C[t]/t^{j_\mu})$$ is now a group scheme of finite type.
	In a totally similar way, the action of $G_\cO$ on $\Gr$ restricts to $\Gr^{(N)}$ for each $N\in\nn$ and that this restriction factors through the quotient $$G_\cO\twoheadrightarrow G_{\cO}^{(j_N)}$$ for a sufficiently large $j_N$. In what follows, we will privilege the filtration by $\Gr^{(N)}$'s in that it extends in a slightly simpler way to the Beilinson-Drinfeld Grassmannian.
		\end{recall}
	For $j$ a natural number, and $Z\subset X$ a closed subscheme, let $Z_{(j)}$ denote the $j$-thickening of $Z$, i.e. the scheme $(Z, \cO_X/\cI_Z^j)$.

	\begin{defin}\label{pro-group}Let $j$ be a natural number, and $I\in \fs$. We define $$G_{\cO, I}^{(j)}$$ as the group scheme, relative to $X^I$, classifying $$\{x_I\in X^I, g\in G((\Gamma_{x_I})_{(j)})\}.$$
	\end{defin}
	
	
	\begin{rem}\label{remark-GO2j}Let $I\in \fs$, $X\to X^I$ be the diagonal morphism, and $j\in \nn$. Then we have isomorphisms
	\begin{gather*}G_{\cO,I}^{(j)}\times_{X^I}X\simeq G_{\cO,X}^{(|I|\cdot j)}\\
		G_{\cO,I}^{(j)}\times_{X^I}X^\id\simeq \big(G_{\cO,X}^{(j)}\big)^I\times_{X^I}X^\id.
		\end{gather*}
		
		\begin{proof}
			The first part follows from the fact that, if $x_I=(x,\dots,x)\in X^I$ for some $x\in X$, the subscheme $\Gamma_{x_I}$ is defined as the sum of $|I|$ copies of the divisor $\{x\}\subset X$. Hence, its ideal of definition is $\cI_x^{|I|}$, and thus $\Gamma_{x_I}^{(j)}$ is the closed subscheme supported at $x$ and with structure sheaf $$(x,\cO_X/\cI_x^{|I|\cdot j}).$$
			The second part is straightforward from \cref{factorization-GO}.
		\end{proof}
	\end{rem}
	\begin{rem}\label{smoothness-arc-group} It is easy to see that $$G_{\cO,I}\simeq\lim_j G_{\cO,I}^{(j)}.$$
		The functor $G_{\cO,I}^{(j)}$ is a smooth group scheme of finite type over $X^I$ (by \cite[Lemma 2.5.1]{Raskin-Principal-II}). Smoothness may seem a bit counter-intuitive, since the fiber of $G_{\cO,I}$ (say $I=\{1,2\}$) over a point in the diagonal $X\subset X^2$ is given by a copy of $G_\cO$, while for instance the fiber over a point in the disjoint locus of $X^2$ is given by $G_\cO\times G_\cO$. However, one cannot argue that this contradicts flatness, because we are dealing with infinite-dimensional objects. And in fact, when one truncates to $G_{\cO,I}^{(j)}$, the following happens. The fiber of $G_{\cO,I}^{(j)}$ at a point $(x,x)$ on the diagonal is $$G(\C[t]/(t^{2j}))$$ by \cref{remark-GO2j}. On the other hand, the fiber at a point $(x_1,x_2)$ outside the diagonal is $$G_{\cO}^{(j)}\times G_{\cO}^{(j)}.$$
		Therefore, the dimensions of these fibers are $$\dim G(\C[t]/(t^{2j}))=(\dim G)^{2j}$$ and $$\dim (G(\C[t]/t^{j}))^2=(\dim G)^{2j}.$$
	\end{rem}

	\begin{rem}\label{action-factors-BD}Let $I\in \fs,N\in \nn$. The action of $G_{\cO, I}$ on $\Gr_I$ over $X^I$ described in \cref{action-L} restricts to each $\Gr^{(N)}_{I}$. By the same proof of \cite[Corollary 3.7]{Richarz}, the restriction factors through a quotient $G_{\cO,I}^{(j_{N})}$ for a sufficiently large natural number $j_{N}$. Note also, again by the same proof, that $j_N$ is independent of $I$.\end{rem}
	
	The following definitions are inspired by \cite[2.4.3]{Achar-Riche}:
	
	\begin{defin}
		Let $I\in\fs,N\in\nn, j \geq j_{N}$.
		We define $$\hck_{I}^{(N,j)}=G_{\cO, I}^{(j)}\backslash \Gr_{I}^{(N)}.$$ as the fpqc quotient stack in the category $\Stk_{/X^I}$. 
		
	\end{defin}

	Since the action of $G_{\cO,I}^{(j)}$ respects the stratification of $\Gr_{I}^{(N)}$, each $\hck_{I}^{(N,j)}$ is a stratified \'etale stack over $X^I$, locally of finite type, in the sense of \eqref{stratified-stacks}. Also, its structure map to $X^I$ is stratified when we endow $X^I$ with the incidence stratification.
	Therefore, for $j\geq j_N$, we obtain a well-defined object $$\hck_{I}^{(N,j)}\in {\StrStklft}_{/X^I}.$$
	\begin{note}
		From now on, we will fix a function $\mathbb N\to \mathbb N, N\mapsto j_N$, witnessing a choice of index such that the action of $G_{\cO,I}$ on $\Gr_I^{(N)}$ factors through $G_{\cO,I}^{(j_N)}$. As remarked above, we can fix a uniform choice which works for every $I$. The specific choice is totally irrelevant and all results are independent of it.
	\end{note}
	
\begin{prop}\label{transitions-are-uni}For $N\in\nn, j\geq j'\geq j_N$, the transition maps $\hck_{I}^{(N,j)}\to \hck_{I}^{(N,j')}$ belong to the class $\uni$ defined in \cref{uni}.\end{prop}
\begin{proof}
First of all, the maps are smooth by \cite[\texttt{Tag 02K5}]{Stacks-Project} (see also \cite[2.3]{Wedhorn}, \cite[Lemma 2.5]{Raskin-Principal-II}). More precisely, they are smooth quotients relative to $X^I$, in particular they are representable.

Note then that that the kernel $K_{j,j'}$ of the map of group schemes $G_{\cO}^{(j)}\to G_{\cO}^{(j')}$ is unipotent.
Moreover, the Beilinson-Drinfeld version of $K_{j,j'}$, i.e. the relative kernel of $G_{\cO,I}^{(j)}\to G_{\cO,I}^{(j')}$, splits as $K_{j,j'}^{|J|}\times_{X^I}X^\phi$ over each stratum $X^\phi$ of $X^I$ indexed by a partition $[\phi:I\to J]$.
Hence, by factorization (\cref{factorization-GO}), over each stratum $X^\phi$ the map $G_{\cO,I}^{(j)}\to G_{\cO,I}^{(j')}$ is a quotient map with fiber $K_{j,j'}^{|J|}$, which achieves the proof. Note that the pullbacks of $G_{\cO,I}^{(j)}$ and $G_{\cO,I}^{(j')}$ to strata of $X^I$ are themselves strata inside $G_{\cO,I}^{(j)}$ and $G_{\cO,I}^{(j')}$ respectively.
\end{proof}

	\begin{defin}
		We define, for $j\geq j_1$, $$\hck_I^{(j)}=\hck_I^{(1,j)}$$ and $$\hck_I=``\lim_{j\geq j_1}"\ \hck_I^{(1,j)}\in \Pro(\StrStklft).$$\end{defin}

	

Hence, the object $\hck_I$ belongs to the full subcategory $$\Pro_\uni(\StrStklft)\subset \Pro(\StrStklft)$$ from \cref{uni}.

\subsection{The Hecke stack over the Ran space}
 \begin{defin}
 	The \textit{Ran presheaf} of $X$ is the colimit $$\Ran(X)=\colim_{I\in \fs^\op}X^I$$ in the category $\Fun(\Aff_\C^\op,\sets)$, where the diagram is the one that associates to a map $I\to J$ the induced diagonal map $X^J\to X^I$.
 \end{defin}

The formation of this colimit loses any kind of descent, see e.g. \cite[Warning 2.4.4]{Gaitsgory-Tamagawa}.

 \begin{rem}The functor of points of $\Ran(X)$ can be described as $$\Ran(X)(\Spec R)=\{S\subset X(R)\textup{ nonempty unordered finite subset}\}.$$ For a proof, see e.g. \cite[Lemma 3.3]{WM}.\end{rem}
 
 \begin{notation}
 	For $S\in \Ran(X)(R)$, we denote by $\Gamma_S$ the divisor $\sum_{x_i\in S}\Gamma_{x_i}$.
 \end{notation}
We now want to promote the association $I\mapsto \hck_I$
		 to a functor  $\fs^\op\to \Pro(\StrStklft)$.

		 	\begin{lem}\label{GrRan}A surjection $\tau:I\to J$ in $\fs$ induces a closed immersion $\Gr_J^{(1)}\hookrightarrow \Gr_I^{(1)}$, and this determines a functor $$\fs^\op\to \Str\Sch^\lft_\C$$
		 $$I\mapsto \Gr_I^{(1)}$$ whose colimit $$\Gr_\Ran=\colim_{I\in \fs^\op}\Gr_I^{(1)}\in \PSh(\Str\Sch_\C^\lft)$$ lives over the algebraic Ran space of $X$ and classifies the datum of $$(S\subset X(R), \cF\in \Bun(X_R), \alpha:\cF|_{X_R\setminus \Gamma_S}\triv\cT|_{X_R\setminus \Gamma_S}).$$ \end{lem}
		 \begin{proof}
		 	Let $\widetilde\tau$ be the diagonal $X^J\to X^I$ induced by $\tau$. We define the sought-after closed immersion as \begin{equation}\label{embedding-GrIJ}(x_J,\cF,\alpha)\mapsto (\widetilde\tau\circ x_J,\cF, \alpha).\end{equation} 
		 	One can easily see that this is stratified. To prove the functor-of-points description, let us define the following category $\cJ$:
		 	
		 	\begin{gather*}\cJ=\{I\in \fs,\underline N\in  \nn^I\}\\
		 		\Hom((I,\underline N),(J,\underline M))= \{\phi:J\twoheadrightarrow I\mid N_i\leq \sum_{j\in J\mid\phi(j)=i}M_j\ \forall i\in I\}\end{gather*}
		 		For $(I,\underline N)\in \cJ$, let $\Gr_I^{(\underline N)}$ be the closed subscheme of the ind-scheme $\Gr_I$ defined by setting 
		 		$$\Gr_I^{(\underline N)}=\{x_I\in X^I, \cF,\alpha\mid\cO^n_X(-\sum_{i\in I}N_i\Gamma_{x_i})\subset \cF\subset \cO^n_X(\sum_{i\in I}N_i\Gamma_{x_i})\}$$ for $\GL_n$ and then proceeding as in \cref{GrN}.
		 		There is a functor \begin{gather*}\cJ^\op\to \Str\Sch_\C^\lft\\
		 			I\mapsto \Gr_I^{(\underline N)}
		 			\end{gather*}
		 		sending a map in $\cJ$ to the restriction of \eqref{embedding-GrIJ} to $\Gr_I^{(\underline N)}$ (the condition that $N_i\leq \sum_{j\in J\mid\phi(j)=i}M_j$ ensures that this restrictions takes values in $\Gr_I^{(\underline M)}$). One can see that the functor of points appearing in the statement is equivalent to the colimit $$\colim_{(I,\underline N)\in\cJ^\op}\Gr_{I}^{(\underline N)}.$$ Now, the functor \begin{gather*}F:\fs\to \cJ\\
		 			I\mapsto (I,\const_1)\end{gather*} is initial. Indeed, for any $(I,\underline N)\in \cJ$ we can consider the object $J=\sqcup_{i\in I}N_i$ and the canonical surjection $J\to I$ induced by the definition of $J$. This induces a morphism in $\cJ$ between $(J,\const_1)$ and $(I,\underline N)$, hence the overcategory $F/(I,\underline N)$ is nonempty. For any other morphism $\widetilde \tau:(J',\const_1)\to (I,\underline N)$ in $\cJ$ with underlying surjection $\tau:J'\to I$, we have that for every $i\in I$ then $N_i\leq \sum_{j\in J'\mid \tau(j)=i}1$. Hence there exist surjections $\nu_i:\tau^{-1}(i)\to N_i$ for each $i\in I$, which assemble to a surjection $\nu: J'\to \sqcup_{i\in I}N_i$. By construction, $\widetilde \tau$ factors through the image of $\nu$ under $F$.
		 			
		 			 Therefore, we have an induced isomorphism at the level of colimits, which concludes the proof.
		 \end{proof}
		
			\begin{construction}\label{HckRan} Observe now that, given a surjection $I\to J$, for any $j\in\nn$ there exist $j'$ and a map of relative group schemes $$G_{\cO,J}^{(j')}\to G_{\cO, I}^{(j)}$$ over the diagonal $X^J\to X^I$. The index $j'$ need not be the same as $j$: for instance, take $I=\{1,2\}, J=\{1\}$. Then we are looking at the map $$G_{\cO,X}^{(2j)}\simeq  G_{\cO, X^2}^{(j)}\times_{X^2,\Delta}X\hookrightarrow G_{\cO,X^2}^{(j)}.$$

			We thus obtain a map of pro-relative group schemes \begin{equation}\label{JI-progroups}``\lim_{j\in \nn}{}" G_{\cO,J}^{(j)}\to ``\lim_{j\in \nn}{}" G_{\cO,I}^{(j)}\end{equation} over the diagonal $X^J\to X^I$. This, together with \cref{GrRan} and the fact that the map $\Gr_J^{(1)}\to \Gr_I^{(1)}$ from \cref{GrRan} is equivariant relatively to the map \eqref{JI-progroups}, induces a map 
		 
		 $$``\lim_{j\geq j_1}" \hck_{J}^{(j)}\to ``\lim_{j\geq j_1}" \hck_{I}^{(j)}$$
		 in $\Pro(\StrStklft)$.
		We therefore have a well-defined functor $$\fs^\op\to \Pro(\StrStklft)$$
		$$I\mapsto ``\lim_{j\geq j_{1}}"\hck_{I}^{(j)}.$$
		
	
	

We can now consider the colimit $$\hck_\Ran=\colim_{I\in \fs} ``\lim_{j\geq j_1}"\hck_I^{(j)}$$ in the category $\PSh(\Pro_\uni(\StrStklft))$, which we denote by $\PStrStk$ (\cref{stratified-prestacks}). We have a natural map $\hck_\Ran\to \Ran(X)$ in $\PStrStk$.
	
		\end{construction}


	\subsection{The convolution Beilinson-Drinfeld Grassmannian}\label{convolution-grassmannians}

	Our goal now is to transfer the convolution diagram from \cref{quotient-convolution-diagram} to the ``Beilinson--Drinfeld'' setting.
	
	\begin{defin}\label{truncated-loop-group}Given $I_1,I_2\in \fs,[\phi:I_1\twoheadrightarrow J]$ a partition, $\mu_J\in (\xt)^J,N,j\in \nn$. We define \begin{gather*}G_{\cK,I_1,I_2}^{(\infty,j)}=\{x_{I_1}\in X^{I_1}, x_{I_2}\in X^{I_2}, \cF\in \Bun(X), \alpha:\cF|_{X\setminus \Gamma_{x_{J}}}\triv\cT|_{X\setminus \Gamma_{x_{J}}}, \mu:\cF|_{(\Gamma_{x_{I_2}})_{(j)}}\triv \cT|_{(\Gamma_{x_{I_2}})_{(j)}}\}\\
			G_{\cK,I_1,I_2, \phi, \mu_J}^{(j)}=G_{\cK,I_1,I_2}^{(\infty,j)}\times_{\Gr_{I_1}}\Gr_{I_1,\phi,\mu_J}\\
			G_{\cK,I_1,I_2}^{(N,j)}=G_{\cK,I_1,I_2}^{(\infty,j)}\times_{\Gr_{I_1}}\Gr_{I_1}^{(N)},\end{gather*}
		where the map $G_{\cK,I_1,I_2}^{(j)}\to \Gr_{I_1}$ is the one that only remembers $(x_{I_1}, \cF,\alpha)$. We also define
		\begin{gather*}
			G_{\cK,I_1,I_2, \phi, \mu_J}=G_{\cK,I_1,I_2}\times_{\Gr_{I_1}}\Gr_{I_1,\phi,\mu_J}\\
			G_{\cK,I_1,I_2}^{(N)}=G_{\cK,I_1,I_2}\times_{\Gr_{I_1}}\Gr_{I_1}^{(N)}.
		\end{gather*}
	\end{defin}
	
	\begin{defin}\label{defin-Conv-N}
	Let $k\geq 1,I_1,\dots,I_k\in \fs,N\geq 0$. We define the scheme $$\Conv_{I_1,\dots,I_k}^{(N)}=G_{\cK,I_1, I_2}^{(N)}\times^{G_{\cO,I_2}}G_{\cK,I_2,I_3}^{(N)}\times^{G_{\cO,I_3}}\dots\times^{G_{\cO,I_k}}\Gr_{I_k}^{(N)}$$ as the quotient of $$G_{\cK,I_1,I_2}^{(N)}\times_{X^{I_2}}\dots\times_{X^{I_k}} \Gr_{I_k}^{(N)}$$
	with respect to the action of $\prod_{i=2,\dots,k}G_{\cO,I_i}$ described as follows. For $i=2,\dots, k-1$, each $G_{\cO,I_i}$ acts on $G_{\cK,I_{i-1},I_i}^{(N)}\times_{X^{I_i}}G_{\cK,I_i,I_{i+1}}^{(N)}$,
	relatively to $X^{I_i}$, as in \cref{action-L}. 
	For $i=k$, $G_{\cO,I_k}$ acts on $G_{\cK,I_{k-1},I_k}^{(N)}\times_{X^{I_k}}\Gr_{I_k}^{(N)}$ again like in \cref{action-L}. 

	In particular, this is a schematic quotient.
	\end{defin}
	\begin{rem}\label{truncated-convolution-Grassmannian}For any $j\geq j_1$, the expression above can be rewritten as $$G_{\cK,I_1,I_2}^{(N,j)}\times^{G_{\cO,I_2}^{(j)}}G_{\cK,I_2,I_3}^{(N,j)}\times^{G_{\cO,{I_3}}^{(j)}}\dots\times^{G_{\cO,{I_k}}^{(j)}}\Gr_{I_k}^{(N)}$$ (cf. \cite[Discussion after Lemma 5.2.3]{Zhu}).
	\end{rem}
	\begin{defin}Let $I_1,\dots,I_k\in \fs.$ We define the \textit{convolution Grassmannian} associated to $I_1,\dots,I_k$ as $$\Conv_{I_1,\dots, I_k}=G_{\cK,I_1, I_2}\times^{G_{\cO,I_2}}G_{\cK,I_2,I_3}\times^{G_{\cO,I_3}}\dots\times^{G_{\cO,I_k}}\Gr_{I_k}$$ where the notation with the superscripts has the same meaning as in \cref{defin-Conv-N}.\end{defin}
	
	\begin{rem}The convolution Grassmannian is filtered by the $\Conv_{I_1,\dots,I_k}^{(N)}$ hence it is an ind-scheme.\end{rem}
	\begin{rem}\label{moduli-interpr-Conv}
		The convolution Grassmannian classifies the datum \begin{gather*}\{(x_{I_1},\dots,x_{I_k}),x_{I_j}\in X^{I_j}\textup{ for each }j=1,\dots, k,\cF_1,\dots,\cF_k,\\\alpha:\cF_1|_{X\setminus x_{I_1}}\simeq\cT|_{X\setminus x_{I_1}}, \eta_j:\cF_j|_{X\setminus x_{I_j}}\simeq \cF_{j-1}|_{X\setminus x_{I_j}}, j=2,\dots,k\}.\end{gather*}
	\end{rem}

	\begin{notation}\label{notation-string}
		Let us fix the following notation. A general element of $\Conv_{I_1,\dots,I_k}$ will be denoted by $$\Big(\begin{tikzcd}
			\cT & {\cF_1} & {\cF_2} &\dots & {\cF_k}
			\arrow["{X\setminus x_{I_1}}","\alpha"', from=1-2, to=1-1]
			\arrow["{X\setminus x_{I_2}}","\eta_2"', from=1-3, to=1-2]
			\arrow["{X\setminus x_{I_2}}","\eta_3"', from=1-4, to=1-3]
			\arrow["{X\setminus x_{I_k}}","\eta_k"', from=1-5, to=1-4]
		\end{tikzcd}\Big).$$
		Of course this is just a symbolic notation, in that each of the arrows drawn here is defined over a (potentially) different open set.
	\end{notation}
		Let now $k\geq 1,$ $I_1,\dots,I_k, I=I_1\sqcup\dots\sqcup I_k, J\in \fs,$ and $[\phi:I\twoheadrightarrow J]$ a partition. Recall the definition of $X^\phi$ from \cref{notation-partitions}.

	\begin{prop}\label{factorization-property-Conv}
		There is an isomorphism $$\Conv_{I_1,\dots,I_k}\times_{X^I}{X^{\phi}}\simeq \prod_{j\in J}\Conv_{\Delta,m_j}\times_{X^{I}} X^{\phi}$$
		where:\begin{itemize}\item $m_j=\#\{h\mid 1\leq h\leq k, \phi^{-1}(j)\cap I_h\neq \varnothing\}$ \item  $\Conv_{\Delta,m_j}:=\Conv_{\{*\},\dots\ m_j \textup{ times } \dots,\{*\}}\times_{X^{m_j}}X$ (the map from $X$ to $X^{m_j}$ being the diagonal)
			\item the map $\prod_{j\in J}\Conv_{\Delta,m_j}\to X^{I_1\sqcup \dots\sqcup I_k}$ is induced by the diagonal map $X^J\to X^{I_1\sqcup\dots\sqcup I_k}$ associated to $\phi$.
		\end{itemize}
	\end{prop}
	
	\begin{proof}(Sketch). This proof has been suggested to us by Robert Cass. We treat the case $k=3$, $I_1=I_2=I_3={1}$. We have three essentially distinct cases:
		\begin{itemize}
			\item $J=\{1,2,3\}$, $\phi=\id$. The isomorphism (adopting \cref{notation-string}) is given by $$\Conv_{I_1,I_2,I_3}\times_{X^3}X^\phi\simeq (\Gr_X\times \Gr_X\times \Gr_X)\times_{X^3}X^\phi$$
			$$(\begin{tikzcd}
				\cT & {\cF_1} & {\cF_2} & {\cF_3}
				\arrow["{X\setminus x_1}","\alpha"', from=1-2, to=1-1]
				\arrow["{X\setminus x_2}","\eta_2"', from=1-3, to=1-2]
				\arrow["{X\setminus x_3}","\eta_3"', from=1-4, to=1-3]
			\end{tikzcd})\mapsto $$$$(\begin{tikzcd}
				\cT & {\cF_1}
				\arrow["{X\setminus x_1}","\alpha"', from=1-2, to=1-1]
			\end{tikzcd},\begin{tikzcd}[column sep="2cm"]
				\cT & {\cF_2}
				\arrow["{X\setminus x_2}","\alpha|_{\mathring X_{x_2}}\circ \eta_2|_{\mathring X_{x_2}}"', from=1-2, to=1-1]
			\end{tikzcd}, \begin{tikzcd}[column sep="2cm"]
				\cT & {\cF_3}
				\arrow["{X\setminus x_3}","\alpha|_{\mathring X_{x_3}}\circ \eta_2|_{\mathring X_{x_3}}\circ \eta_3|_{\mathring X_{x_3}}"', from=1-2, to=1-1])
			\end{tikzcd}$$
			whose inverse is given by gluing sheaves (which can be done since the points are distinct).
			
			\item $J=\{1,2\}, \phi(1)=\phi(3)=1, \phi(2)=2$ (we treat this case and not the case $\phi(1)=\phi(2)=1, \phi(3)=3$ since we want to show that our argument works even when the two equal coordinates are not adjacent to one another).
			The isomorphism is given by $$\Conv_{I_1,I_2,I_3}\times_{X^3}X^\phi\simeq (\Conv_{\Delta,2}\times\Gr_X)\times_{X^3}X^\phi$$
			$$(\begin{tikzcd}
				\cT & {\cF_1} & {\cF_2} & {\cF_3}
				\arrow["{X\setminus x_1}","\alpha"', from=1-2, to=1-1]
				\arrow["{X\setminus x_2}","\eta_2"', from=1-3, to=1-2]
				\arrow["{X\setminus x_1}","\eta_3"', from=1-4, to=1-3]
			\end{tikzcd})\mapsto $$$$(\begin{tikzcd}[column sep="2cm"]
				\cT & {\cF_1} &\cF_3
				\arrow["{X\setminus x_1}","\alpha"', from=1-2, to=1-1]
				\arrow["{X\setminus x_1}","\eta_2|_{\mathring X|_{x_1}}\circ \eta_3|_{\mathring X_{x_1}}"', from=1-3, to=1-2]
			\end{tikzcd},\begin{tikzcd}[column sep="2cm"]
				\cT & {\cF_2}
				\arrow["{X\setminus x_2}","\alpha|_{\mathring X_{x_2}}\circ \eta_2|_{\mathring X_{x_2}}"', from=1-2, to=1-1]
			\end{tikzcd})$$
			
			with inverse $$(\begin{tikzcd}
				\cT & {\cF_1} &\cF_3
				\arrow["{X\setminus x_1}","\alpha"', from=1-2, to=1-1]
				\arrow["{X\setminus x_1}","\eta"', from=1-3, to=1-2]
			\end{tikzcd},\begin{tikzcd}
				\cT & {\cF_2}
				\arrow["{X\setminus x_2}","\beta"', from=1-2, to=1-1]
			\end{tikzcd})\mapsto  $$$$(\begin{tikzcd}[column sep="2cm"]
				\cT & {\cF_1} & {\cF_2} & {\cF_3}
				\arrow["{X\setminus x_1}","\alpha"', from=1-2, to=1-1]
				\arrow["{X\setminus x_2}","\alpha|_{\mathring X_{x_2}}^{-1}\circ\beta|_{\mathring X_{x_2}}"', from=1-3, to=1-2]
				\arrow["{X\setminus x_1}","\eta|_{\mathring X_{x_1}}\circ\alpha|_{\mathring X_{x_1}}^{-1}\circ\beta|_{\mathring X_{x_1}}"', from=1-4, to=1-3]
			\end{tikzcd})$$
			
			\item $J=\{1\}$. The isomorphism is the identity.\end{itemize}\end{proof}

	\begin{defin}\label{notation-x-k}
		Let $x\in X$ be a closed point. We define $$\Conv_{x,k}=\Conv_{\Delta,k}\times_X\{x\}.$$\end{defin}
	
	\begin{rem}\label{remark-conv-is-twisted-product}
		For any $x\in X(\C)$, the object $\Conv_{x,k}$ is isomorphic to $$\Conv_k=\overbrace{G_\cK\times^{G_\cO}\dots\times^{G_\cO}G_\cK}^{k-1}\times^{G_\cO}\Gr$$ (notation as in \cref{definition-of-twisted-product}).
	\end{rem}
	In a similar fashion as \cref{translational-invariance}, one can prove:
	\begin{prop}\label{translational-invariance-conv} Let $X=\A^1_\C, k\geq 1.$ With the notations of \cref{factorization-property-Conv}, the choice of a point $x\in X$ induces a splitting $$\Conv_{\Delta,k}\simeq \Conv_{x,k}\times \A^1_\C.$$
	\end{prop}

	\begin{cor}\label{factorization-for-A1} In the case when $X=\mathbb A^1_\C$, \cref{factorization-property-Conv} specializes to \begin{gather*}\Gr_{I}|_{(\Ac)^\phi}\simeq \big(\prod_J\Gr\big)\times  (\Ac)^\phi\\
		\Conv_{I_1,\dots,I_k}|_{(\Ac)^{\phi}}\simeq( \prod_{j\in J}\Conv_{m_j})\times (\Ac)^{\phi}\end{gather*}
		where the splitting is induced by a choice of any point $x\in \Ac$.
	\end{cor}
	\begin{proof}
		It suffices to apply \cref{translational-invariance} and \cref{translational-invariance-conv} respectively.
	\end{proof}

	\begin{construction}\label{stratification-ConvX}
		One can define a stratification of $\Conv_{\Delta,k}$ (the case $k=1$ being $\Gr_{X}$), as follows. The stratifying poset is $(\xt)^k$. First of all, $G_{\cK, \Delta}:=G_{\cK,\{*\}, \{*\}}\times_{\Gr_{\{*\},\{*\}}}\Gr_{\{*\}}$ (the map $\Gr_{\{*\}}\to \Gr_{\{*\},\{*\}}$ being induced by the diagonal of $X^2$) inherits a stratification over $\xt$ from $\Gr_{\{*\}}=\Gr_X$. 
		This induces a stratification on the product $$G_{\cK,\Delta}\times_X\dots\times G_{\cK,\Delta}\times_X\Gr_{X}$$
		and one can check that this stratification passes to the multiple quotient $$G_{\cK,\Delta}\times^{G_{\cO,X}}G_{\cK,\Delta}\times^{G_{\cO,X}}\dots\times^{G_{\cO,X}}\Gr_{X}.$$
		In other words, let $\mu_1,\dots,\mu_k\in\xt$. 
		The set $$\Conv_{\Delta,k, \mu_1,\dots,\mu_{k}}(\C)$$ is the subset of $\Conv_{\Delta, k}(\C)$ where one imposes the condition that (with the notations of \cref{moduli-interpr-Conv}) $\Inv_x(\alpha)= \mu_1, \Inv_{x}(\eta_i)=\mu_i$ for every $i\geq 2$.

		Let $x\in X$ be a closed point. Then each stratum of $\Conv_{x,k}$ can be identified with $G_{\cK,\mu_1}\times^{G_\cO}\dots\times^{G_\cO}\Gr_{\mu_{k}}$, where $G_{\cK,\mu_i}$ is the preimage of $\Gr_{\mu_i}$ along the quotient map. This definition is the direct generalization from $k=2$ to arbitrary $k$ of \cite[after Lemma 4.3]{MV}.
	\end{construction}
	These strata are not orbits for a group action, but they are smooth.
	\begin{prop}\label{strata-of-conv-are-smooth}Each stratum $$G_{\cK,\mu_1}\times^{G_\cO}\dots\times^{G_\cO}G_{\cK,\mu_{k-1}}\times^{G_\cO}\Gr_{\mu_{k}}$$ is a smooth locally closed subscheme of $\Conv_{k}$.
	\end{prop}
	
	\begin{proof} The following proof has been suggested to us by Mark Macerato. First of all, we note that for $j$ sufficiently large, then $$G_{\cK,\mu_1}\times^{G_\cO}\dots\times^{G_\cO} G_{\cK, \mu_{k-1}}\times^{G_\cO}\Gr_{\mu_{k}}\simeq G_{\cK,\mu_1}^{(j)}\times^{G_{\cO}^{(j)}}\dots\times^{G_{\cO}^{(j)}} G_{\cK,\mu_{k-1}}^{(j)}\times^{G_{\cO}^{(j)}}\Gr_{\mu_{k}}$$ (with the same argument as \cref{truncated-convolution-Grassmannian}). Now, $$G_{\cK,\mu_1}^{(j)}\times G_{\cK,\mu_2}^{(j)}\times\dots\times G_{\cK,\mu_{k-1}}^{(j)}\times\Gr_{\mu_{k}}\to \Gr_{\mu_1}\times\Gr_{\mu_2}\times\dots\times\Gr_{\mu_{k-1}}\times\Gr_{\mu_{k}}$$ is a torsor with fiber $(G_{\cO}^{(j)})^{\times k-1}$, which is a smooth group scheme. Since the base is smooth (\cite[Proposition 2.1.5 (1)]{Zhu}), the total space is smooth as well. Now, the map $G_{\cK,\mu_1}^{(j)}\times G_{\cK,\mu_2}^{(j)}\times\dots\times\Gr_{\mu_{k}}\to G_{\cK,\mu_1}^{(j)}\times^{G_{\cO}^{(j)}}G_{\cK,\mu_2}^{(j)}\times^{G_{\cO}^{(j)}}\dots\times^{G_{\cO}^{(j)}}\Gr_{\mu_{k}}$ is the fpqc schematic quotient of a smooth scheme with respect to the group $(G_{\cO}^{(j)})^{\times k-1}$. In particular, it is an fpqc covering, and therefore by \cite[\texttt{Tag 02VL}]{Stacks-Project} we conclude.
	\end{proof}
	
	\begin{construction}\label{BD-stratifications}
		
		We can now define a stratification for $\Conv_{I_1,\dots,I_k}$. Recall the notation in \cref{factorization-property-Conv}. The stratifying poset will be $$\Cw_{I_1,\dots,I_k}=\{[\phi:I_1\sqcup\dots\sqcup I_k\twoheadrightarrow J], \mu_J=(\mu_j^{h_{1}^j},\dots,\mu_j^{h^j_{m_j}})\in\prod_{j\in J} (\xt)^{m_j}\}$$ where ``$\Cw$'' stays for ``coweight'' and, for each $j\in J$, the $h_i^j$'s are those indexes for which $\phi^{-1}(j)\cap I_{h_i}\neq\varnothing$.
		The stratification for $\Conv_{I_1,\dots,I_k}$ is then defined as $$\Conv_{I_1,\dots,I_k,\phi,\mu_{J}}=\left(\prod_{j\in J}\Conv_{\Delta,m_j, \mu_j^{h_1^j},\dots,\mu_j^{h_{m_j}^j}}\right)\times_{X^I} X^{\phi}\hookrightarrow \Conv_{I_1,\dots,I_k}$$ where the embedding is induced by \cref{factorization-property-Conv}.
	\end{construction}
	
	\begin{rem}\label{action-on-ConvBD}Let $I_1,\dots,I_k\in\fs,I=I_1\sqcup \dots\sqcup I_k,N\in\nn,j\geq j_N$. There is an action of $G_{\cO,I}$ over $\Conv_{I_1,\dots, I_k}$, relative to $X^I$, which modifies the first trivialization at all points ${x_{I_1},\dots, x_{I_k}}$.
	
	This factors as an action of $G_{\cO,I}^{(j)}$ on $\Conv_{I_1,\dots,I_k}^{(N)}$ relative over $X^{I}$.
	\end{rem}

	\begin{rem}\label{action-on-conv}
		Let us inspect the behaviour on different strata of the action of $G_{\cO,I}^{(j)}$ on $\Conv_{I_1,\dots,I_k}^{(N)}$ defined above. For simplicity, we look at $k=2,I_1=I_2=\{*\}$, and we distinguish the two cases of equal points $x_1=x_2=x$ and of two distinct points $x,y$. In the first case, the action is just the action of $G_{\cO}^{(2j)}$ on the first component of  $G_{\cK}^{(2N)}\times^{G_\cO}\Gr^{(2N)}$. In the second case, with the notations of \cref{moduli-interpr-Conv}, the modification of $\alpha$ at $y$ propagates to $\eta=\eta_2$ through factorization: more precisely, up to the isomorphism $\Conv_{x,y}\simeq{\Gr_x\times\Gr_y}$ induced by \cref{factorization-property-Conv} the action splits as the canonical componentwise left action $$G_{\cO,x}^{(j)}\times G_{\cO,y}^{(j)}\circlearrowright\Gr_{x}^{(N)}\times \Gr_{y}^{(N)}.$$ This follows directly from inspecting the proof of \cref{factorization-property-Conv}.

		In particular, the action at $y$ is not trivial: this may seem in contradiction with the principle applied for instance in the proof of \cref{lemma-messy-equivariance-with-BL}, where it is said that modifying a trivialization away from its ``critical points'' ($x$ in this case) does not change the datum up to some isomorphism $\Phi$. The point here is that this isomorphism $\Phi$ need not be compatible with the rest of the datum, specifically with the isomorphism $\eta_2$ in the notations of \cref{moduli-interpr-Conv}, which is defined on $X\setminus \{y\}$.
	\end{rem}

	\begin{defin}\label{truncated-hck}Let $k\geq 1,I_1,\dots, I_k\in\fs, I=I_1\sqcup\dots\sqcup I_k,j\geq j_{1}$. We define
		$$\hck_{I_1,\dots,I_k}^{(j)} = G_{\cO, I}^{(j)}\backslash \Conv_{I_1,\dots, I_k}^{(1)}$$
		
	\end{defin}
	
	\begin{rem}
		\cref{truncated-hck} is functorial in $(I_1,\dots,I_k)\in \fs^\op$, in the sense that given surjections $I_1\to J_1,\dots I_k\to J_k$ we have a map of pro-objects $``\lim_{j\geq j_1} " \hck_{J_1,\dots,J_k}^{(j)}\to ``\lim_{j\geq j_1} " \hck_{I_1,\dots,I_k}^{(j)}$, exactly as in \cref{HckRan}.
		
		This yields a functor \begin{gather*}\label{BD-Hecke}(\fs^\op)^{\times k}\to \Pro_\uni(\StrStklft)\\
		(I_1,\dots,I_k)\mapsto``\lim_{j\geq j_1} "  \hck_{I_1,\dots,I_k}^{(j)}.\end{gather*}
			\end{rem}
		
		\begin{defin}\label{defin-hck-Ran-k}For $k\geq 1$, we define
	
		$$\hck_{\Ran,k}=\colim_{I_1,\dots, I_k\in \fs^\op} ``\lim_{j\geq j_1}" \hck_{I_1,\dots,I_k}^{(j)}$$ as the colimit in the category $\PSh(\Pro_\uni(\StrStklft))=\PStrStk.$
		
		We also define $\hck_{\Ran,0}=\Spec \C$.
\end{defin}

	Let $k\in \nn, x\in X$. We define $\hck_{\Delta,k}$ and
$\hck_{x,k}$ in a similar way to \cref{notation-x-k}.

\begin{notation}
	Let $k\geq 1$. We denote by $(\Ran(X)^{\times k})_\disj$ the subfunctor of $(\Ran(X)^{\times k})$ spanned by $k$-uples of systems of points $S_1,\dots,S_k\subset X$ such that $\Gamma_{S_i}\cap \Gamma_{S_j}=\varnothing$ for all $1\leq i\neq j\leq k$.
\end{notation}
The results regarding the convolution Grassmannian imply the following:

\begin{prop}\label{factorization-Hecke-multiple}
	In the notation of \cref{factorization-property-Conv}, we have stratified equivalences
	$$\hck_{I_1,\dots,I_k}\times_{X^I}{X^{\phi}}\simeq \prod_{j\in J}\hck_{\Delta,m_j}\times_{X^{I}} X^{\phi}.$$ If $X=\A^1_\C$ we also have
	$$\hck_{\Delta,k}\simeq \hck_{x,k}\times \A^1_\C$$
	$$\hck_{I_1,\dots,I_k}\times_{X^I}{X^{\phi}}\simeq (\prod_{j\in J}\hck_{x,m_j})\times (\A^1_\C)^{\phi}.$$
	
	Finally, $$\hck_{\Ran}\times_{\Ran(X)}(\Ran(X)^{\times k})_\disj\simeq \hck_{\Ran,k}\times_{\Ran(X)^{\times k}}(\Ran(X)^{\times k})_\disj\simeq \hck_{\Ran}^{\times k}\times_{\Ran(X)^{\times k}}(\Ran(X)^{\times k})_\disj$$
	
	where in the first fiber product the map $(\Ran(X)^{\times k})_\disj\to \Ran(X)$ is the union map.
	
\end{prop}
\begin{proof}
	The first three points are straightforward from \cref{factorization-property-Conv}. The third point follows by passing to the colimit in the other ones and in \cref{factorization-property-BD}, \cref{factorization-GO}, but not immediately, since $\fs^\op$ is not filtered. The following argument, suggested by Emanuele Pavia, circumvents this problem (we explain it with $k=2$ for simplicity). Let $I_1,I_2\in \fs, I=I_1\sqcup I_2$, and denote by $(X^{I_1}\times X^{I_2})_\disj$ the open subscheme $\{(x_{I_1},x_{I_2})\in X^{I_1}\times X^{I_2}\mid  \Gamma_{x_a}\cap \Gamma_{x_b}=\varnothing\ \forall a\in I_1,b\in I_2\}$. Note that $$\colim_{I_1,I_2\in \fs^\op}(X^{I_1}\times X^{I_2})_\disj=(\Ran(X)\times\Ran(X))_\disj.$$ Now, in the diagram 
	
		\[\begin{tikzcd}
		{\hck_{I_1,I_2}\times_{X^{I_1\sqcup I_2}}(X^{I_1}\times X^{I_2})_\disj} & {\hck_{I_1}\times\hck_{I_2}} \\
		{(X^{I_1}\times X^{I_2})_\disj} & {X^{I_1}\times X^{I_2}} \\
		{(\Ran(X)\times\Ran(X))_\disj} & {\Ran(X)\times\Ran(X)}
		\arrow[from=1-1, to=1-2]
		\arrow[from=1-1, to=2-1]
		\arrow[from=1-2, to=2-2]
		\arrow[from=2-1, to=2-2]
		\arrow[from=2-1, to=3-1]
		\arrow[from=2-2, to=3-2]
		\arrow[from=3-1, to=3-2]
	\end{tikzcd}\]
	
	the top square is Cartesian (this is proven in the same way as the first part of the present Proposition), whereas the bottom square is Cartesian by straightforward verification. Therefore, the outer square is Cartesian for every $I_1,I_2\in \fs$, and by universality of colimits we get \begin{gather*}(\hck_{\Ran}\times\hck_{\Ran})\times_{\Ran(X)\times\Ran(X)}(\Ran(X)\times\Ran(X))_\disj\simeq \\ \hck_{\Ran,2}\times_{\Ran(X)\times \Ran(X)}(\Ran(X)\times \Ran(X))_\disj.\end{gather*} The other equivalence is proven in the same way.
\end{proof}

\subsection{The BD-convolution diagram as a 2-Segal object}\label{section-BD-Segal}

\begin{rem}\label{BD-convolution-diagram}
Let $k,N\in \nn,I_1,\dots, I_k\in\fs, I=I_1\sqcup\dots\sqcup I_k, j\geq j_N$. We have a diagram
	\begin{equation}\label{j-convolution}\begin{tikzcd}
		{G_{\cK,I_1,I_2}^{(N,j)}\times_{X^{I_2}}\dots\times_{X^{I_{k-1}}} G_{\cK,I_{k-1},I_k}^{(N,j)}\times_{X^{I_k}}\Gr_{I_k}^{(N)}} & {G_{\cK,I_1,I_2}^{(N,j_{N})}\times^{G_{\cO,I_2}^{(j)}}\dots\times^{G_{\cO,I_{k-1}}^{(j_{N})}} G_{\cK,I_{k-1},I_k}^{(N,j_{N})}\times^{G_{\cO,I_k}^{(j_{N})}}\Gr_{I_k}^{(N)}} \\
		{\prod_{I=1}^k\Gr_{I_i}^{(N)}} & {\Gr_{I}^{(N)}}
		\arrow["{q_{I_1,\dots,I_k}^{(N,j)}}", from=1-1, to=1-2]
		\arrow["{p_{I_1,\dots,I_k}^{(N,j)}}"', from=1-1, to=2-1]
		\arrow["{m_{I_1,\dots,I_2}^{(N)}}", from=1-2, to=2-2]
	\end{tikzcd}\end{equation}
	
	where: 
	\begin{itemize} \item the left vertical map is the projection to the quotient of the action of $\prod_{i=2}^kG_{\cO,I_i}$ (relative to $X^{I_2\sqcup\dots\sqcup I_k}$) induced by \cref{action-from-the-right};
	\item the horizontal map is the quotient map by the actions defined in \cref{action-L} (see also the definition of \cref{defin-Conv-N});
	\item the right vertical map arises as follows. Let $I_1,I_2,I_3\in \fs,I=I_1\sqcup I_2\sqcup I_3, j\geq j_1$. Then there is a map \begin{equation}\label{multiplication-BD}G_{\cK,I_1,I_2}\times_{X^{I_2}} G_{\cK,I_2,I_3}\to G_{\cK, I_1\sqcup I_2,I_3}\end{equation} which sends the datum \begin{gather*}(x_{I_1},x_{I_2}, x_{I_3}, \cF,\cG\in\Bun(X),\alpha:\cF|_{X\setminus \Gamma_{x_{I_1}}}\triv\cT|_{X\setminus \Gamma_{x_{I_1}}}, \beta:\cF|_{X\setminus \Gamma_{x_{I_2}}}\triv\cT|_{X\setminus \Gamma_{x_{I_2}}},\\ \mu:\cF|_{\widetilde X_{\Gamma_{x_{I_2}}}}\triv\cT|_{\widetilde X_{ \Gamma_{x_{I_2}}}}, \nu:\cF|_{\widetilde X_{\Gamma_{x_{I_3}}}}\triv\cT|_{\widetilde X_{ \Gamma_{x_{I_3}}}})\end{gather*} to the datum
	
	$$(x_{I_1}+x_{I_2}, x_{I_3}, \cH,\gamma:\cH|_{X\setminus \Gamma_{x_{I_1}}+\Gamma_{x_{I_2}}}\triv\cT|_{X\setminus \Gamma_{x_{I_1}}+\Gamma_{x_{I_2}}}, \nu)$$ where $\cH$ is the $G$-bundle obtained by gluing $\cF$ and $\cG$ along $\alpha|_{\mathring X_{\Gamma_{x_{I_2}}}}\circ \mu^{-1}|_{\mathring X_{\Gamma_{x_{I_2}}}}$ (this makes use of the Beauville--Laszlo theorem, cf. \cref{prop:Grloc}), and $\gamma$ is the trivialization inherited from $\alpha$ via the gluing procedure.
	
	It is easy to see that this map passes to the quotient $$G_{\cK,I_1,I_2}\times^{G_{\cO,I_2}} G_{\cK,I_2,I_3}.$$ 
	Analogously, there is a map $G_{\cK,I_1,I_2}\times \Gr_{I_2}$ which also passes to the quotient $G_{\cK,I_1,I_2}\times \Gr_{I_2}.$ These maps induce a map $$m_{I_1,\dots,I_k}:G_{\cK,I_1,I_2}\times^{G_{\cO,I_2}}\dots\times^{G_{\cO,I_k}}\Gr_{I_k}\to \Gr_{I_1\sqcup\dots\sqcup I_k}.$$ 
	These maps are compatible with the filtrations of source and target and thus restrict to maps $m_{I_1,\dots,I_k}^{(N)}$ like in \eqref{j-convolution}.

\end{itemize}
		
	 Note that $q^{(N)}_{I_1,\dots,I_k}$ and $m^{(N)}_{I_1,\dots I_k}$ do not depend on $j$. For $x\in X$, the pullback of this whole diagram along the diagonal $\Spec \C\xhookrightarrow{x} X\xrightarrow{\Delta} X^I$ is isomorphic to a diagram
	
	\begin{equation}\begin{tikzcd}
			{G_{\cK}^{(|I_1|\cdot N, j)}\times\dots\times G_{\cK}^{(|I_{k-1}|\cdot N,j)}\times\Gr^{(N)}} & {G_{\cK}^{(|I_1|\cdot N,|I_2|j)}\times^{G_{\cO}^{(|I_2|j)}}\dots\times^{G_{\cO}^{(|I_{k-1}|j)}} G_{\cK}^{(|I_{k-1}|N,|I_k|j)}\times^{G_{\cO}^{(|I_k|j)}}\Gr^{(|I_k|N)}} \\
			{\prod_{i=1}^k\Gr^{(|I_i|N)}} & {\Gr^{(|I|N)}}
			\arrow["{q_{I_1,\dots,I_k}^{(N)}}", from=1-1, to=1-2]
			\arrow["{p_{I_1,\dots,I_k}^{(N)}}"', from=1-1, to=2-1]
			\arrow["{m_{I_1,\dots,I_2}^{(N)}}", from=1-2, to=2-2]
	\end{tikzcd}\end{equation}

 	naturally generalizing \eqref{eq:convdiag}.
	
	 We also have a diagram of groups, relative to $X^I$,
	\begin{equation}\label{convolution-GO}\begin{tikzcd}
		{G_{\cO,I}^{(j)}\times_{\prod_{i=2}^kX^{I_i}}\prod_{i=2}^kG_{\cO,I_i}^{(j)}} & {G_{\cO,I}^{(j)}} \\
		{\prod_{i=1}^kG_{\cO,I_i}^{(j)}} & {G_{\cO,I}^{(j)}}
		\arrow[from=1-1, to=1-2]
		\arrow[from=1-1, to=2-1]
		\arrow[shift right, no head, from=1-2, to=2-2]
		\arrow[shift left, no head, from=1-2, to=2-2]
	\end{tikzcd}\end{equation}

where the left map is induced by the map $G_{\cO, I}^{(j)}\to G_{\cO,I_1}^{(j)}\times_{X^{I_1}}X^I$ associated to the embedding $\Gamma_{x_{I_1}}\hookrightarrow\Gamma_{x_I}$, and the horizontal map is the projection. 

\begin{lem}\label{lemma-messy-equivariance-with-BL}The vertices of \eqref{convolution-GO} act respectively on the vertices of \eqref{j-convolution}, and the maps in \eqref{j-convolution} are equivariant with respect to the maps in \eqref{convolution-GO}.\end{lem}

\begin{proof}
The only part that requires some work is equivariance of the leftmost vertical arrow. The proof is an application of the Beauville--Laszlo theorem: in a few words, modifying a trivialization $\alpha$ defined on $X\setminus\Gamma_{x_{I_1}}$ around some points not included in $x_{I_1}$ produces an equivalent datum in $\Gr_{I}$. We provide all the details of the proof here below.

First of all, by definition of the various actions, it suffices to prove the claim for $N=\infty$ and without the index $j$. Equivariance in the last $k-1$ components is straightforward, and witnesses the phenomenon that allows to ``shift'' the right multiplication action of $G_\cO$ on $G_\cK$ to the ``antidiagonal'' action of $G_\cO$ on $G_{\cK}\times G_\cK$.

We are left to check equivariance of the quotient map $G_{\cK,I_1,I_2}\to \Gr_{I_1}\times_{X^{I_1}}X^I$ with respect to the restriction map $G_{\cO,I}\to G_{\cO,I_1}\times_{X^{I_1}}X^I$. 
Let thus $x_I\in X(R)^I,g\in G(\widetilde\Gamma_{x_I}), \cF\in \Bun(X_R), \alpha:\cF|_{X_R\setminus \Gamma_{x_I}}\triv\cT|_{X_R\setminus\Gamma_{x_I}}.$ 
We want to compare the modification of the datum $(\cF,\alpha)$ by $g$ and the modification of the same datum by $g|_{\widetilde\Gamma_{x_I}}$. 
Equivalently, let us assume that $g$ restricts to the identity element on $\widetilde\Gamma_{x_{I_1}}$, so that the second modification is trivial. 

Let thus $\cG$ be the first modification, arising as the gluing of the data $$\cF|_{\widetilde\Gamma_{x_I}},\quad \cT_{X\setminus \Gamma_{x_I}}$$ along the isomorphism $g|_{\widetilde\Gamma_{x_I}}\circ \alpha|_{\widetilde\Gamma_{x_I}}$. 
Let $\beta$ be the inherited trivialization on $X_R\setminus\Gamma_{x_{I_1}}$. 
Let also $\widecheck\can$ be the canonical isomorphism between $\cG|_{X_R\setminus \Gamma_{x_I}}$ and $\cT_{X_R\setminus \Gamma_{x_I}}$, and $\widehat\can$ be the canonical isomorphism between $\cG|_{\widetilde\Gamma_{x_I}}$ and $\cF|_{\widetilde\Gamma_{x_I}}$. 
Note that by construction \begin{equation}\label{equation-canonicals}\widecheck\can|_{\mathring\Gamma_{x_I}}\simeq g|_{\mathring \Gamma_{x_I}}\alpha|_{\mathring \Gamma_{x_I}}\circ\widehat \can|_{\mathring \Gamma_{x_I}}.\end{equation}

We want to exhibit an isomorphism $\Phi$ between $\cF$ and $\cG$, commuting with $\alpha$ and $\beta$ on $X_R\setminus\Gamma_{x_{I_1}}$. We define it via the Beauville--Laszlo theorem, by assigning isomorphisms $\phi, \psi, \xi$ between the restrictions of $\cF$ and $\cG$ to the schemes $$\widetilde \Gamma_{x_{I_1}},\quad \widetilde\Gamma_{x_I}\setminus\Gamma_{x_{I_1}},\quad X_R\setminus \Gamma_{x_I} $$ respectively, and by checking that they are pairwise compatible. 

We define
\begin{gather*}\phi:\cF|_{\widetilde \Gamma_{x_{I_1}}}\xrightarrow{\widehat\can|^{-1}_{\widetilde\Gamma_{x_{I_1}}}}\cG|_{\widetilde\Gamma_{x_{I_1}}}\\
\psi:\cF|_{\widetilde\Gamma_{x_I}\setminus \Gamma_{x_{I_1}}}\xrightarrow{\widehat \can|^{-1}_{\widetilde\Gamma_{x_I}\setminus\Gamma_{x_{I_1}}}\circ\alpha|_{\widetilde\Gamma_{x_I}\setminus\Gamma_{x_{I_1}}}^{-1}\circ g^{-1}|_{\widetilde\Gamma_{x_I}\setminus\Gamma_{x_{I_1}}}\circ \alpha|_{\widetilde\Gamma_{x_I}\setminus\Gamma_{x_{I_1}}}}\cG|_{\widetilde\Gamma_{x_I}\setminus\Gamma_{x_{I_1}}}\\
\xi:\cF|_{X_R\setminus \Gamma_{x_I}}\xrightarrow{\alpha|_{X_R\setminus \Gamma_{x_I}}}\cT|_{X_R\setminus\Gamma_{x_I}}\xrightarrow{\widecheck \can|^{-1}_{X_R\setminus \Gamma_{x_I}}}\cG|_{X_R\setminus\Gamma_{x_I}}
\end{gather*}

Note that, in the definition of $\psi$, we have used that $\alpha$ is defined outside of $\Gamma_{x_{I_1}}$ and not just outside of $\Gamma_{x_{I}}$.

By using that $g|_{\widetilde\Gamma_{x_{I_1}}}$ is the identity by assumption, and the identity \eqref{equation-canonicals}, one checks that these three isomorphisms are pairwise compatible.

The verification that the resulting $\Phi$ commutes with $\alpha$ and $\beta$ can be done along the same lines.
\end{proof}
	
%
%
%
%
%
	\begin{lem}	In the notations of \eqref{j-convolution}, the map $q_{I_1,\dots,I_k}^{(N,j)}$ induces an equivalence after passing to the quotient by the actions of \eqref{convolution-GO}.\end{lem}

	\begin{proof} It suffices to check this after pulling back to the strata of $X^{I_1}\times \dots\times X^{I_k}$, and therefore by factorization it is sufficient to prove it over the diagonal, hence over a single point $x\in X$. But in that setting, $q_{I_1,\dots,I_k}^{(N,j)}$ restricts to the map \begin{equation}\label{q-on-fibers}G_{\cK}^{(|I_1|N,|I_2|j)}\times\dots\times G_{\cK}^{(|I_{k-1}|N,|I_k|j)}\times\Gr^{(|I_k|N)} \to G_{\cK}^{(|I_1|N,|I_2|j)}\times^{G_{\cO}^{(|I_2|j)}}\dots\times^{G_{\cO}^{(|I_{k-1}|j)}}G_{\cK}^{(|I_{k-1}|N,|I_k|j)}\times^{G_{\cO}^{(|I_k|j)}}\Gr^{(|I_k|N)}
		\end{equation} 
	which exhibits the target as the quotient of the source with respect to the action of $\prod_{i=2}^kG_{\cO}^{(|I_i|j)}$ inducing the twisted product on the right-hand-side. The map \eqref{q-on-fibers} is therefore equivariant with respect to the morphism of groups $G_{\cO}^{(|I|j)}\times\prod_{i=2}^kG_\cO^{(|I_i|j)}\to G_{\cO}^{(|I|j)}$ given by projection on the first component, which concludes the proof.
\end{proof}
\end{rem}

\begin{rem}\label{BD-quotient-convolution-diagram}
	
	Let $N\geq 0, j\geq j_N, I_1,\dots,I_k\in\fs$. As a consequence of the observations made in \cref{BD-convolution-diagram}, by passing to the quotient in \eqref{j-convolution} with respect to the actions listed in \cref{BD-convolution-diagram}, we obtain a diagram of stratified quotient stacks \begin{equation}\label{truncated-convolution-body}\begin{tikzcd}& \hck_{I_1,\dots, I_k}^{(N,j)}\arrow[ld, "\overline p_{I_1,\dots,I_k}^{(N,j)}"']\arrow[rd, "\overline m_{I_1,\dots,I_k}^{(N)}"]&\\
		\hck_{I_1}^{(N,j)}\times\dots\times\hck_{I_k}^{(N,j)}&&\hck_{I_1\sqcup\dots\sqcup I_2}^{(N,j)}.\end{tikzcd}\end{equation}
\end{rem}

%
%

\begin{rem}
	Recall that we have a natural map in $\PStrStk$ from $\hck_{\Ran,k}$ to $\Ran(X)^k$ (the one that only ``remembers'' the systems of points).
\end{rem}

We will now establish a semisimplicial structure on the collection of the $\hck_{\Ran,k}$'s.

\begin{construction}\label{semisimplicial-structure}
	Let $k\geq 1, 0\leq i\leq k$, and let $d_i$ be the injective ordered map $[k-1]\to [k]$ missing the index $i$. For each $I_1,\dots,I_k\in \fs, j\geq j_1,$ we have maps 
	\begin{itemize}
		\item[$i=0$)] $\delta_{0,I_1,\dots,I_k}^{(j)}:\hck_{I_1,\dots,I_k}^{(j)}\to \hck_{I_2,\dots,I_k}^{(j)}$ induced by the projection $$G_{\cK,I_1,I_2}^{(j)}\times_{X^{I_2}} G_{\cK,I_2,I_3}^{(j)}\times_{X^{I_3}}\dots\times_{X^{I_k}}\Gr^{(1)}_{I_k}\to G_{\cK,I_2,I_3}^{(j)}\times_{X^{I_3}}\dots\times_{X^{I_k}}\Gr^{(1)}_{I_k}$$ that forgets the first component.
		
		\item[$i=k$)] $\delta_{k,I_1,\dots,I_k}^{(j)}:\hck_{I_1,\dots,I_k}^{(j)}\to \hck_{I_1,\dots,I_{k-1}}^{(j)}$ induced by the composition of the projection  $$G_{\cK,I_1,I_2}^{(j)}\times_{X^{I_2}} G_{\cK,I_2,I_3}^{(j)}\times_{X^{I_3}}\dots\times_{X^{I_k}}\Gr^{(1)}_{I_k}\to G_{\cK,I_1,I_2}^{(j)}\times_{X^{I_3}}\dots\times_{X^{I_{k-1}}}G_{\cK,I_{k-1},I_k}^{(j)}$$ with the map induced by the quotient $G_{\cK,I_{k-1},I_k}^{(j)}\to \Gr_{I_k}^{(1)}$ (recall that, as usual, $G_{\cK,I_{k-1},I_k}^{(j)}$ stays for $G_{\cK,I_{k-1},I_k}^{(1,j)}$).
		
		\item[$i\neq 0,k)$] $\delta_{i,I_1,\dots,I_k}^{(j)}:\hck_{I_1,\dots,I_k}^{(j)}\to \hck_{I_1,\dots,I_i\sqcup I_{i+1},\dots,I_k}^{(j)}$
		induced by the maps $$G_{\cK,I_i,I_{i+1}}^{(j)}\times G_{\cK, I_{i+1},I_{i+2}}\to G_{\cK,I_i\sqcup I_{i+1}}^{(j)}$$ (for $i<k-1$) or $$G_{\cK,I_{k-1},I_{k}}^{(j)}\times \Gr_{\cK, I_{k}}\to \Gr_{\cK,I_{k-1}\sqcup I_{k}}^{(1)}$$ (for $i=k-1$) defined in \eqref{multiplication-BD} and following.
	\end{itemize}
	
	We can thus define a map $$\delta_i=\colim_{I_1,\dots,I_k\in \fs^\op}``\lim_{j\geq j_1}"\delta_{i,I_1,\dots,I_k}^{(j)}:\hck_{\Ran,k}\to\hck_{\Ran,k-1}.$$
\end{construction}

\begin{rem}\label{Segal-and-stratifications}
	
	The maps defined in \cref{semisimplicial-structure} are stratified. Ultimately, this can be reduced to the following statement: for any point $x\in X$, the map $$\delta_{i,x,m_j}:\hck_{x,m_j}\to\hck_{x,m_j-1}$$ 
	which glues the data in the places $i,i+1$ is stratified. This is implied by the fact that the right leg in \eqref{truncated-convolution-body} is.
\end{rem}

\begin{prop}\label{simplicial}\cref{semisimplicial-structure} defines a semisimplicial object, i.e. the given maps satisfy the simplicial identities.\end{prop}
\begin{proof}We need to prove that, for every $k\geq 1, 0\leq i<h\leq k$, there are isomorphisms $$\delta_i\delta_h\simeq \delta_{h-1}\delta_i.$$ Equivalently, we prove that the statement is true for the $\delta_{i,I_1,\dots,I_k}^{(j)}$'s, and moreover, it suffices to prove it after pulling back to strata of $X^{I}$, where $I=I_1\sqcup\dots\sqcup I_k$. This last reduction step is not completely obvious, since for instance, the equality of two morphism of schemes which are not reduced cannot, in general, be checked on strata. However, by \cite[Theorem 1.2.1]{James-GrRan}, $\Gr_\Ran$ can be written ad $$\colim_{I\in \fs^\op}\Gr_I^\red$$ and therefore passing to reductions does not change the object at the Ran level. The same argument holds for the $\hck_{\Ran,k}$, because we have a map $\hck_{\Ran,k}\to \Gr_\Ran^{\times k}$ which is a relative $G_{\cO,\Ran}$ torsor, in particular a smooth map. In other words, 
	
	\begin{itemize}\item For $I\in\fs$, $G_{\cO,I}$ is smooth by \cref{smoothness-arc-group}, hence reduced.
		\item For $I_1,I_2\in\fs$, $G_{\cK,I_1,I_2}$ is a $G_{\cO,{I_2}}$-torsor over $\Gr_{I_1}\times X^{I_2}$, relative to $X^{I_2}$. If we pull this torsor back along $\Gr_{I_1}^\red\to \Gr_{I_1}$, we obtain another $G_{\cO,I_2}$-torsor, whose source is now reduced by \'etale-localness \cite[\texttt{Tag 06QM}]{Stacks-Project}, smoothness of the fiber and reducedness of the base. By the universal property of the reduction, we get that $$G_{\cK,I_1,I_2}\times_{\Gr_{I_1}}\Gr_{I_1}^\red\simeq G_{\cK, I_1,I_2}^\red.$$
		
		\item By a similar argument, $$\Conv_{I_1,\dots,I_k}\simeq \Conv_{I_1,\dots,I_k}\times_{\Gr_{I_1}\times \dots\times \Gr_{I_k}}\Gr_{I_1}^\red\times \dots\times \Gr_{I_k}^\red.$$ 
		More precisely, we use the previous point to deduce the statement for $$G_{\cK,I_1,I_2}\times_{X^{I_2}}\dots \times_{X^{I_k}}\Gr_{I_k},$$ and then again \'etale descent for reducedness to obtain the statement for the convolution Grassmannian.
		
		\item We have \begin{gather*}(\hck^{(j)}_{I_1,\dots,I_k})^\red\simeq G_{\cO,I_1\sqcup \dots\sqcup I_k}^{(j)}\backslash \Conv_{I_1,\dots,I_k}^\red\simeq G_{\cO,I_1,\dots,I_k}^{(j)}\backslash (\Conv_{I_1,\dots,I_k}\times_{\Gr_{I_1}\times \dots\times \Gr_{I_k}}\Gr_{I_1}^\red\times \dots\times \Gr_{I_k}^\red\\
			\simeq \hck_{I_1,\dots,I_k}^{(j)}\times_{\Gr_{I_1}\times \dots\times \Gr_{I_k}}\Gr_{I_1}^\red\times \dots\times \Gr_{I_k}^\red\end{gather*}
		where we used \'etale descent for reducedness for the first equivalence, the previous point for the second equivalence, and universality of quotients for the last equivalence.
		
		\item Finally, we apply the previous point, the fact that the functor $\cc\to \Pro(\cc)$ preserves finite limits, that limits commute with limits, and again universality of colimits to obtain $$\hck_{\Ran,k}=\colim_{I_1,\dots,I_k\in \fs^\op}``\lim_{j\geq j_1} "(\hck_{I_1,\dots,I_k}^{(j)})^\red.$$
		
		\end{itemize} We can thus prove that the statement is true for the reduction of the $\delta_{i,I_1,\dots,I_k}^{(j)}$'s, and this statement can be checked on strata. 
		Moreover, by \cite[Lemma 4.2.2]{James-GrRan}, the factorization property holds for $\Gr_\Ran^\red$ as well (and with a similar proof for its variations), which allows us to use factorization also in the reduced setting.
	
	 By factorization, we reduce to two cases:
	\begin{itemize}
		\item Proving the statement over the stratum $$\cS_{I_1,\dots,I_k,\disj}=\{x_I\mid \forall i', h'\leq k, x_a\neq x_b\forall a\in I_{i'},b\in I_{h'}\}.$$
		This case follows again by factorization.
		\item Proving the statement after pullback to the diagonal $X\to X^I$. This case reduces to a statement about associativity of the multiplication on $G_\cK$ (cf. \cite[§2]{NPnote}).
		
	\end{itemize}
\end{proof}

\begin{defin}\label{HckRanbullet}
	We denote the semisimplicial object established in \cref{simplicial} by $$\hck_{\Ran,\bullet}:\dinj\to \PStrStk.$$
\end{defin}

The crucial property of this structure, in order to encode the associativity of the convolution product, is the following:

\begin{prop}\label{hck-is-2-Segal}The semisimplicial object $\hck_{\Ran,\bullet}$ enjoys the 2-Segal property, that is the equivalent conditions of \cite[Proposition 2.3.2]{DK1}.\end{prop}
\begin{proof}We will prove the case $k=3$, for simplicity. First of all, we prove that we can reduce to proving that the map \begin{equation}\label{truncated-segal}
	\hck_{I_1,I_2,I_3}^{(j)}\to \hck_{I_2,I_3}^{(j)}\times_{\hck_{I_2\sqcup I_3}^{(j)}}\hck_{I_1,I_2\sqcup I_3}^{(j)}\end{equation} is an equivalence for every $I_1,I_2,I_3,j\geq j_1$, where the maps in the pullback are respectively $(\delta_{1,I_2,I_3}^{(j)}\circ \delta_{0,I_1,I_2,I_3}^{(j)})$ and $\delta_{2,I_1,I_2, I_3}^{(j)}$. Indeed, since limits commute with limits, we only need to prove that we can check the statement at the level of pro-objects (i.e. before passing to colimits in $I_1,I_2,I_2\in \fs^\op$). Consider the square

	\begin{equation}\label{universality}\begin{tikzcd}
		{\hck_{\Ran,I_2,I_3}} & {\hck_{\Ran,I_2\sqcup I_3}} & {\hck_{\Ran,2}} \\
		{\hck_{I_2, I_3}} & {\hck_{I_2\sqcup I_3}} & {\hck_{\Ran}}
		\arrow[from=1-1, to=1-2]
		\arrow[from=1-1, to=2-1]
		\arrow[from=1-2, to=1-3]
		\arrow[from=1-2, to=2-2]
		\arrow[from=1-3, to=2-3]
		\arrow[from=2-1, to=2-2]
		\arrow[from=2-2, to=2-3]
	\end{tikzcd}\end{equation}
	 where $\hck_{\Ran,I_2,I_3}=\colim_{I_1\in \fs^\op}\hck_{I_1,I_2,I_3}$ and similarly for $\hck_{\Ran,I_2\sqcup I_3}$. If we assume the statement about \eqref{truncated-segal}, then the left-hand square is Cartesian by universality of colimits. The right-hand square is also Cartesian, because front, bottom and back faces of the commutative cube 
	
		\[\begin{tikzcd}
		& {\hck_{\Ran,I_2\sqcup I_3}} && {\hck_{\Ran,2}} \\
		{\hck_{I_2\sqcup I_3}} & {} & {\hck_{\Ran}} & {} \\
		{} & {\Ran(X)\times X^{I_2\sqcup I_3}} && {\Ran(X) \times \Ran(X)} \\
		{X^{I_2\sqcup I_3}} & {} & {\Ran(X)}
		\arrow[from=1-2, to=1-4]
		\arrow[from=1-2, to=2-1]
		\arrow[from=1-2, to=3-2]
		\arrow[from=1-4, to=2-3]
		\arrow[from=1-4, to=3-4]
		\arrow[from=2-1, to=2-3]
		\arrow[from=2-1, to=4-1]
		\arrow[from=2-3, to=4-3]
		\arrow[from=3-2, to=3-4]
		\arrow[from=3-2, to=4-1, "\textup{pr}_2"]
		\arrow[from=3-4, to=4-3, "\textup{pr}_2"]
		\arrow[from=4-1, to=4-3]
	\end{tikzcd}\]

	are Cartesian. Hence the outer square in \eqref{universality} is Cartesian, and by universality of colimits in a category of presheaves we conclude.
	
	Let us thus prove the statement for the $\delta^{(j)}_{I_1,\dots,I_k}$'s. By factorization (and the workaround including passing to reductions explained in the proof of \cref{simplicial}), it suffices to:
	
	\begin{itemize}
		\item prove the statement after restricting to the stratum $$\cS_{I_1,I_2,I_3,\disj}=\{x_{I_1},x_{I_2},x_{I_3}\mid x_a\neq x_b\forall a\in I_2,b\in I_2\textup{ or }a\in I_2,b\in I_3\textup{or }a\in I_1,b\in I_3\}.$$ There, by \cref{factorization-property-BD}, \cref{factorization-GO} and \cref{factorization-Hecke-multiple}, the statement is equivalent to proving that the square
		
		\[\begin{tikzcd}
			{(\hck_{I_1}^{(j)}\times\hck_{I_2}^{(j)}\times \hck_{I_3}^{(j)})|_{\cS_{I_1,I_2,I_3, \disj}}} & {(\hck_{I_1}^{(j)}\times\hck_{I_2}^{(j)}\times \hck_{I_3}^{(j)})|_{\cS_{I_1,I_2,I_3,\disj}}} \\
			{(\hck_{I_2}^{(j)}\times \hck_{I_3}^{(j)})|_{\cS_{I_2,I_3,\disj}}} & {(\hck_{I_2}^{(j)}\times \hck_{I_3}^{(j)})|_{\cS_{I_2,I_3,\disj}}}
			\arrow[shift left, no head, from=1-1, to=1-2]
			\arrow[no head, from=1-1, to=1-2]
			\arrow[from=1-1, to=2-1]
			\arrow[from=1-2, to=2-2]
			\arrow[no head, from=2-1, to=2-2]
			\arrow[shift left, no head, from=2-1, to=2-2]
		\end{tikzcd}\]
		
		is Cartesian, which is of course true.
		
		\item prove the statement after restriction to the diagonal. We can as well assume $|I_1|=|I_2|=|I_3|$, up to considering suitable surjections $I_i'\to I_i, i=1,2,3$. Then we need to prove that, for any $N\in \nn, j\geq j_N$, the diagram 
		\[\begin{tikzcd}
			{G_{\cO}^{(j)}\backslash G_{\cK}^{(N,j)}\times^{G_{\cO}^{(j)}}G_{\cK}^{(N,j)}\times^{G_{\cO}^{(j)}}\Gr^{(N)}} & {G_{\cO}^{(j)}\backslash G_{\cK}^{(N,j)}\times^{G_{\cO}^{(j)}}\Gr^{(2N)}} \\
			{G_{\cO}^{(j)}\backslash G_{\cK}^{(N,j)}\times^{G_{\cO}^{(j)}}\Gr^{(N)}} & {G_{\cO,j}\backslash \Gr^{(2N)}}
			\arrow["{m_{23}}"', from=1-1, to=1-2]
			\arrow["{p_{23}}", from=1-1, to=2-1]
			\arrow["{p_2}"', from=1-2, to=2-2]
			\arrow["m", from=2-1, to=2-2]
		\end{tikzcd}\]
		
		is Cartesian (note that here we do need $N$ to be arbitrary, since pulling back to diagonals takes into account the cardinalities of the $I_i$'s). This latter claim is true by \cite[Proposition 2.3]{NPnote} with $H=G_\cK, K=G_\cO$.
	\end{itemize}
\end{proof}

\begin{thm}\label{hck-algebra-correspondences}
	Let $G$ be a complex reductive group and $X$ a complex smooth curve. The object $$\hck_{\Ran}\in\PStrStk$$ from \cref{HckRan}
	carries a nonunital associative algebra structure in $\Corr(\PStrStk)^\times$, extending the convolution diagram \eqref{BD-quotient-convolution-diagram}.
	
	The fiber of this algebra structure at any singleton $\{x\}\in \Ran(X)$ encodes the convolution diagram for the Hecke stack (\cref{quotient-convolution-diagram}) and its associativity.
\end{thm}
\begin{proof}It suffices to apply \cite[Proposition 8.1.5]{DK1} (or, more precisely, its nonunital part) to the 2-Segal semisimplicial object $\hck_{\Ran,\bullet}$ (\cref{semisimplicial-structure}, \cref{hck-is-2-Segal}).\end{proof}


\section{Topological factorization of the Hecke stack}\label{Section-3}

\subsection{Consequences of analytification}\label{Section-3.1}

\begin{defin}\label{Ran-space-topology} Let $M$ be a topological manifold of dimension $m$. The \textit{Ran space} of $M$ is the set of nonempty finite subsets of $M$, endowed with its so-called \textit{metric topology}, i.e. the topology induced by the following base: $$
	\{\prod_{i=1}^k\Ran(D_i)\mid \{D_i\} \textup{ finite family of pairwise disjoint disks in } M\}.$$
\end{defin}
Note that the given family is actually a family of subsets of $\Ran(M)$, in the sense that the union map $\prod_{i=1}^k\Ran(D_i)\to \Ran(M)$ is injective whenever the disks are pairwise disjoint. We denote by $\star_i \Ran(D_i)$ the image of this map.

\begin{rem}\label{two-Rans}The set $\Ran(M)$ can be equivalently presented as the $\colim_{I\in\fs} M^I$ in $\sets$. This carries a natural colimit topology, which is finer than the one presented above. For a thorough comparison between those, and many further considerations, see \cite{Cepek-Lejay} and \cite{Sylvain-Ran}. When unspecified, by $\Ran(M)$ we will always mean the Ran space with its metric topology.\end{rem}

One of the main features of the Ran space is to encode the notion of \textit{factorization algebra} in a particularly efficient way. Actually, a theorem by Jacob Lurie says even more, showing that a family over the Ran space with some conditions induces a ``sheaf of $\E_m$-algebras'' over $M$ ($m$ being the dimension of $M$).

\begin{thm}[{part of \cite[Theorem 5.5.4.10]{HA}}]Let $M$ be a topological manifold of dimension $m$, and $\cc^\otimes$ a symmetric monoidal $\infty$-category where the tensor product preserves colimits separately in both variables. Any \textup{constructible factorizable} cosheaf over $\Ran(M)$ with values in $\cc$ induces a nonunital $\E_m$-algebra structure on its stalk at any singleton $\{x\}\in \Ran(M)$.
\end{thm}

\begin{recall}\label{recall-factorizing-constructible}We refer to \cite[Section 5.5.4]{HA} for the definitions. Here we just recall that:
	\begin{itemize}\item two open subsets $U,V$ of $\Ran(M)$ are declared to be \textit{independent} if for every $S\in U,T\in V$, then $S\cap T=\varnothing$. For example, for any two open subsets $D,D'$ of $M$, $\Ran(D)$ and $\Ran(D')$ are independent if and only if $D\cap D'=\varnothing$. For $U,V$ independent open subsets of $\Ran(M)$, one denotes $U\star V=\{S\sqcup T\mid S\in U,T\in V\}$. Again, this is homeomorphic to $U\times V$.
		\item  there is an operadic structure $\Fact$ over the category of open subsets of $\Ran(M)$, encoding the ``partial operation'' $U\star V$ for independent $U,V$.
		
		\item a \textit{factorizable cosheaf} over $\Ran(M)$ is a map of operads $$A^\otimes:\Fact\to \cc^{\otimes}$$ satisfying the following conditions:
		\begin{itemize}\item the underlying functor $A: \Open(\Ran(M))\to \cc$ is a cosheaf;
			\item for any two independent open subsets $U,V$ of $\Ran(M)$, the map $A(U)\otimes^\cc A(V)\to A(U\star V)$ induced by the fact that $A$ is a map of operads is an equivalence in $\cc$.
		\end{itemize}
		\item such a functor is \textit{constructible} if, as a cosheaf over $\Ran(M)$ with values in $\cc$, it is hypercomplete and satisfies the following (cf. \cite[Proposition 5.5.1.14]{HA}). Let $\{D_i\}_{i\in I}$, $\{D_i'\}_{i\in I}$ be finite sequences of disks in $M$ such that $D_i\cap D_j=D_i'\cap D_j'=\varnothing$ for $i\neq j$ and $D_i\subset D_i'$. Then $A$ sends the induced map \begin{equation}\label{inclusion-of-disks}\star_{i}\Ran(D_i)\hookrightarrow\star_i\Ran(D'_i)\end{equation} into an equivalence in $\cc$.
	\end{itemize}
\end{recall}

Recall the definition of the functor $$(-)^\an:\PStrStk\to \PStrTStk,$$ from \cref{big-stratified-analytification}. By construction, this functor preserves finite limits. By composing the semisimplicial object $\hck_{\Ran,\bullet}$ from \cref{HckRanbullet} with $(-)^\an$ we obtain a semisimplicial object $\hck_{\Ran,\bullet}^\an$, which is again 2-Segal because $(-)^\an$ preserves finite limits. This object comes with a natural map towards $\Ran(X)^\an$ with the colimit topology (by functoriality) and therefore towards $\Ran(X^\an)$ with the metric topology described in \cref{Ran-space-topology}. For instance, if we consider $\hck^\an_{\Ran,1}$, this is given by 
$$\colim_{I\in\fs^\op} ``\lim"_{j\geq j_1}(G_{\cO,I}^{(j)})^\an\backslash(\Gr_{G,I}^{(1)})^\an.$$
Note that the single terms appearing in this formula are quotients in the category of stratified topological stacks (\cref{stratified-topological-stacks}).

From now on, when proving a property for $\hck^\an_{\Ran,k}$ we will always argue as ``we prove the property for each term and the property is stable under the relevant colimits and formal limits''.


\begin{defin}\label{cosheaf}For any $U\subset \Ran(X^\an)$ open subset, and $k\geq 1$, let 
	\begin{gather*}\Conv_{U,k}=\Conv_{\Ran,k}^\an\times_{\Ran(X^\an)^k}U^k\in\PStrat\\
	\hck_{U,k}=\hck_{\Ran,k}^\an\times_{\Ran(X^\an)^k}U^k\in\PStrat.\end{gather*} We define functors
	\begin{gather*}\Open(\Ran(M))\times\dinj\to \PStrat\\
	(U,[k])\mapsto \Conv_{U,k}\end{gather*}
	and 
	\begin{gather*}\Open(\Ran(M))\times\dinj\to \PStrat\\
	(U,[k])\mapsto \hck_{U,k}\end{gather*} on objects, and in the natural way on morphisms.
\end{defin}
Indeed, an inclusion $U\subset V$ naturally induces embeddings \begin{gather*}\Conv_{U,k}\hookrightarrow \Conv_{V,k}\\
	\hck_{U,k}\hookrightarrow\hck_{V,k}\end{gather*} which are clearly natural in $k$.
\begin{rem}\label{D^I}
	Note that by universality of colimits
	
	$$\Conv_{U,k}\simeq \colim_{I_1,\dots,I_k}\Conv^\an_{I_1,\dots,I_k}\times_{\Ran(X^\an)^k}U^k.$$
	resp. 
	$$\hck_{U,k}\simeq \colim_{I_1,\dots,I_k}\hck^\an_{I_1,\dots,I_k}\times_{\Ran(X^\an)^k}U^k.$$ We denote the terms appearing in the colimit by $\Conv_{U,I_I,\dots,I_k}$ resp. $\hck_{U,I_1,\dots,I_k}$ or, when $U=\Ran(D)$ for a disk $D$, by $\Conv_{D^{I_1}, \dots, D^{I_k}}$ resp. $\hck_{D^{I_1}, \dots, D^{I_k}}$, to stress the fact that each of them is equivalent to $$\Conv_{I_1,\dots,I_k}^\an\times_{X^{I_1\sqcup \dots\sqcup {I_k}}}D^{I_1\sqcup\dots\sqcup I_k}$$ resp. $$\hck_{I_1,\dots,I_k}^\an\times_{X^{I_1\sqcup \dots\sqcup {I_k}}}D^{I_1\sqcup\dots\sqcup I_k}.$$
\end{rem}
\cref{cosheaf} is the first step in the direction of building a factorization algebra out of $\hck_{\Ran,\bullet}^\an$. The following step is to upgrade $\hfb$ to a functor $$\Fact\times\dinj\to \PStrat^\times$$ suitably encoding the ``factorization property'' of the Beilinson--Drinfeld Grassmannian.

\begin{notation}
	Let $M$ be a topological manifold. We denote by $$(\Ran(M)^k\times\Ran(M)^k)_{\disj}$$ the topological space $$\{S_1,\dots,S_k,T_1\dots,T_k\mid S_i\cap T_i=\varnothing,i=1,\dots,k\}\subset\Ran(M)^k\times\Ran(M)^k.$$
\end{notation}

\begin{rem}\label{fusion-map}
	Consider two independent open subsets $U$ and $V$ of $\Ran(X^\an)$. We have the following diagram
	\begin{equation}\label{factprop}\begin{tikzcd}[column sep="0.7cm"]
			{\hck_{U,k}\times \hck_{V,k}} \arrow[d,"\pi"] \arrow[r] & {\hdis^\an} \arrow[r] \arrow[d]          & {\hrk^\an} \arrow[d] \\
			U^k\times V^k \arrow[r, "\subset"]                                            & (\Ran(X^\an)^k\times\Ran(X^\an)^k)_\disj \arrow[r, "\union"] & \Ran(X^\an)^k      ,                
	\end{tikzcd}\end{equation}
	where the left hand square is by definition a pullback in $\PStrat$, and the right top horizontal map is induced by \cref{semisimplicial-structure} and factorization with $k=2$, and then by applying $(-)^\an$. 
	Note that by \cref{big-stratified-analytification}, $(-)^\an$ preserves finite limits. 

	In concrete, this map is induced by the maps \begin{gather*}\Gr_{I_1}\times\Gr_{I_2}|_{(X^{I_1}\times X^{I_2})_\disj}\to \Gr_{I_1\sqcup I_2}\\ (x_{I_1},\cF,\alpha,x_{I_2},\cG,\beta)\mapsto (x_{I_1}+x_{I_2},\cH, \gamma)\end{gather*} where $\cH$ is the gluing of $\cF$ and $\cG$ along the isomorphism $\beta^{-1}|_{X\setminus \Gamma_{x_{I_1}+x_{I_2}}}\circ \alpha|_{X\setminus \Gamma_{x_{I_1}+x_{I_2}}}$ and $\gamma$ is the trivialization inherited from $\alpha$ and $\beta$ via the gluing procedure.
	
	Note also that the bottom composition in \eqref{factprop} coincides with $U\times V\to U\star V\hookrightarrow\Ran(X^\an)$, the first map being the one taking unions of systems of points; hence, by the universal property of the fibered product, $\hck_{U,k}\times\hck_{V,k}$ admits a map towards $\hck_{U\star V,k}$, which we call $\chi_{U,V,k}$.
\end{rem}

\begin{prop}\label{op-k}Let $k\in \mathbb N$. There is a well-defined map of operads $$\hck^\fact_k:\Fact\to\PStrat^\times$$ 
	assigning $$(U_1,\dots, U_m)\mapsto \hck_{U_1,k}\times\dots\times \hck_{U_m,k}.$$
	$$(U\subset V)\mapsto [\hck_{U,k}\hookrightarrow \hck_{V,k}]$$
	$$((U,V)\to (U\star V))\mapsto \chi_{U,V,k}:\hck_{U,k}\times \hck_{V,k}\to \hck_{U\star V, k}.$$
\end{prop}
\begin{proof}To prove that this is this association is functorial, it suffices to prove that the diagram $$
	\begin{tikzcd}[column sep=-1.5cm, row sep=1cm]
		& \Conv_{U\star V, k}\times\Conv_{W, k} \arrow[rd, "{\chi_{U\star V,W,k}}"]  &                            \\
		\Conv_{U,k}\times\Conv_{V,k}\times \Conv_{W,k} \arrow[rd, "{\id\times \chi_{V,W,k}}"'] \arrow[ru, "{\chi_{U,V,k}\times \id}"] &                                                                       & \Conv_{U\star V\star W,k} \\
		& \Conv_{U,k}\times\Conv_{V\star W, k} \arrow[ru, "{\chi_{U,V\star W,k}}"'] &                           
	\end{tikzcd}$$
	(notations as in \cref{cosheaf}) commutes in $\PStrat$. Now this is true because the operation of gluing is associative, as it is easily checked by means of the defining property of the gluing of sheaves.\\
	Finally, to prove that the functor $\hck_k^\fact$ is a map of operads, we use the characterization of inert morphisms in a Cartesian structure provided by \cite[Proposition 2.4.1.5]{HA}. Note that:
	\begin{itemize}
		\item An inert morphism in $\Fact$ is a morphism of the form $$(U_1,\dots, U_m)\to (U_{\phi^{-1}(1)},\dots U_{\phi^{-1}(n)})$$ covering some inert arrow $\phi:\langle m \rangle \to \langle n \rangle$ where every $i\in\langle n\rangle^\circ$ has exactly one preimage $\phi^{-1}(i)$. 
		\item An inert morphism in $\PStrat^{\times}$ is a morphism of functors $\overline \alpha$ between $f:\mathscr P(\langle m\rangle^\circ)^\op\to \PStrat$ and $g:\mathscr P(\langle n\rangle^\circ)^\op\to\PStrat$, covering some $\alpha:\langle m\rangle \to \langle n \rangle$, and such that, for any $S\subset \langle n \rangle$, the map induced by $\overline \alpha$ from $f(\alpha^{-1}S)\to g(S)$ is an equivalence in $\PStrat$.
	\end{itemize}
	By definition, $\hck^\fact_k(U_1,\dots,U_m)$ is the functor $f$ assigning $$T\subset \langle m\rangle^\circ\mapsto \prod_{j\in T}\hck_{U_j,k},$$ and analogously $\hck^\fact_k(U_{\phi^{-1}(1)},\dots,U_{\phi^{-1}(m)})$ is the functor $g$ assigning $$S\subset \langle n\rangle^\circ\mapsto \prod_{i\in S}\hck_{U_{\phi^{-1}(i)},k}.$$ But now, if $\alpha=\phi$ and $T=\phi^{-1}(S)$, we have the desired equivalence.
	
\end{proof}

For the following result, we specialize to the case $X=\Ac$. Consider the class $\MDisk(\C)$ of metric balls in $\Ac$, i.e. open subspaces of the form $$\{z\in (\Ac)^\an\mid|z-z_0|<r\}$$ for $z_0\in (\Ac)^\an, r\in \rr_{>0}.$
Let $D_i'\subset D_i, i=1,\dots,n$, be elements in $\MDisk(\C)$. 
There is an inclusion \begin{equation}\label{metric-balls-inclusion}\star_{i=1}^n\Ran(D'_i)\subset\star_{i=1^n} \Ran(D_i)\end{equation} which in turn induces a map $$\prod_{i=1}^n\hck_{\Ran(D_i')}\to\prod_{i=1}^n \hck_{\Ran(D_i)}$$ in $\PStrat$.

Recall the notion of stratified homotopy equivalence ($\she$ and $\pshe$) from \cref{defin-subm-prestacks}.

\begin{prop}\label{factorization+homotopy-invariance}Let $X=\Ac$, $k\in \mathbb N$. The functor $\hf_k$ from \cref{op-k} satisfies the factorization condition (see \cref{recall-factorizing-constructible}) and sends maps of the form \eqref{metric-balls-inclusion} to $\pshe$ in $\Mor(\PStrat)$.
\end{prop}

\begin{proof}Recall \cref{big-stratified-analytification} and \cref{factorization-Hecke-multiple} for $k=2$. The diagram $$
	\begin{tikzcd}
		(\Ran(X)^\an\times \Ran(X)^\an)_{\disj} \arrow[r] \arrow[d] & \Ran(X)^\an\arrow[d] \\
		(\Ran(X^\an)\times \Ran(X^\an))_\disj \arrow[r]            & \Ran(X^\an)          
	\end{tikzcd}$$ (where in the first row we are considering the colimit topology and in the second row we are considering the metric topology) is Cartesian. We thus have the factorization condition by pulling back everything along $U^k\times V^k\to (\Ran(X^\an)^{k}\times\Ran(X^\an)^k)_\disj$.

	As for the second property, by using factorization we reduce to the case of a single inclusion of disks $i:D'\subset D$. That is, we need to prove that the induced map $$\hck_{\Ran(D'), k}\to \hck_{\Ran(D),k}$$ is a stratified homotopy equivalence in $\PStrat$. This amounts to proving that for each $I_1,\dots,I_k\in \fs,N\in\nn$, the maps 
	\begin{gather*}\Conv_{(D')^{I_1},\dots, (D')^{I_k}}^{(N)}\to \Conv_{D^{I_1},\dots,D^{I_k}}^{(N)}\\
	G_{\cO,(D')^{I_1}}\to G_{\cO,D^{I_1}}\end{gather*}
	 are resp. stratified homotopy equivalences and homotopy equivalences, and in a compatible way (i.e. the homotopy inverse of the first map is equivariant with respect to the second one, and one can choose homotopies which are compatible with respect to the actions on source and target). A sketch of the proof for the first one, in the case $k=1$, is given in \cite[Proposition 3.17]{HY}. The case $k=1$ together with the group case and equivariance are proven in \cite[Theorem A, Theorem C]{WM}. The case $k>1$ is proven \textit{verbatim} with the same argument as \cite[Theorem A, Theorem C]{WM}.
\end{proof}

\begin{rem}\label{not-cosheaf}
From the properties of \cref{recall-factorizing-constructible}, we did not prove that the association $U\mapsto \hck_{U,k}$ is a hypercomplete\footnote{Since alle categories involved are ordinary categories, the adjective ``hypercomplete'' is redundant. It will become relevant after taking constructible cosheaves, see \cref{hypercompleteness}.} cosheaf. This property becomes true after taking categories of constructible sheaves, that is after applying the functor $\Cons$ from \cref{take-constructibles-final}: see \cref{hypercompleteness}.
Note that for $I_1,\dots,I_k\in \fs, j\geq j_1$, the functor $U\mapsto \hck_{I_1,\dots,I_k}^{(j)}$ is indeed a hypercomplete cosheaf by \cref{lemma-hypercompleteness}. However, this does not trivially extend to the functor $U\mapsto \hck_{I_1,\dots,I_k}$, because cofiltered limits do not commute with sifted colimits and hence one cannot formally transfer descent to the level of pro-objects. 

Note as well that we only proved that maps of the form \eqref{metric-balls-inclusion} (and not \eqref{inclusion-of-disks}) are sent to stratified homotopy equivalences, which are not equivalences in $\PStrTStk$. However, stratified homotopy equivalences are sent to equivalences of $\infty$-categories under $\Cons(-;\cE)$, the functor which takes constructible sheaves with coefficients in a symmetric monoidal presentable stable $\infty$-category $\cE$ (\cref{homotopy-invariance-of-cons}). Also, proving the property for maps of the form \eqref{metric-balls-inclusion} instead of \eqref{inclusion-of-disks} is sufficient for our purposes, as we will see in the proof of \cref{Sph-is-factorizable-constructible}.
\end{rem}

\begin{rem}\label{fusion-and-faces}The constructions performed in the proof of \cref{op-k} are compatible with the face maps of $\hck^\an_{\Ran,\bullet}$ defined in \cref{semisimplicial-structure}, since for any $k\in \dinj, i=1,\dots k$, the square $$
	\begin{tikzcd}
		\hdis \arrow[d, "\delta_i\times \delta_i"] \arrow[r] & \hrk \arrow[d, "\delta_i"] \\
		(\hck_{\Ran,k-1}\times \hck_{\Ran,k-1})_\disj \arrow[r]       & \hck_{\Ran,k-1}.
	\end{tikzcd}$$ commutes.
\end{rem}








Therefore:

\begin{thm}\label{hck-fact}Let $G$ be a reductive complex group and $X=\Ac$. \cref{op-k} induces a well-defined map of operads $$\hck^\fact:\Fact\times \E_1^\nun\to \Corr(\PStrat)^\times$$ such that: \begin{itemize}\item in the first variable, it satisfies the factorization property in the sense of \cref{recall-factorizing-constructible}, and sends maps of the form \eqref{metric-balls-inclusion} to stratified homotopy equivalences;
		\item for every open $U\subset \Ran(X^\an)$, the restriction to $\{U\}\times \E_1^\nun$ yields the map of operads defined in \cref{hck-algebra-correspondences}, analytified and pulled back from $\Ran(X^\an)$ to $U$.\end{itemize}
\end{thm}
\begin{proof}\cref{fusion-and-faces} yields a functor $$\Fact\times\dinj\to \PStrat^\times$$ which is a map of operads in the first variable. 
In the second variable, the functor satisfies the 2-Segal condition in the same sense as the functor in \cref{hck-is-2-Segal}, since the extended stratified analytification functor preserves finite limits (\cref{big-stratified-analytification}).

This in turn induces a functor\footnote{The notation $\Map_{\textup{Op}_\infty}$ follows \cite{HA} and stays for ``maps of $\infty$-operads''.}$$\dinj\to\Map_{\textup{Op}_\infty}(\Fact,\PStrat^\times)$$ which again satisfies the 2-Segal condition. By invoking \cite[Proposition 8.1.5]{DK1} we obtain a map of operads $$\E_1^\nun\to \Corr(\Map_{\textup{Op}_\infty}(\Fact,\PStrat^\times))^\times.$$ Note that the target admits a map of operads to $$\Map_{\textup{Op}_\infty}(\Fact,\Corr(\PStrat)^\times)^\times$$ (it is easy to provide a map towards the category of functors, and then one can check that it actually takes values in the full subcategory spanned by maps of operads). Finally, we get a functor $$\Fact\times\E_1^\nun\to \Corr(\PStrat)^\times$$ which is a map of operads separately in both variables. The verification that the claimed properties hold is straightforward by restricting to the two separate variables. \end{proof}


\subsection{Setup for taking constructible sheaves}\label{subsection-setup-constructibles}

In the next section we will take constructible sheaves over the geometric objects introduced up to now and prove that this induces the sought-after $\E_3$-algebra structure on $\Sph(G)$. The first step, carried out in the present subsection, will be to check that the mentioned geometric objects (and maps between them) indeed belong to the source category of \cref{take-constructibles-final}.

\begin{rem}\label{stratifying-poset-is-finite}
	For $G=\GL_n, k\geq 1,I_1,\dots,I_k\in \fs,N\in \nn$ fixed, the stratifying poset of $\Conv_{I_1,\dots,I_k}^{(N)}$ is given, with the notations of \cref{BD-stratifications}, by $$\bigg\{J, [\phi:I_1\sqcup\dots\sqcup I_k\twoheadrightarrow J], (\nu_j^{h_i^j})_{j\in J,i=1,\dots,m_j}\in \prod_{j\in J}(\xt)^{m_j}\mid\sum_{j,\phi^{-1}(j)\cap I_h\neq\varnothing}\nu_j^{h}\leq (N,\dots,N)\ \forall h=1,\dots,k\bigg\}$$ which is finite, the bounds being induced by $N$ and the cardinality of $I_1\sqcup\dots\sqcup I_k$. For the case of a general $G$, choose a faithful representation $\rho:G\to \GL_n$: the bounds on coweights are induced by the bounds inherited from the case of $\GL_n$ (they may depend on $\rho$, but this is not an issue).
	
\end{rem}

\begin{prop}\label{BD-is-Whitney}
	Let $I\in \fs, N\in \nn$. The stratified space $(\Gr_{I}^{(N)})^\an$ belongs to the category $\Con$ from \cref{defin-con}.
\end{prop}
\begin{proof} It suffices to prove that the stratification is Whitney. Indeed, this implies that the stratification is conical in the sense of \cref{defin-con}, since strata are smooth manifolds and possess tubular neighbourhoods. From this, we also obtain that the conical neighbourhoods of each point can be chosen to be contractible: hence, the space is locally contractible (which implies that it is of singular shape by \cite[Lemma A.4.14]{HA}). The existence of tubular neighbourhoods has been proven by Mather \cite{TopStab}. Together with Marco Volpe, we provided an explicit reformulation in the language of conically stratified spaces \cite[Construction 3.4]{Whitney}. 
	
	The fact that $(\Gr_I^{(N)})^\an$ is Whitney is proven in \cite[Proposition 4.5.1]{Nadler-Perverse-real}.\end{proof}

\begin{lem}\label{bundle-stability-of-Whitney}
	Let $(X,s)\to (Y,t)$ be a map of stratified spaces embedded in $\rr^N$ for some $N$. Suppose that the map is surjective, and that it is a smooth stratified submersion in the sense of \cref{defin-smooth}. Then $(X,s)$ satisfies the Whitney conditions if an only if $(Y,t)$ does.
\end{lem}
 \begin{proof}Since the Whitney condition is local, and the map is surjective (hence any point in $y$ admits a neighbourhood which is the image of a trivializable neighbourhood in $X$) we can reduce the problem to the case of a product of a space $Y\subset \rr^N$ with a stratification $s$ and a euclidean space $\rr^n$ (considered with its trivial stratification), with projection $\pi:Y\times \rr^n\to Y$. If $(Y,s)$ is Whitney, then the product $Y\times \rr^n$ with the trivial stratification on $\rr^n$ is Whitney. Conversely, suppose that $Y\times \rr^n$ is Whitney with the trivial stratification on $\rr^n$. Pick two strata $W, Z\subset Y$, and two sequences $(w_i)\subset W\subset Y,(y_i)\subset Z\subset Y$ both converging to some $y\in Z$, such that also $w_iy_i\to v, T_{w_i}W\to \tau$ for some line $v$ and some vector space $\tau$. We can define liftings of the $w_i,y_i$'s by $(w_i,0),(y_i,0)$, where $0$ is the origin of $\rr^n$. The sequences of the secants and of the tangent spaces of these new points converge respectively to $v\times 0$ and $\tau\times 0$: therefore, we obtain that $v\subset \tau$ by applying the hypothesis that $Y\times \rr^n$ is Whitney.
\end{proof}

\begin{lem}\label{GO-is-Whitney}
	Let $I\in\fs,j\in\nn$. The stratified space $(G_{\cO,I}^{(j)})^\an$ belongs to $\Con$.
\end{lem}
\begin{proof}
	Since the structure map $G_{\cO,I}^{(j)}\to X^I$ is smooth (\cref{smoothness-arc-group}) and surjective, by \cref{analytification-of-smooth} the map $(G_{\cO,I}^{(j)})^\an\to (X^\an)^I$ is a stratified smooth submersion. Since the incidence stratification on $X^I$ is Whitney, we can apply \cref{bundle-stability-of-Whitney}.
\end{proof}

\begin{prop}\label{Conv-is-Whitney}The stratified space $(\Conv_{I_1,\dots,I_k}^{(N)})^\an$ belongs to $\Con$.
\end{prop}

\begin{proof}The following proof has been suggested to us by Robert Cass.

First of all, strata are smooth by factorization (\cref{factorization-property-Conv}) and \cref{strata-of-conv-are-smooth}. Note then that $(\Gr_{I_1}^{(N)})^\an\times\dots\times(\Gr_{I_k}^{(N)})^\an$ and $(\Conv_{I_1,\dots, I_k}^{(N)})^\an$ admit a common smooth cover $$G_{\cK,I_1,I_2}^{(N,j_N)}\times_{X^{I_2}}\dots\times_{X^{I_{k-1}}}G_{\cK, I_{k-1},I_k}^{(N,j_N)}\times_{X^{I_k}} \Gr_{I_k}^{(N)},$$ which projects onto each of them as a smooth bundle having as fiber the smooth unstratified relative group scheme $G_{\cO,I_2}^{(j_N)}\times\dots\times G_{\cO,I_k}^{(j_N)}$. If we pass to analytifications, $(\Gr_{I_1}^{(N)})^\an\times\dots\times(\Gr_{I_k}^{(N)})^\an$ is Whitney by \cref{BD-is-Whitney}, and by \cref{analytification-of-smooth-torsors} the two covers become surjective topological submersions. It now suffices to apply \cref{bundle-stability-of-Whitney} twice, one in each direction.
\end{proof}

\begin{cor}\label{hecke-is-conical}For any $k$, the object $\hck_{\Ran,k}^\an$ lies in $\PCon$.
\end{cor}
\begin{proof} We need to prove that each term in the formula $$\colim_{I_1,\dots,I_k\in \fs}``\lim_{j\geq j_1}''\colim_{[n]\in \Delta^\op}(G_{\cO,I,(j)}^{\an})^{\times_{X^{I}}{n}}\times_{X^{I}} (\Conv_{I_1,\dots,I_k}^{(1)})^\an$$ (where the inner colimit is the colimit along the usual simplicial diagram encoding the action, and is taken in $\Str\TStk$) belongs to $\Con$.
This statement is a formal consequence of the previous ones: the map $G_{\cO,I}^{(j)}\times_{X^{I}}\Conv^{(1)}_{I_1,\dots,I_k}\to \Conv^{(1)}_{I_1,\dots,I_k}$ is a smooth stratified submersion, hence the source is Whitney.
\end{proof}
Recall now \cref{defin-smooth}. We need to prove that:

\begin{prop}\label{horizontal-are-submersions}The functor $\hck^\fact$ from \cref{hck-fact} takes values in the subcategory $$\Corr(\PCon)^\times_{\all,\subm},$$ that is, that we have the right class of horizontal morphisms in order to apply \cref{take-constructibles-final}.
\end{prop}
\begin{proof}We need to describe the image of an arbitrary morphism $(\alpha, \phi):(U_1,\dots, U_m,\langle k\rangle)\to (V_1,\dots, V_n,\langle h\rangle) $ in $\Fact\times \E_1^\nun$ under the functor $\hck^\fact$. Suppose first that 
$\phi$ is of the form $\phi:\langle 2\rangle\to \langle 1\rangle$ (either one of the two ``projections'' or the ``multiplication'' map). Since the image of $(\alpha, \phi)$ can be factored, by definition, as the composition of the image of $(\alpha, \id)$ under $\hck^\fact(-,\langle 2\rangle)$ and the image of $(\id,\phi)$ under $\hck^\fact((V_1,\dots, V_n),-)$, let us inspect what happens on each component.
\begin{itemize}
	\item any pair $(\alpha, \id)$ will be sent to a vertical morphism.
	
	\item suppose that $\phi$ is inert. Then the pair $(\id_{(U_1,\dots, U_m)}, \phi)$ is sent to a vertical morphism, namely one of the two projections $$(\hck_{U_1}\times \dots\times \hck_{U_m})\times(\hck_{U_1}\times \dots\times \hck_{U_m})\to \hck_{U_1}\times \dots\times \hck_{U_m}.$$
	
	\item suppose finally that $\phi$ is the active morphism. There is the following list of progressive simplifications:
	\begin{itemize}
	\item We can assume that $m=1$. Indeed, since the object $\hck_{\Ran, \bullet}$ is 2-Segal, any other case can be recovered via pullback from the cases treated above (in the same spirit, see the definition before \cite[Proposition 8.1.7 of the arXiv version]{DK1}), and since our classes of vertical and horizontal morphisms are stable under pullbacks we conclude.

	We are thus dealing with the correspondence $$\begin{tikzcd}& \hck_{U,2}\arrow[rd]\arrow[ld]&\\
		\hck_{U}\times \hck_{U}&&\hck_{U}\end{tikzcd}$$ and we want to prove that the left leg belongs to $\subm$. 
	\item It suffices to prove the claim for $U=\Ran(X^\an)$. 

	\item It suffices to prove that the map $$\hck_{\Ran,2}\to\hck_\Ran\times\hck_\Ran$$ is representable and smooth, i.e. its pullback to any $\cX\in \Pro(\StrStklft)$ is a pro-smooth map of pro-stacks, in the sense of algebraic geometry, and that it is a torsor when restricted to any stratum of the target. By definition (\cref{defin-subm-prestacks}), the combination of this and \cref{hecke-is-conical} will ensure that its analytification is in $\subm$.
	\item It is sufficient to prove that for every $I_1,I_2\in \fs,I=I_1\sqcup I_2,j\geq j_{1}$
	\begin{equation}\label{p-bar-is-smooth}\overline p_{I_1,I_2}^{(j)}:\hck_{I_1,I_2}^{(j)}\to \hck_{I_1}^{(j)}\times\hck_{I_2}^{(j)}\end{equation} (notations as in \cref{convolution-grassmannians}) is smooth, and that it is a torsor when restricted to any stratum of the target. Indeed, let $\cX$ be any object in $\Pro(\StrStklft)$ together with a map $\cX\to \hck_\Ran\times\hck_\Ran$ in $\PStrStk$. Then there is a map $$\cX\to \Ran(X)\times \Ran(X),$$ which will factor via some $X^{I_1}\times X^{I_2}, I_1,I_2\in \fs$, because any representable object in a category of presheaves is atomic. Hence, the map $\cX\to \hck_\Ran\times\hck_\Ran$ factors as $$\cX\to \hck_{I_1}\times \hck_{I_2}$$ and therefore we have an equivalence $$\cX\times_{\hck_{\Ran}\times\hck_\Ran}\hck_{\Ran,2}\simeq \cX\times_{\hck_{I_1}\times\hck_{I_2}}\hck_{I_1,I_2}.$$ Note that this fiber product belongs to $\Pro(\StrStklft)$ as wanted, and the projection from it to $\cX$ is pro-smooth whenever the map \eqref{p-bar-is-smooth} is a smooth map of stacks for every $j$.
	\item	Recall that there is a morphism of schemes $$p_{I_1,I_2}^{(j)}:G_{\cK,I_1,I_2}^{(1,j)}\times_{X^{I_2}}\Gr_{I_2}^{(1)}\to \Gr_{I_1}^{(1)}\times \Gr_{I_2}^{(1)}$$ which takes the right $G_{\cO,I_2}^{(j)}$-quotient in the first component and is the identity in the second one. This fits into a commutative square of stacks

	$$
	\begin{tikzcd}
	G_{\cK,I_1,I_2}^{(1,j)}\times_{X^{I_2}}\Gr_{I_2}^{(1)}\arrow[r,"p_{I_1,I_2}^{(j)}"]\arrow[d]&\Gr_{I_1}^{(1)}\times \Gr_{I_2}^{(1)}\arrow[d]\\
	\hck_{I_1,I_2}^{(j)}\arrow[r,"\overline p_{I_1,I_2}^{(j)}"]& \hck_{I_1}^{(j)}\times\hck_{I_2}^{(j)}
	\end{tikzcd}
	$$
	where the leftmost vertical arrow exhibits the target as the quotient of the source, relative to $X^I$, with respect to the action of $G_{\cO,I}^{(j)}$ from \cref{BD-convolution-diagram}), and the rightmost vertical arrow exhibits the target as the quotient of the source, relative to $X^I$, with respect to the action of $G_{\cO,I_1}^{(j)}\times G_{\cO,I_2}^{(j)}$ Now, the map $p_{I_1,I_2}^{(j)}$ is a $G_{\cO,I_2,(j)}$-torsor, which implies that $\overline p_{I_1,I_2}^{(j)}$ is smooth in the sense of \cite[\texttt{Tag 075U}]{Stacks-Project}. We need the topological version of this: namely, by applying \cref{smoothness-arc-group} and \cref{analytification-of-smooth} to $p_{I_1,I_2}^{(j)}$, we obtain that $\big(\overline p_{I_1,I_2}^{(j)}\big)^\an$ belongs to $\subm'$ in the sense of \cref{defin-subm-prestacks}.

	Finally, the map is a torsor when restricted to any stratum of the target by the factorization property.
	\end{itemize}
\end{itemize}\end{proof}

\section{The \texorpdfstring{$\E_3$}{E3}-structure}\label{Section-4}

\subsection{Spherical Hecke category over the Ran space}
\begin{construction}\label{sphfact}Let $\cE$ a presentable stable symmetric monoidal $\infty$-category. We consider the composition $$\Sph(G;\cE)^\fact:\Fact\times \E_1^\nun\xrightarrow{\hck^\fact}\Corr(\PCon)^\times_{\all,\subm}\xrightarrow{\Cons^{\otimes, \corr}_\cE}\Prro_\cE$$ where the first functor is the one from \cref{horizontal-are-submersions} and the second functor is the one constructed in \cref{take-constructibles-final}.
Note that $\Sph(G;\cE)^\fact$ sends a morphism of the form $(\alpha, \id_{\langle k\rangle})$, where $\alpha$ is of the form \eqref{metric-balls-inclusion}, to an equivalence of $\infty$-categories by \cref{hck-fact} and \cref{take-constructibles-final}. Moreover, it satisfies the factorization property from \cref{recall-factorizing-constructible} since $\hck^\fact$ does (\cref{factorization+homotopy-invariance}) and $\Cons^\otimes_\cE$ is symmetric monoidal: in other words, by \cref{take-constructibles-final} the functor $$\Cons(\hck_U;\cE)\otimes_\cE \Cons(\hck_V;\cE)\to\Cons(\hck_U\times\hck_V;\cE)$$ is an equivalence for each independent $U,V\subset \Ran(X^\an)$. \end{construction}
We now want to prove that the functor $\Sph(G;\cE)^\fact$ is a hypercomplete cosheaf when restricted to $\Open(\Ran(M))$.

\begin{prop}\label{unipotent-equivalence}
Fix $I_1,\dots, I_k\in \fs$, and $j\geq j'\geq j_1$. Let $f:\hck_{I_1,\dots,I_k}^{(j)}\to \hck_{I_1,\dots,I_k}^{(j')}$ be the transition map. Then the functor  $$f^*:\Cons((\hck_{I_1,\dots,I_k}^{(j')})^\an;\cE)\to \Cons((\hck_{I_1,\dots,I_k}^{(j)})^\an;\cE)$$ 
is an equivalence.
\end{prop}
\begin{proof}
This follows from the fact that the transition maps are in $\uni$ (\cref{transitions-are-uni}), hence their analytification is in $\tri$ by \cref{uni-tri}, and one can apply \cref{tri-equivalences}.
\end{proof}

\begin{rem}\label{description-colimit-Sph}The $\infty$-category \eqref{underlying-cat-stalk} is, by the construction in \cref{take-constructibles-final}, equivalent to 
\begin{equation}\label{description-of-Sph-1}\colim_{I\in \fs^\op}\hspace{-0.2cm}^{\Prr_\cE}\Cons(\hck_{D^I}^{(j)};\cE)\end{equation} where $j$ is any number greater than $j_1$. 
The transition functors, up to a change of index $j$ (which induces an equivalence on the categories of constructible sheaves by \cref{unipotent-equivalence}) in $I$ are given by $*$-pushforward along the closed embeddings\footnote{Recall that these are closed embeddings of unions of strata, hence $*$-pushforward preserves constructible sheaves.} $\hck_{D^I}\to \hck_{D^J}$ induced by $J\twoheadrightarrow I$. This colimit is not filtered. As a consequence of \cite[Theorem 5.5.3.13]{HTT}, it corresponds to the limit in $\widehat\cat_{\infty,\cE}$\footnote{This is the $\infty$-category of large $\cE$-linear $\infty$-categories, defined similarly to \cref{small-linear-categories}.} taken with $^*$-pullback functoriality.

This means that \eqref{description-of-Sph-1} can be rewritten as 

\begin{equation}\label{description-of-Sph-2}\lim_{I\in \fs}\hspace{-10pt} ^{(-)^{*}}\Cons(\hck_{D^I}^{(j)};\cE).\end{equation}

 Here $\Cons(\hck_{D^I};\cE)$ coincides with $\Cons(\hck_{D^{I}}^{(j)};\cE)$ for any $j\geq j_1$. The reason we chose to adopt the pro-object perspective is that, without the possibility of increasing $j$, some maps are impossible to define (notably, the ones associated to surjections $I\to J$ in \cref{HckRan}).

Finally, each term $\Cons(\hck_{D^I}^{(j)};\cE)$ is computed, again by construction, \`a la Bernstein-Lunts, i.e. as $${\colim_{[n]\in\Delta^\op}} ^{\Prr_R,(-)_\vdash}\Cons((G_{\cO,D^I}^{(j)})^{\times_{D^I} n}\times_{D^I} \Gr_{D^I}^{(1)};\cE)\simeq {\lim_{[n]\in\Delta}}^{\lCatE,(-)^*}\Cons((G_{\cO,D^I}^{(j)})^{\times_{D^I} n}\times_{D^I} \Gr_{D^I}^{(1)};\cE)$$ where the limit is taken along the simplicial diagram encoding the action of $G_{\cO,D^I}^{(j)}$ on $\Gr_{D^I}^{(1)}$, with pullback functoriality.\end{rem}

\begin{lem}\label{lemma-hypercompleteness}
	Let $p:Y\to Z$ be a map of topological spaces, and let $\cF$ be the functor \begin{gather*}\Open(Z)\to \Top\\
		U\mapsto p^{-1}(U).
	\end{gather*} Then $\cF$ is a hypercomplete cosheaf on $Z$.
\end{lem}
\begin{proof}
	The claim amounts to the fact that, for every $U\in \Open(Z)$ and any hypercovering $\cU:\dinj\to \Open(Z)$ of $U$, the map $$\colim_{n\in \dinj}p^{-1}(\cU_n)\to p^{-1}(U)$$ between open subsets of $Y$ is a homeomorphism, which is true.
\end{proof}
\begin{prop}\label{hypercompleteness}
	The restriction of the functor $\Sph(G;\cE)^\fact$ from \cref{sphfact} to $\Open(\Ran(M))\times \{<k>\}$ is a hypercomplete cosheaf for every $<k>\in \E_1^\nun$. 
\end{prop}

\begin{proof}Note that, for each $I_1,\dots,I_k\in \fs, N\in \nn, j\geq j_1$, the functors $U\mapsto \Conv_{U,I_1,\dots,I_k}^{(N)}$ and $U\mapsto G_{\cO,U,I}^{(j)}$ are of the form appearing in \cref{lemma-hypercompleteness}, and hence hypercomplete cosheaves. Therefore, the functor $U\mapsto G_{\cO,U,I}^{(j)}\backslash\Conv_{U,I_1,\dots,I_k}^{(1)}$ (for $j\geq j_N$) is again hypercomplete, since it arises as a colimit of hypercomplete cosheaves. 

Therefore, by \cref{cons-hyperdescent}, the functor $U\mapsto \Cons(\hck_{U,I_1,\dots,I_k}^{(j)}; \cE)$ is a hypercomplete cosheaf as well. 

By the discussion in \cref{description-colimit-Sph}, this functor coincides with $U\mapsto \Cons(\hck_{U,I_1,\dots,I_k};\cE)$. Finally, the functor $U\mapsto \Cons(\hck_{\Ran,k};\cE)$ is a hypercomplete cosheaf because it arises as a colimit of hypercomplete cosheaves.\end{proof}
Summing up:

\begin{thm}\label{Sph-is-factorizable-constructible}Let $G$ be a complex reductive group, $X=\Ac, M=X^\an=\C$, and $\cE$ a presentable stable symmetric monoidal $\infty$-category.
	
	The functor $$\Sph(G;\cE)^\fact:\Fact\times\E_1^\nun\to \Prlo_\cE$$ from \cref{sphfact} has the following properties:
	\begin{itemize}
		\item It is a hypercomplete cosheaf in the first variable, satisfying constructibility and the factorization property in the sense of \cite[Definition 5.5.4.1]{HA}, \cite[Theorem 5.5.4.10]{HA}.
		\item It is a map of operads in the second variable.
	\end{itemize}
\end{thm}
\begin{proof}
	The only part that requires justification is the constructibility property. We know that the functor sends maps of the form \eqref{metric-balls-inclusion} to equivalences of $\infty$-categories, by \cref{factorization+homotopy-invariance} and \cref{homotopy-invariance-of-cons}, and that it is a hypercomplete cosheaf by \cref{hypercompleteness}. Note that this is sufficient to apply \cite[Proposition 5.5.1.14]{HA} and conclude that $\Sph(G;\cE)^\fact$ is constructible in the first variable, in the sense of \cite[Definition 5.5.4.1]{HA}. Indeed, the proof of the ``if'' direction of \cite[Proposition 5.5.1.14]{HA} works in the same way if one restricts to a full suboperad of $\Disk(M)^\otimes$ whose colors form a base of $M$, which is the case for $\MDisk(\C)^\otimes\subset \Disk(\C)^\otimes$.
\end{proof}	
	\begin{construction}\label{algebra-structure-on-Sph-Ran}We can now apply \cite[Theorem 5.5.4.10]{HA} to $\Sph(G;\cE)^\fact$ and obtain an $\E_M^\nun\times\E_1^\nun$-algebra with values in $\Prlo_\cE$, where $M=\C$, therefore a 2-dimensional real manifold. Its restriction to any disk $D\subset \C$ has the same value for different $D$'s (up to equivalence in $\Prro_\cE$), since an $\E_M^\nun$-algebra is in particular locally constant, and by \cite[Example 5.4.5.3]{HA}, this restriction is naturally an $\E_2^\nun$-algebra. This means that the stalk of $\Sph(G;\cE)^\fact$ in the first variable at any point $\{x\}\in \Ran(X^\an)$, which is just its value at any disk $D$ containing $x$, induces a map of operads $$\E_2^\nun\times \E_1^\nun\to \Prro_\cE.$$ Its underlying $\infty$-category is \begin{equation}\label{underlying-cat-stalk}\Cons(\hck_{\Ran(D)};\cE)\end{equation} for any (irrelevant) choice of a disk $D$ around $x$. Also the choice of $x$ is irrelevant, in the sense that different choices give noncanonically equivalent $\infty$-categories with equivalent algebra structures (this is easily seen by choosing a path between any two points in $X^\an$ a finite number of disks covering the image of this path).

We will refer to the operation implicit in the $\E_2^\nun$-component as \textit{fusion} and to the one implicit in the $\E_1^\nun$-component as \textit{convolution}.
\end{construction}

\subsection{Specialization to a point}\label{subsection-specialization}

Recall that we denote by $\hck_x$ or just $\hck$ the fiber of $\hck_{\{1\}}$ at any point $x\in \Ac(\C)$. Our goal in this subsection is to transfer the algebra structure established in \cref{algebra-structure-on-Sph-Ran} to the $\infty$-category $$\Cons(\hck^\an;\cE),$$ i.e.  $\Cons_{G_\cO^\an}(\Gr^\an;\cE)$ with the notation of \cref{rem-alg-top-cons}.

\begin{rem}\label{from-Sph-Ran-to-Sph}Let $j\geq j_1,N\in \nn$. Let $D\hookrightarrow\Ran(D)$ denote the closed embedding corresponding to the cardinality $1$ stratum. Recall that we have defined $\Gr_{D}^{(N)}=\Gr_{\{1\}}^\an\times_{(\Ac)^\an}D\simeq\Gr^\an\times D$, $G_{\cO,D}^{(j)}=G_{\cO,\{1\}}^\an\times_{(\Ac)^\an}D\simeq G_\cO^\an\times D$ and $\hck_D=\hck_{\{1\}}^\an\times_{(\Ac)^\an} D\in \PCon$ (the equivalences follow from \cref{translational-invariance}).

First of all, notice that for any choice of a disk $D$ in $(\Ac)^\an$ there is an equivalence of $\infty$-categories 
\begin{equation}\label{Sph-x-SPh-D}\Cons_{G^\an_\cO}(\Gr^\an;\cE)\simeq \Cons_{G^\an_{\cO,D}}(\Gr^\an_D;\cE)\end{equation}
induced by pulling back along the projection $\Gr_D^\an\to \Gr^\an$ (the inverse functor is given by pulling back along the embedding $\Gr^\an\times\{x_0\}\to \Gr^\an\times D$ (for any choice of a point $x_0$ in $D$). This is true by \cref{homotopy-invariance-of-cons}. Therefore, to give an $\E_3$-algebra structure on $\Cons_{G_\cO^\an}(\Gr^\an;\cE)$ is the same as giving an $\E_3$-algebra structure on $\Cons_{G_{\cO,D}}(\Gr_D;\cE)$. By the very nature of this identification, the equivalence does not actually depend on the choice of $x$ and $D$.

Now, there is a pullback diagram in $
\PCon$ $$\begin{tikzcd}\hck_D\arrow[r,hook,"i"]\arrow[d]& \hck_{\Ran(D)}\arrow[d]\\
	D\arrow[r, hook]&\Ran(D)\end{tikzcd}$$
and an adjunction 

\begin{equation}\begin{tikzcd}
		{\Cons(\hck_{D};\cE)} & {\Cons(\hck_{\Ran(D)};\cE)}
		\arrow[""{name=0, anchor=center, inner sep=0}, "{i_*}"', shift right=2, from=1-1, to=1-2]
		\arrow[""{name=1, anchor=center, inner sep=0}, "{i^!}"', shift right=2, from=1-2, to=1-1]
		\arrow["\dashv"{anchor=center, rotate=-90}, draw=none, from=1, to=0]
	\end{tikzcd}
\end{equation}
where $i_*$ is fully faithful because $i$ is an equivariant closed embedding of a union of strata.
\end{rem}

Note that $i^!$ preserves constructible sheaves with respect to the given stratifications: this follows straightforwardly by taking limits and colimits from \cref{upper-shriek-constructibles} (which \textit{per se} only deals with the case of stratified topological spaces -- and not stacks or prestacks).

\begin{prop}[{Dual version of \cite[Proposition 2.2.1.9]{HA}}]\label{dualized-transfer}
Let $\cO^\otimes$ be an $\infty$-operad and $\cc^\otimes\to \cO^\otimes$ be a coCartesian fibration. Suppose given colocalization functors $R_X:\cc_X\to \cc_X$ for each $X\in \cO$, and denote by $$\begin{tikzcd}\cD_X\arrow[""{name=0, anchor=center, inner sep=0}, "{L_X}"', shift right=2, from=1-1, to=1-2]
		\arrow[""{name=1, anchor=center, inner sep=0}, "{R_X}"', shift right=2, from=1-2, to=1-1]
		\arrow["\dashv"{anchor=center, rotate=90}, draw=none, from=1, to=0]&\cc_X\end{tikzcd}$$ the resulting adjunctions. Suppose that the condition in \cite[Definition 2.2.1.6]{HA} is satisfied by the $R_X$'s. Then $\cD$ denote the full subcategory of $\cc$ spanned by all objects of the form $L_X(C)$ for some $X\in \cO, C\in \cc_X$. Let $\cD^\otimes$ be the full subcategory of $\cc^\otimes$ defined in \cite[Before Proposition 2.2.1.9]{HA}. Then $\cD^\otimes$ inherits an $\cO^\otimes$-monoidal structure for which $R$ upgrades to an $\cO^\otimes$-monoidal functor.
\end{prop}
\begin{proof}Let $\cc^{\vee, \otimes}$ be the dual $\cO$-monoidal structure on $\cc^\op$. The given adjunctions induce adjunctions $$\begin{tikzcd}(\cD_X)^\op\arrow[""{name=0, anchor=center, inner sep=0}, "{L_X^\op}"', shift right=2, from=1-1, to=1-2]
		\arrow[""{name=1, anchor=center, inner sep=0}, "{R_X^\op}"', shift right=2, from=1-2, to=1-1]
		\arrow["\dashv"{anchor=center, rotate=-90}, draw=none, from=1, to=0]&(\cc_X)^\op\end{tikzcd}$$
		for each $X\in \cO$. We are thus in the situation of \cite[Proposition 2.2.1.9]{HA}, which yields an $\cO^\otimes$-monoidal structure on a subcategory $\cA^\otimes$ of $\cc^{\vee,\otimes}$ whose underlying category is $\cD^\op$. By passing to dual $\cO$-monoidal structures again, we obtain a full subcategory $\cD^\otimes=\cA^{\vee, \otimes}$ and an adjunction 

		$$\begin{tikzcd}\cD^\otimes\arrow[""{name=0, anchor=center, inner sep=0}, "{L}"', shift right=2, from=1-1, to=1-2]
		\arrow[""{name=1, anchor=center, inner sep=0}, "{R}"', shift right=2, from=1-2, to=1-1]
		\arrow["\dashv"{anchor=center, rotate=90}, draw=none, from=1, to=0]&\cc^\otimes\end{tikzcd}$$
		over $\cO^\otimes$ and such that $L$ is $\cO^\otimes$-monoidal.
\end{proof}

\begin{thm}\label{transfer-theorem}
	There is a canonical $\E_2^\nun\times\E_1^\nun$-algebra structure on $\Cons(\Hck_D;\cE)$ such that the functor $i^!$ upgrades to a lax-monoidal functor.
\end{thm}
\begin{proof}The idea is to transfer the $\E_2^\nun\times \E_1^\nun$-algebra structure to $\Cons_{G_{\cO,D}}(\Gr_D;\cE)$ by applying \cref{dualized-transfer} with $\cO^\otimes=\E_2^\nun\times \E_1^\nun,\cc^\otimes=\Cons(\hck_{\Ran(D)};\cE)$ and $L_X=i^!$ (note that our operad has only one colour $X$).

Recall thus from \cref{algebra-structure-on-Sph-Ran} that there are two compatible operations on $\Cons(\hck_{\Ran(D)};\cE)$, which we call $\star$ (convolution, the one parametrized by the $\E_1^\nun$-variable) and $\odot$ (fusion, the one parametrized by the $\E_2^\nun$-variable). In order to apply Lurie's result, we need to verify that for any $$\cA,\cA',\cB\in \Cons(\hck_{\Ran(D)}; \cE),$$ and a morphism $$f:\cA\to \cA'$$ such that $i^!f$ is an equivalence in $\Cons(\hck_D;\cE)$, the natural maps $$i^!(\cA\star\cB)\to i^!(\cA'\star \cB)$$ and $$i^!(\cA\odot \cB)\to i^!(\cA\odot\cB)$$ are equivalences.

First of all, by the description in \cref{algebra-structure-on-Sph-Ran}, we can fix $I_1,I_2$ and assume that $\cA,\cA'\in \Cons(\hck_{D^{I_1}};\cE), \cB\in \Cons(\hck_{D^{I_2}};\cE)$ (actually, we could even assume $I_1=I_2$, but it is instructive to see what happens in the general case).

We need to prove two things: \begin{itemize}
	\item[$\star$] For the convolution case, we need to prove that $$i^!(m_{I_1,I_2*}p_{I_1,I_2}^*(\cA\boxtimes \cB))\to i^!(m_{I_1,I_2*}p_{I_1,I_2}^*(\cA'\boxtimes \cB))$$ is an equivalence, where the notations are as in the following diagram: 
	
	\[\begin{tikzcd}
		{\hck_D\times \hck_D} & {\hck_{D,D}} & {\hck_{D^2}} & {\hck_D} \\
		{\hck_{D^{I_1}}\times \hck_{D^{I_2}}} & {\hck_{D^{I_1},D^{I_2}}} & {\hck_{D^{I_2\sqcup I_2}}}
		\arrow["{j'}"', from=1-1, to=2-1]
		\arrow["{p_{I_1,I_2}}"', from=2-2, to=2-1]
		\arrow["{\widetilde j}"', from=1-2, to=2-2]
		\arrow["p"', from=1-2, to=1-1]
		\arrow["m", from=1-2, to=1-3]
		\arrow["m_{I_1,I_2}", from=2-2, to=2-3]
		\arrow["j"', from=1-3, to=2-3]
		\arrow["d"', from=1-4, to=1-3]
		\arrow["i", from=1-4, to=2-3]
	\end{tikzcd}\]
	
	Here $i$ stays for the map $i$ read at the $I_1\sqcup I_2$-level, and $d$ for the map pulled back from the diagonal of $D^2$. Note that both squares and the triangle commute, the second square is a pullback and the map $m_{I_1,I_2}$ is ind-proper. Moreover, we have equivalences $m_*\simeq m_\vdash, m_{I_1,I_2*}\simeq m_{I_1, I_2\vdash}$. To see this, it suffices to see that $m_*$ and $m_{I_1,I_2*}$ preserve constructible sheaves. By ind-properness of $m,m_{I_1,I_2}$ and proper base change for (plain, not necessarily constructible) sheaves, it suffices to check this after pullback to each stratum of $X^{I_1\sqcup I_2}.$ There, the map can be realized as a product of copies of the multiplication map $m:\hck_{x,2}\to \hck_x$ for some point $x$. This latter map descends from a $G_\cO$-equivariant map $\Conv_{x,2}\to \Gr_x$; therefore, pushforward along it preserves equivariant sheaves which are constructible with respect to \textit{some} stratification. But equivariant constructible sheaves over $\Gr_x$ with respect to \textit{some} stratification are automatically constructible with respect with to \textit{the} stratification by Schubert cells by \cref{equivariant-constructibles-are-constructible-Schubert}.

	Therefore, we can apply proper base change and conclude that 
	\begin{gather*}i^!(m_{I_1,I_2*}p_{I_1,I_2}^*(\cA\boxtimes \cB))\simeq d^!j^!m_{I_1,I_2*}p_{I_1,I_2}^*(\cA\boxtimes \cB)\\ \simeq d^!m_*\widetilde j^!p_{I_1,I_2}^*(\cA\boxtimes \cB)\simeq d^!m_*p^*j'^!(\cA\boxtimes \cB).\end{gather*} where the last equivalence follows by functoriality of $(-)^!$ and the fact that, $p$ being pro-smooth, $p^*$ can be presented as a twist of $p^!$ (more precisely, a limit of twists of the $p^{(j)!}$'s, with the notation of the proof of \cref{horizontal-are-submersions}).
	
	Note now that $j'$ corresponds to the map $i\times i$ read at the $(I_1,I_2)$-level, and therefore the last expression equals $$d^!m_*p^*(i^!\cA\boxtimes i^!\cB).$$ If we read this construction functorially in $\cA$, we conclude that the map $i^!(\cA\odot\cB)\to i^!(\cA\odot\cB)$ induced by $f$ is an equivalence.
	
	\item[$\odot$]For the fusion product, the condition involves the diagram
	
	\begin{equation}\label{fusion-comparison}\begin{tikzcd}
			{\hck_{D_1}\times\hck_{D_2}} & {\hck_{D^2}} & {\hck_D} \\
			 {\hck_{{D_1}^{I_1}}\times\hck_{{D_2}^{I_2}}} & {\hck_{D^{I_1\sqcup I_2}}}
			\arrow["u_{1,1}", hook, from=1-1, to=1-2]
			\arrow["d"', hook', from=1-3, to=1-2]
			\arrow["{j'}"', hook, from=1-1, to=2-1]
			\arrow["{d_{I_1\sqcup I_2}}", hook, from=1-3, to=2-2]
			\arrow["j", hook, from=1-2, to=2-2]
			\arrow["{u_{I_1,I_2}}", hook, from=2-1, to=2-2]
	\end{tikzcd}\end{equation}
	Note first of all that $d_{I_1\sqcup I_2}=dj$ corresponds to the map $i$ read at the $I_1\sqcup I_2$-level.
	Up to identifying the first column with the second one, $\cA\odot\cB$ corresponds to $(u_{I_1,I_2})_\vdash(\cA\boxtimes\cB)$ and
	$$i^!(\cA\odot\cB)\simeq d^!j^!(u_{I_1,I_2})_\vdash(\cA\boxtimes\cB).$$
	Therefore, it suffices to show that the base-change map \begin{equation}\label{base-change-u-j}j^!(u_{I_1,I_2})_\vdash\to u_\vdash j'^!\end{equation} is an equivalence. By passing to left adjoints, one is left to check that $$(u_{I_1,I_2})^*j_!\to j'_!u^*$$ is an equivalence, which follows from proper base change of sheaves.

\end{itemize}
\end{proof}

\begin{note}In a previous version of the present paper, $i^*$ was used in \cref{transfer-theorem} in place of $i^!$. We thank Katsuyuki Bando for noticing a mistake in the original proof, which ultimately led us to replace $i^*$ with $i^!$.\end{note}

\begin{notation}\label{Sph_x}We denote the algebra structure inherited by $\Cons(\hck_x^\an;\cE)$ by means of \eqref{Sph-x-SPh-D} as $$\Sph(G)_x^\otimes\in \Alg_{\E_2}^\nun(\Alg_{\E_1}^\nun(\Prro_\cE)).$$\end{notation}

\subsection{Main result and t-exactness}
By convenience, given three $\infty$-operads $\cO,\cO',\cO''$, we will call ``bilinear maps'' those maps $\cO\times\cO'\to \cO''$ which are maps of operads separately in each variable.
\begin{prop}\label{E2-unital}Let $x$ be any point in $X^\an$. The bilinear map $\Sph(G;\cE)_x^\otimes:\E_2^\nun\times \E_1^\nun\to \Prro_\cE$ extends to a bilinear map $\E_2\times\E_1^\nun\to \Prro_\cE$, which we again denote by $\Sph(G;\cE)_x^\otimes$.\end{prop}
\begin{proof}We can apply \cite[Theorem 5.4.4.5]{HA}: that is, it suffices to exhibit a quasi-unit for any $$\Sph(G;\cE)_x^\otimes (-,\langle k\rangle)$$ functorial in $\langle k\rangle\in \E_1^\nun$. Consider the map (natural in $k$) $\Spec \C\to \hck_{x,k}$ represented by the sequence $(\cT,\dots, \cT,\id|_{X\setminus \{x\}}, \dots, \id_{X\setminus \{x\}})\in \hck_{x,k}$. Note now that this induces a map $$e_k:*\to \hck^\an_{x,k}.$$
We want to check that the map $e_k$ is a quasi-unit in the sense of \cite[Definition 5.4.3.1.]{HA} for any $k$, functorially in $k$. But both maps $\hck_{x,k}^\an\to \hck_{x,k}^\an\times \hck_{x,k}^\an\to \hck_{x,k}^\an$ induced by $e_k$ (by targeting respectively the first or the second factor of the product) are the identity, since gluing with the trivial $G$-torsor along with its trivial trivialization does not change the original torsor. At the level of constructible sheaves, the unit is given by the pushforward of the constant sheaf $\mathbf 1_\cE$ along $e_k$.
\end{proof}

\begin{prop}The bilinear map $\Sph(G;\cE)_x^\otimes:\E_2\times\E_1^\nun\to \Prro_R$ extends to a bilinear map $\Sph(G;\cE)_x^\otimes:\E_2\times\E_1\to \Prro_R$.\end{prop}
\begin{proof}Again, it suffices to exhibit a quasi-unit. Let us denote by $\mathbf 1$ the pushforward along the trivial section $t:*\to \hck_{x,1}^\an, t(*)=(\cT,\cT,\id|_{X\setminus x}),$ of the constant sheaf with value $\mathbf 1_\cE$.\\
The proof is given in \cite[Proposition IV.3.5]{Reich}. We just rewrite it in our notation. We drop the superscript $(-)^\an$ everywhere for simplicity. Let us assume that the entry in the variable $\E_2$ is $\langle 1\rangle$ by simplicity (the general case is just ``a direct power'' of this one). We denote by $\star$ the $\E_1$-product of equivariant constructible sheaves on $\Gr_x$ described by $\Sph(G;\cE)^\otimes_x(\langle1\rangle,-)$. For any $F\in \Cons_{G_\cO}(\Gr,R)$ we can compute the product via the convolution diagram $$
\begin{tikzcd}[column sep="0.3cm"]
	& {G_\cK\times \Gr} \arrow[ld, "p"'] \arrow[r, "q"] & G_\cK\times^{G_\cO}\Gr \arrow[rd, "m"] &     \\
	\Gr\times \Gr&                                                                 &                                                        & \Gr.
\end{tikzcd}$$

This diagram extends to $$
\begin{tikzcd}
	&               & {G_\cK\times \Gr} \arrow[ld, "p"'] \arrow[r, "q"] & {G_\cK\times^{G_\cO}\Gr} \arrow[rd, "m"] &     \\
	*\times \Gr \arrow[r, "t\times \id"] \arrow[rrru, "j", bend left] & \Gr\times \Gr &                                                                 &                                                        & \Gr,
\end{tikzcd}$$
where $j$ is the closed embedding $(\cF,\alpha)\mapsto (\cT,\id|_{X\setminus x},\cF,\alpha)$ whose image is canonically identified with $\Gr$. Let $F\in \Cons_{G_\cO}(\Gr;\cE)$. We want to prove that $\mathbf 1\widetilde\boxtimes F\simeq j_*(\mathbf 1_\cE\boxtimes F)$, i.e. that
$$q^* j_*(\mathbf 1_\cE\boxtimes F)\simeq p^*(t\times \id)_*(\mathbf 1_\cE\boxtimes F).$$
Note that because of the consideration about the image of $j$ the support of both sides lies in $G_\cO\times \Gr\subset G_\cK\times \Gr$, and this yields a restricted diagram
$$
\begin{tikzcd}
	& G_\cO\times \Gr \arrow[r, "\widetilde q"] \arrow[ld, "\widetilde p"'] & \Gr \arrow[rd, "\widetilde m"] &     \\
	\Gr \arrow[rru, "\sim", "j"'] &                                           &                     & \Gr.
\end{tikzcd}$$ This proves the claim. By applying $m_*$ we obtain $$\mathbf 1\star F\simeq m_*(j_*(\mathbf 1_\cE \boxtimes F))=\mathbf 1_\cE\boxtimes F=F$$ since $mj=\id$.
\end{proof}

Thanks to these results, our functor $\Sph(G;\cE)_x^\otimes$ from \cref{Sph_x} is finally promoted to a bilinear map $\E_2\times \E_1\to \Prro_R$. By the Additivity Theorem (\cite[Theorem 5.1.2.2]{HA}), this is the same as an $\mathbb E_3$-algebra object in $\Prro_R$.

Note also that, by the convolution presentation of the monoidal law, the functor $\Sph(G,\cE)_x^\otimes$ takes values in $\Prlro_\cE$: indeed, by the same argument as in the first part of the proof of \cref{transfer-theorem} and by properness, the functor $\overline m_\vdash$ coincides both with $m_*$ and with $m_!$, and hence it is a right and left adjoint. For convenience, we will say that our algebra takes values in $\Prlo_\cE$ from now on.

Summing up:
\begin{thm}[Main theorem]\label{final-theorem}Let $G$ be a complex reductive group and $\cE$ be a presentable stable symmetric monoidal $\infty$-category. There is an object $\Sph(G;\cE)^\otimes\in \Alg_{\E_3}(\Prlo_\cE)$ having as underlying object the topological spherical Hecke category $$\Sph(G;\cE)^\topp=\Cons_{G_\cO^\an}(\Gr_G^\an;\cE)$$ (see \cref{spherical-categories}).\end{thm}

Note that, when $R$ is a commutative ring, either discrete, prodiscrete or $\ell$-adic (i.e. an algebraic extension of $\Q_\ell$), this specializes to \cref{final-theorem-in-introduction} for $\cE=\Mod_R^\cont$ in the sense of \cref{profinite-etc}.

\begin{cor}[Small spherical Hecke category]\label{corollary-small-subcategory}In the same setting as \cref{final-theorem}, there is an induced $\E_3$-monoidal structure in $\cat^\times_{\infty,\cE^\omega}$ on $$\Cons_{G_\cO^\an}(\Gr_G^\an; \cE^\omega).$$ 
	
	Let $R$ be a commutative ring, noetherian and of finite global dimension, and $\cE=\Mod_R$. Then, on objects, the restriction of this product to equivariant perverse sheaves coincides with the classical (commutative) convolution product of perverse sheaves \cite{MV}, up to the perverse truncation of the derived tensor product appearing in the definition of the latter (cf. also \cref{t-exactness-of-convolution}).\end{cor}
\begin{proof}
The (not full) inclusion $\Prro_{\cE}\to \widehat{\cat}_{\infty,\cE}^\times$ (the $\infty$-category of large $\cE$-linear $\infty$-categories) is lax monoidal, i.e. it is a map of operads. Therefore, $\Cons_{G_\cO^\an}(\Gr^\an;\cE)$ has an induced $\E_3$-algebra structure in $\widehat{\cat}^\times_{\infty,\cE}$. By using the convolution formula, one sees that the convolution product restricts to $\Cons_{G_\cO^\an}(\Gr^\an,\cE^\omega)$: again, this follows from the fact that $\overline p^*$ preserves $\cE^\omega$-valued sheaves (since it is induced by restriction along functors between categories of exit paths), and that the same is true for $\overline m_*$\footnote{Indeed, let $f:X\to Y$ be a map between topological spaces, $\cE$ an $\infty$-category. If $\cF\in \Shv(X;\cE)$ and $U$ is open in $Y$, then $$(f_*\cF)(U)=\cF(f^{-1}(U))$$ which belongs to $\cE^\omega$ if $\cF$ takes values in $\cE^\omega$.}. Since the inclusion\footnote{Here we mean the functor sending a small category to itself.} $\cat_{\infty,\cE^\omega}^\times\subset \widehat \cat_{\infty,\cE^\omega}^\times$ is strong symmetric monoidal, we obtain the first part of the statement.
The claim regarding perverse sheaves follows from what observed in \cref{quotient-convolution-diagram}.
\end{proof}

\begin{cor}[Renormalization]\label{corollary-Ind-completion}Let us assume that we are in the same setting of \cref{corollary-small-subcategory}, and that $\cE$ is compactly generated. Then there is an induced $\E_3$-monoidal structure on $$\Sph(G;\cE)^\textup{ren}=\Ind(\Cons^\fd_{G_\cO}(\Gr;\cE))$$ (see \cref{renormalized-spherical-category}) as an object of $\Prlo_\cE$.\end{cor}
\begin{proof}
Since $\cE$ is compactly generated, we have that $\Ind(\cE^\omega)\simeq \cE$. Recall first of all that the functor $$\Ind:\Catex\to \Prl_\st$$ is symmetric monoidal: this follows from \cite[Proposition 4.8.1.8]{HA}) with $\mathcal K=\varnothing$ and $$\mathcal K'=\{\kappa\textup{-filtered simplicial sets, for some regular cardinal }\kappa\}$$ and the fact that colimits are generated by filtered colimits and finite colimits. Therefore, there is an induced symmetric monoidal functor $$\Mod_{\cE^\omega}(\Catex)\to\Mod_\cE(\Prl_\st).$$ This functor factors through $\Prl_\cE$, because $\Ind$ of an exact functor $F$ is strongly continuous, i.e. its right adjoint $G$ preserves colimits, and hence by \cite[Remark 6.6]{DAG-VII} $G$ is automatically $\cE$-linear.
\end{proof}

Again, the latter results specialize to \cref{corollary-small-subcategory-introduction} and \cref{corollary-Ind-completion-introduction} in the ring case.

\begin{rem}\label{hearts-and-Ek}Let $\cc^\otimes$ be an $\E_k$-algebra in $\Cat^\ex$, the $\infty$-category of stable $\infty$-categories and exact functors between them. Suppose given a t-structure on $\cc$ which is compatible with the algebra structure, in the sense of \cite[Example 2.2.1.3]{HA} (intuitively, the subcategory $\cc_{\geq 0}$ should be closed under tensor). Then by \cite[Proposition 2.2.1.8, Proposition 2.2.1.9]{HA} the heart $\cc^\heartsuit$ of the t-structure canonically inherits an $\E_k$-algebra structure (the proof goes along the same lines of \cite[Example 2.2.1.10]{HA}, which deals with the case $\E_\infty$).

Note that this ``induced $\E_k$-structure'' procedure is functorial: given a stable-exact and t-left exact (in the sense of \cite[Definition 1.3.3.1]{HA}) $\E_k$-monoidal functor $\cc^\otimes\to \cD^\otimes$, one obtains an additive functor $\cc^\heartsuit\to\cD^\heartsuit$ between abelian categories, which can be viewed as the composition of the $\E_k$-monoidal functors $$\cc^\heartsuit\hookrightarrow\cc_{\leq 0}\xrightarrow{F|_{\cc_{\leq 0}}} \cD_{\leq 0}\xrightarrow{\tau_{\leq 0}}\cD^\heartsuit$$ (notice that $F$ restricts to the coconnective parts by left-t-exactness).
\end{rem}

\begin{rem}\label{t-exactness-of-convolution}Let $R$ be a commutative discrete ring. The small spherical Hecke category carries a canonical t-structure inherited from the perverse t-structure on bounded categories of (finite-dimensional) constructible sheaves. Indeed, the Bernstein-Lunts presentation $$\Sph(G;R)^{\locc}\simeq \lim_{n}\cD_\textup{c}^\fd(G_\cO^{\times n}\times\Gr;R)$$ of \cref{equivariant-varying-strat} establishes a canonical t-structure on the limit by \cite[Section 2.5]{Bernstein-Lunts}.

The convolution product on $\Sph(G;R)^{\textup{loc.c}}$ is perverse left t-exact, and t-exact when $R$ is a field. This follows from the considerations recalled in \cref{definition-of-twisted-product}, as we explain below.

\begin{itemize}\item The derived tensor product is always left t-exact for the perverse t-structure (this is easy to deduce since the derived tensor product is cohomology-left exact. Moreover, with coefficients in a field, it is actually perverse t-exact, see \cite[Lemma 4.1]{MV}).

\item Pullback along maps in $\subm\subset\Mor(\Con)$, equidimensional of relative dimension $d$, is perverse t-exact up to a shift by $d$ equal to the relative dimension of the map: this is proven in the exact same way as for instance \cite[4.2.4]{BBDG} for schemes, but in the topological setting. In brief, let $f:(X,s)\to (Y,t)$ be a map of stratified topological spaces arising as analytification of stratified complex schemes. One uses cohomological exactness of $f^*$ to prove left t-exactness, and Poincar\'e duality to prove that $f^!\simeq f^*[d]$. But $f^!$ is right t-exact, hence the statement. See also \cite[Proposition 10.3.15]{Kashiwara-Schapira}.

\item Let $$\begin{tikzcd}X\arrow[r,"f'"]\arrow[d,"\pi'"]&Y\arrow[d,"\pi"]\\ \cX\arrow[r,"f"]& \cY\end{tikzcd}$$ be a commutative diagram with $\cX, \cY\in \ConStk$, $\pi, \pi'\in \widetilde\subm, f\in \subm'$.  Then $f^*[d]$ is perverse t-exact, where $d=\dim f'+\dim \pi-\dim \pi'$: see \cite[§4]{Laszlo-Olsson-perverse} for definitions and results in the context of algebraic geometry, which translate \textit{verbatim} to the complex-topological context, as in the previous point.


\item The map $\overline p^{(N,j)}: \hck_2^{(N,2j)}\to \hck^{(N,j)}\times\hck^{(N,j)}$ is a smooth map of stacks of relative dimension $0$ (but it is not an isomorphism: indeeed, it is not representable).
\item The map $m$ in the convolution diagram is stratified semi-small, \cite[Lemma 6.4]{Baumann-Riche}.
	
\item Pushforward along proper stratified semi-small maps is t-exact (see the proof of \cite[Proposition 6.1]{Baumann-Riche}).

\end{itemize}


Thus, one can apply \cref{hearts-and-Ek} to $\Sph(G;R)^\locc$ with $k=3$ and recover the convolution product of perverse sheaves and its commutativity constraint, as mentioned in \cref{heart-of-derived-Satake}. Indeed, although in the original formula the perverse truncation appears right after performing the derived tensor product (and not at the end of the process), there is no actual difference with our formula in this respect, since the rest of the procedure defining the convolution product is perverse t-exact.

\end{rem}


\begin{appendices}

\section{Recollections and complements in Geometric Langlands}\label{appendix-Geometric-Satake}
\subsection{The Satake category}
Let us resume from the definition of affine Grassmannian, recalled in \cref{affine-Grassmannian}.


\begin{recall}[{see \cite[Theorem 1.1.3]{Zhu}}]\label{perverse-and-stratification}There is a natural action of $G_\cO$ on $\Gr_G$ by left multiplication, whose orbits define an algebraic stratification of $\Gr_G$ over the poset $\mathbb X_\bullet(T)^+$ of dominant coweights of the Cartan group $T$ of $G$. When viewed from the point of view of the complex-analytic topology on $\Gr_G$, this stratification satisfies the so-called \textit{Whitney conditions} (for a proof, see \cite{MO-Gr-is-Whitney}). One can characterize the stratification as follows. The Cartan decomposition\footnote{The symbol $\bigsqcup$ should only be understood as a set-theoretical decomposition, not a topological or scheme-theoretic one.} $$G_\cK=\bigsqcup_{\mu\in \xt}G_\cO t^\mu G_\cO$$ induces a partition $$\Gr_G=\bigsqcup_{\mu\in \xt}G_\cO t^\mu.$$ For an element $g$ of $G_\cK$, the associated $\mu$ is denoted by $$\Inv(g)$$ and is the same for every $g'$ in the same right $G_\cO$-class, i.e. $\Inv(-)$ factors through $G_\cK\to \Gr_G$.

	If $G=\textup{GL}_n$, $\xt$ can be realized (noncanonically) as the set $$\{(\mu_1,\dots,\mu_n)\mid \mu_1\geq\dots\geq\mu_n\}$$ and via this identification $t^\mu$ is exactly the diagonal matrix $\textup{diag}(t^{-\mu_i})$.
	
	For $\mu\in \xt$, one defines $$\Gr_\mu=\{\Lambda\in\Gr\mid \Inv(\Lambda)=\mu\}.$$

	There is a natural filtration of $\Gr$ by finite-dimensional projective schemes $\Gr_{\leq \mu}={\bigcup_{\nu\leq \mu}}\Gr_\nu$. Moreover, the action $G_\cO\circlearrowright \Gr_G$ preserves $\Gr_{\leq\mu}$, and actually each $G_\cO\circlearrowright \Gr_{\leq \mu}$ factors through the quotient $G_\cO \twoheadrightarrow G_{\cO}^{(j)}= G(\C\taylor/t^{j}\C\taylor)$ for any $j$ larger than some $j_\mu$ (\cite[Lemma IV.1.4]{Reich}). This is actually the reason why the orbits form a stratification in the first place (\cite{MO-Gr-is-Whitney}). \end{recall}

\begin{defin}\label{truncated-GO}
	Let $j\geq 1$. We define $$G_{\cO}^{(j)}=G(\C\taylor/t^{j}\C\taylor).$$
\end{defin}

\begin{defin}Let $R$ be a ring. The category of $G_\cO$-equivariant perverse sheaves on $\Gr_G$ (or \textit{Satake category}) with values in $R$-modules is $$\Perv_{G_\cO}(\Gr;R):=\colim_{\mu\in \xt}\Perv_{G_{\cO}^{(j_\mu)}}(\Gr_{\leq \mu};R)$$ (see \cite[5.1 and A.1]{Zhu}).\end{defin}
The definition of each term is independent of $j_\mu$ because of \cite[Lemma A.1.4]{Zhu}.

\begin{recall}[{cf. \cite[Lemma 3.1.7]{Zhu}}]\label{affine-formal-completions}Let $X$ be a smooth complex curve, $R$ a complex ring, and $x\in X(R)$ an $R$-point. There is a well-defined formal completion of $\Gamma_x$, the graph of $x$ in $X_R$, which is a formal scheme whose (discrete) ring of functions is $\widehat \oo_{\Gamma_x}$. We consider the \textit{affine formal neighbourhood} of $x$, defined as the map
	$$\widetilde{(X_R)}_{\Gamma_x}=\Spec \widehat\oo_{\Gamma_x}\to X_R$$ coming from \cite[Lemma 3.1.7]{Zhu}, \cite[Proposition 2.12.6]{BD-Hitchin}.
	Note that the source of this map is always isomorphic, \'etale-locally on $R$ (and noncanonically, since the isomorphism depends on the map $x$), to $\Spec R\taylor$. In particular, each closed point $x$ of $X$ admits an affine formal neighbourhood $$\widetilde X_x\simeq\Spec \C\taylor\to X.$$
	In general, for a $R$ a commutative ring and $x\in X(R)$, consider the square 
	\[\begin{tikzcd}
		\mathring{(X_R)}_{\Gamma_x} & \widetilde{(X_R)}_{\Gamma_x} \\
		{X_R\setminus \Gamma_x} & {X_R}
		\arrow[from=1-1, to=2-1]
		\arrow[from=1-1, to=1-2]
		\arrow[from=1-2, to=2-2]
		\arrow[from=2-1, to=2-2]
	\end{tikzcd}\]
 where the upper left vertex is by definition the pullback of the span. This pullback is again an affine scheme, called the \textit{punctured affine formal neighbourhood} of $x$. If $R=\C$ and $x$ is a closed point of $X$, we obtain $$\mathring X_x\simeq \Spec \C\laurent\to X.$$
 \end{recall}

\begin{defin}Let $\mathbf{Bun}_G$ be the moduli stack of $G$-torsors over $\C$. If a scheme $Z$ over $\C$ is given, we define the relative version $$\mathbf{Bun}_G^Z:\CAlg\to \Grpd$$
	$$R\mapsto \{\mbox{$G$-torsors over }Z\times \Spec R\}=\Bun(Z_R).$$
\end{defin}
\begin{prop}\label{prop:Grloc}For any closed point $x$ of a smooth projective complex curve $X$, the functor $\Gr_G$ is equivalent to the following:
	\begin{equation}\Gr_G^{\tloc}:R\to \{\cF\in \Bun((\widetilde{X_R})_{x\times \Spec R}), \alpha:\cF|_{\mathring{(X_R)}_{x\times \Spec R}}\xrightarrow{\sim}\cT_{\mathring{(X_R)}_{\{x\}\times \Spec R}}\}.\end{equation}
	\end{prop}

\begin{proof}
	The proof is explained for instance in \cite[Proposition 1.3.6]{Zhu}.
\end{proof}

We will need the following version of the affine Grassmannian as well.
\begin{construction}\label{rem:Grglob}Let $G=\GL_n$. Define $\Gr_G^{\glob}$ as the fiber of the restriction map $\mathbf{Bun}_G^X\to \mathbf{Bun}_G^{X\setminus \{x\}}$ at the trivial $G$-torsor, i.e. as the functor $$R\to \{\cF\in \Bun(X_R),\alpha:\cF|_{X_R\setminus(\{x\}\times \Spec R)}\xrightarrow{\sim}\cT_{X_R\setminus(\{x\}\times \Spec R)}\}.$$ Indeed, in the diagram of groupoids
	$$\begin{tikzcd}\Gr_G^{\glob}(R) \arrow[r] \arrow[d] & \Bun(X_R) \arrow[d] \arrow[r] \arrow[d] & \Bun((\widetilde{X_R})_{\{x\}\times\Spec R}) \arrow[d]  \\
		\{\cT|_{X\setminus \{x\}}\} \arrow[r]          & \Bun(X_R\setminus (\{x\}\times \Spec R)) \arrow[r]     & \Bun((\mathring{X_R})_{\{x\}\times \Spec R})
	\end{tikzcd}$$
	the right-hand square is Cartesian by the Beauville--Laszlo Theorem \cite{BL}, more precisely in the form of \cite[Remark 2.3.7]{BD-Hitchin}. Since the left-hand square is Cartesian by definition, the outer square is Cartesian. Therefore, $\Gr_G^{\glob}(R)$ is isomorphic to the fiber of the restriction map $\Bun((\widetilde{X_R})_{\{x\}\times \Spec R})\to\Bun((\mathring{X_R})_{\{x\}\times \Spec R})$ at the trivial bundle, and this is exactly $\Gr^{\tloc}_G(R)$. For more details, and for the case of an arbitrary reductive $G$, see \cite[Theorem 1.4.2]{Zhu}.\end{construction}

\begin{rem}\label{filtration-and-stratification-of-Gr}Let $G$ be a complex reductive group, $X$ a smooth complex curve, $x\in X(\C)$. In the case $G=\GL_n$, one can filter $\Gr$ by $$\Gr^{(N)}=\{\cF\in\Bun(X), \alpha:\cF|_{X\setminus \{x\}}\triv\cT|_{X\setminus \{x\}}\mid \cO_X^n(-N)\subset \cF\subset \cO_X^n(N)\},$$ $N\in \nn$. This filtration is compatible with the stratification and the filtration appearing in \cref{perverse-and-stratification}: see e.g. \cite[§2.3]{Kamnitzer-Mutiah-Weekes}. In the case of a general $G$, a similar filtration is achieved by means of the choice of a faithful representation $G\to \GL_n$ for some $n$ (see \cite[Propositions 1.2.5, 1.2.6]{Zhu}), and a similar compatibility result holds (see again \cite[§2.3]{Kamnitzer-Mutiah-Weekes}).\end{rem}

\begin{construction}\label{definition-of-twisted-product}Let $R$ be a commutative discrete ring. We recall now the tensor structure given by {\bf convolution product} on $\cP\mathrm{erv}_{G_\cO}(\Gr_G;R)$. A more detailed account is given in \cite[Section 1, Section 5.1, 5.4]{Zhu}.	
	Consider the diagram \begin{equation}\label{eq:convdiag}\begin{tikzcd}
			& G_\cK\times \Gr_G \arrow[ld, "p"] \arrow[r, "q"] & G_\cK\times^{G_\cO} \Gr_G \arrow[rd, "m"'] &      \\
			\Gr_G\times \Gr_G &                                                          &                                                           & \Gr_G
		\end{tikzcd}
	\end{equation} where ${G_\cK}\times^{G_\cO}\Gr_G$ (or $\Conv_G$) is the stack quotient of the product ${G_\cK}\times \Gr_G$ with respect to the ``anti-diagonal'' left action of ${G_\cO}$ defined by $\gamma\cdot(g,[h])=(g\gamma,[\gamma^{-1} h])$. The map $p$ is the projection to the quotient on the first factor and the identity on the second one, the map $q$ is the projection to the quotient by the mentioned action of $G_\cO$.
	
	Note that the left multiplication action of ${G_\cO}$ on ${G_\cK}$ and on $\Gr_G$ induces a left action of ${G_\cO}\times {G_\cO}$ on $\Gr_G\times \Gr_G$. It also induces an action of ${G_\cO}$ on ${G_\cK}\times \Gr_G$ given by $(\mbox{left multiplication}, \id)$ which canonically projects to an action of $G_\cO$ on ${G_\cK}\times^{G_\cO} \Gr_G$. Note that $p,q$ and $m$ are equivariant with respect to these actions (more precisely, $p$ is $G_\cO\times G_\cO$-eqivariant, whereas $q$ and $m$ are $G_\cO$-equivariant).
	
	Now if $\mathcal A_1,\mathcal A_2$ are two $G_{\cO}$-equivariant perverse sheaves on $\Gr_G$, one can define a convolution product \begin{equation}\label{convprodbase}\mathcal A_1\star \mathcal A_2=m_*\widetilde{\mathcal A}\end{equation} where $m_*$ is the derived direct image functor, and $\widetilde{\mathcal A}$ is a perverse sheaf on ${G_\cK}\times^{{G_\cO}} \Gr_G$ which is equivariant with respect to the left action of ${G_\cO}$ and such that $q^*\widetilde{\mathcal A}=p^{*}(^\textup{p}\cH^0(\mathcal A_1\boxtimes \mathcal A_2))$. \footnote{The tensor product denotes the derived tensor product in the derived category. If the stalks of the two sheaves are flat, e.g. when the ring of coefficients is a field, the external tensor product is already perverse, see \cite[Lemma 4.1]{MV}. In general, one needs to consider the perverse truncation, as in the formula.} Note that such an $\widetilde{\mathcal A}$ exists because $q$ is the projection to the quotient and $\cA_2$ is $G_\cO$-equivariant, and one can prove that $\widetilde \cA$ is again perverse.
	
	The functor $m_*$ carries perverse sheaves to perverse sheaves: indeed, it can be proven that $m$ is ind-proper, i.e. it can be represented by a filtered colimit of proper semi-small maps of schemes. By \cite[Lemma III.7.5]{KW}, and the definition of $\mathcal P\mathrm{erv}_{{G_\cO}}(\Gr_G;R)$ as a filtered colimit, this ensures that $m_*$ carries perverse sheaves to perverse sheaves.
\end{construction}

It is important to stress that, at every step, we are implicitly assuming our sheaves to be supported on some $\Gr_{\leq\mu}$.

Observations similar to \cref{rem:Grglob} prove the following:

\begin{prop}\label{moduli-interpretation} We have the following equivalences of groupoids:
	
	$${G_\cO}(R)\simeq\mathrm{Aut}((\widetilde{X_R})_{x\times \Spec R},\cT)$$
	
	$${G_\cK}(R)\simeq\{\cF\in \Bun(X_R),\alpha:\cF|_{(X\setminus \{x\})\times \Spec R}\simeq \cT|_{(X\setminus \{x\})\times \Spec R},$$$$\mu:\cF|_{(\widetilde{X_R})_{x\times \Spec R}}\simeq \cT|_{(\widetilde{X_R})_{x\times \Spec R}}\}$$
	
	$$(G_{\cK}\times^{G_\cO}\Gr)(R)\simeq\{\cF\in \Bun(X_R),\alpha:\cF|_{(X\setminus\{x\})\times \Spec R}\simeq \cT|_{(X\setminus \{x\})\times \Spec R},$$$$\cG\in \Bun(X_R),\eta:\cF|_{(X\setminus\{x\})\times\Spec R}\simeq \cG_{(X\setminus\{x\})\times \Spec R}\}$$
\end{prop}

\subsection{The convolution product via quotient stacks}\label{appendix-convolution}

\begin{defin}\label{Hecke-stack}
	We define the complex stack $$G_\cO\backslash \Gr$$ as the fpqc quotient stack of $\Gr$ by the left action of $G_\cO$. We also define $$G_\cO\backslash (G_\cK\times^{G_\cO}\Gr)$$ as the fpqc quotient stack of $G_\cK\times^{G_\cO}\Gr$ by the left action of $G_\cO$ on the first factor $G_\cK$. 
\end{defin}

\begin{prop}
	There is an equivalence of stacks between $G_\cO\backslash \Gr$ and the functor $$\Aff_\C^\op\to \Grpd$$
	$$\{\cF_0,\cF_1\in\Bun(\Spec R\taylor),\eta:\cF_0|_{\Spec R\laurent}\xrightarrow{\sim}\cF_1|_{\Spec R\laurent}\}.$$
\end{prop}

\begin{proof}
	There is a map $\pi$ from $\Gr$ to the moduli space appearing in the statement, described as follows. Let $\cT$ be the trivial $G$-bundle on $\Spec R\taylor$; then $\pi(\cF, \alpha):=(\cT, \cF, \alpha^{-1})$. Let us show that this map is essentially surjective. Since any $G$-torsor on $\Spec R\taylor$ is trivializable locally in $\Spec R$, for any triple $(\cF_0, \cF_1,\eta)$ as in the statement one can find, locally in $\Spec R$, $\mu:\cF_0\xrightarrow{\sim}\cT$ and consequently $\alpha:\cF_1|_{\Spec R\laurent}\triv \cT|_{\Spec R\laurent}$ such that $\eta=\alpha^{-1}\circ \mu|_{\oD}$. The fact that this is local in $\Spec R$ is not a problem, since we are considering the quotient stack on the left-hand-side. Thus, the triple $(\cT, \cF_1,\alpha^{-1})$ is isomorphic to $(\cF_0,\cF_1,\eta)$ by means of the isomorphism $(\mu, \id):(\cF_0,\cF_1,\eta)\to (\cT,\cF_1, \alpha^{-1})$.
	
	To conclude the proof, it suffices to prove that the fiber of $\pi$ at each $R$-point of the right-hand side is $G_\cO\times_\C \Spec R$. But the fiber at $(\cF_0,\cF_1,\eta)$ is the set of those $(\alpha,\mu)$, $\alpha:\cF_1|_{\Spec R\laurent}\xrightarrow{\sim}\cT|_{\Spec R\laurent},\mu:\cF_0\xrightarrow{\sim}\cT$ such that $\alpha^{-1}\circ \mu|_{\Spec R\laurent}=\eta$, which in turn amounts to the set (in particular a $0$-truncated groupoid) of $\mu$'s since $\alpha$ is completely determined by $\eta$ and $\mu$. But this is $G_\cO$, since any two trivializations on $\Spec R\taylor$ are connected by a unique automorphism of $\cT$ on $\Spec R\taylor$.
\end{proof}

In a similar way, one can prove that $$G_\cO\backslash (G_\cK\times^{G_\cO}\Gr)\simeq \{\cF_0,\cF_1,\cF_2\in \Bun(\hat D), \eta_1:\cF_0|_{\mathring D}\simeq \cF_1|_{\mathring D}, \eta_2:\cF_0|_{\mathring D}\simeq \cF_2|_{\mathring D}\}.$$

\begin{rem}\label{quotient-convolution-diagram}Consider the diagram of stacks
	\begin{equation}\label{stacky-convolution}\begin{tikzcd}
			& {(G_\cO\times G_\cO)\backslash(G_\cK\times \Gr)} & {G_\cO\backslash (G_\cK\times^{G_\cO}\Gr)} \\
			{G_\cO\backslash\Gr\times G_\cO\backslash \Gr} && {} & G_\cO\backslash\Gr
			\arrow["r", dotted, from=1-3, to=2-1]
			\arrow["\overline p"',from=1-2, to=2-1]
			\arrow["\sim", from=1-2, to=1-3]
			\arrow["\overline m", from=1-3, to=2-4]
	\end{tikzcd}\end{equation}
	where the action of the second copy of $G_\cO$ on $G_\cK\times G_\cO$ is the antidiagonal one described in \eqref{eq:convdiag}, all other actions are induced by the left multiplication action of $G_\cO$ on $G_\cK$. Then:\begin{itemize}
		\item the horizontal map is an equivalence (and therefore a map $r$ as above is defined such that the diagram commutes);
		
		\item a $G_\cO$-equivariant perverse sheaf on $\Gr$ is the same thing as a perverse sheaf on $G_\cO\backslash \Gr$, in the following sense. Note that the latter is an ind-stack of finite type, hence we can define $$\Perv(G_\cO\backslash \Gr):=\colim_\mu\Perv\left(G_\cO^{(j_\mu)}\backslash \Gr_{\leq \mu}\right),$$ where the expression in the colimit is meant in the sense of \cite[Section 4]{Laszlo-Olsson-perverse}. Finally, the claimed equivalence follows from \cite[Remark 5.5]{Laszlo-Olsson-perverse}. Similar considerations can be done for the other vertices of the diagram \eqref{stacky-convolution};
		
		\item under the identification at the previous point, the convolution product is equivalently described (up to the perverse truncations of the derived tensor product) by \begin{equation}\label{perverse-convolution-formula}\cA_1\star \cA_2=\overline m_*(r^*(\cA_1\boxtimes \cA_2)).\end{equation} Note that pullbacks and pushforwards are defined, at the level of the terms in the colimit at the previous point, as pullback and pushforwards of elements of the derived category.
		
		In particular, we can say that the diagram of stacks
		\begin{equation}\label{hecke-conv-appendix}\begin{tikzcd}
			& {G_\cO\backslash (G_\cK\times^{G_\cO}\Gr)} \\
			{G_\cO\backslash\Gr\times G_\cO\backslash\Gr} && \hck
			\arrow["r", from=1-2, to=2-1]
			\arrow["{\overline m}"', from=1-2, to=2-3]
		\end{tikzcd}\end{equation} ``correctly models'' the convolution product of $G_\cO$-equivariant perverse sheaves over the affine Grassmannian (up to the perverse truncation of the derived tensor product appearing in the original definition).
	\end{itemize} 
\end{rem}
\begin{lem}The map $r$ can be described as $$r(\cF_0, \cF_1, \cF_2,\eta_1,\eta_2)=((\cF_0,\cF_1,\eta_1),(\cF_1,\cF_2,\eta_2)).$$ 
\end{lem}
\begin{proof}
	A priori, $r$ works as follows: choose $$\mu_0:\cF_0\triv\cT$$
	$$\mu_1:\cF_1\triv\cT$$
	$$\alpha_1:\cF_1|_{\Spec\C\laurent}\triv\cT|_{\Spec\C\laurent}$$
	$$\alpha_2:\cF_2|_{\Spec\C\laurent}\triv\cT|_{\Spec\C\laurent}$$
	such that $\alpha_1^{-1}\mu_0|_{\Spec\C\laurent}=\eta_1,\alpha_2^{-1}\circ \mu_1|_{\Spec\C\laurent}\simeq \eta_2$.
	
	Then by definition $$r(\cF_0,\cF_1,\cF_2,\eta_1,\eta_2)=((\cT,\cF_1,\alpha_1^{-1}),(\cT,\cF_2,\alpha_2^{-1})).$$ But now, there are squares of isomorphisms on $\Spec\C\laurent$
	$$\begin{tikzcd}\cT\arrow[r,"\alpha_1^{-1}"]&\cF_1\\
		\cF_0\arrow[u, "\mu_0|_{\Spec\C\laurent}"]\arrow[r,"\eta_1"]& \cF_1\arrow[u,"\id"] 
	\end{tikzcd}$$$$
	\begin{tikzcd}\cT\arrow[r,"\alpha_2^{-1}"]&\cF_2\\
		\cF_1\arrow[u, "\mu_1|_{\Spec\C\laurent}"]\arrow[r,"\eta_2"]& \cF_2\arrow[u,"\id"]\end{tikzcd}$$ where the vertical maps are induced by the isomorphisms on $\Spec\C\taylor$, and we conclude.
\end{proof}
\begin{rem}Note that this map does not coincide with the map induced by the isomorphism $(\id,m):G_\cK\times^{G_\cO}\Gr\to \Gr\times\Gr$ (cf. \cite[(1.2.14)]{Zhu}), which is instead described by $$(\cF_0, \cF_1, \cF_2,\eta_1,\eta_2)\mapsto ((\cF_0, \cF_1,\eta_1),(\cF_0,\cF_2,\eta_1\circ\eta_2)).$$ Indeed, to prove that this is the same map we would need to build a commuting square
	
	$$
	\begin{tikzcd}\cT\arrow[r,"\alpha_2^{-1}"]&\cF_2\\
		\cF_0\arrow[u]\arrow[r,"\eta_2\eta_1"]& \cF_2\arrow[u,"\id"]\end{tikzcd}$$
	
	but here there is no reason why $\alpha_2\eta_2\eta_1$ would extend to the complete disk: we know that there exist $\widetilde\mu, \widetilde\alpha$ such that $\eta_2\eta_1=\widetilde\alpha^{-1}\widetilde\mu|_{\mathring D}$, which in general lies in $G_\cK$ and not in $G_\cO$.\end{rem}

\begin{rem}\label{appendix-truncated-convdiag}
Let $N\in \nn, j\geq j_N$. Let $G_{\cK}^{(N,j)}=G_{\cK}\times^{G_\cO}G_{\cO}^{(j)}$. The diagram \eqref{hecke-conv-appendix} admits a truncated version \begin{equation}\label{hecke-conv-appendix-truncated}
		\begin{tikzcd}
			& {\hck_2^{(N,2j)}} \\
		{\hck^{(N,j)}\times \hck^{(N,j)}} && \hck^{(2N,2j)}
		\arrow["r", from=1-2, to=2-1]
		\arrow["{\overline m}"', from=1-2, to=2-3]
	\end{tikzcd}
\end{equation}

where \begin{gather*}\hck^{(N,j)}=G_{\cO}^{(j)}\backslash\Gr^{(N)}\\ \hck_2^{(N,2j)}=G_{\cO}^{(2j)}\backslash (G_\cK^{(N,j)}\times^{G_\cO^{(j)}}\Gr^{(N)}).\end{gather*} and similar.
\end{rem}
We can assemble these objects into the following:

\begin{defin}\label{Hck-as-ind-pro}
	We denote the ind-pro-stack $$``\colim_{N\in\nn}"``{\lim_{j\geq j_N}}"\hck^{(N,j)}$$ by $$\hck,$$ and similarly
	$$\hck_2=``\colim_{N\in\nn}"``{\lim_{j\geq j_N}}"\hck_2^{(N,j)}.$$
\end{defin}

One can check that the convolution formula arising from push-pull along this diagram and \eqref{perverse-convolution-formula} agree. In particular, the choice of $j_N, j$ does not change the formula.

To see this, it suffices to note the following. First of all, by definition of the categories of perverse sheaves, the original formula for the convolution product is actually the one arising from the diagram $$	\begin{tikzcd}
	& {\hck_2^{(N)}} \\
	{\hck^{(N)}\times \hck^{(N)}} && \hck^{(2N)}
	\arrow["r", from=1-2, to=2-1]
	\arrow["{\overline m}"', from=1-2, to=2-3]
\end{tikzcd}$$ 

where \begin{gather*}\hck^{(N)}=G_{\cO}\backslash\Gr^{(N)}\\ \hck_2^{(N)}=G_{\cO}\backslash (G_\cK^{(N)}\times^{G_\cO}\Gr^{(N)}).\end{gather*}  and similar.

If $j\geq j_{N}$, we can further truncate the diagram to \eqref{hecke-conv-appendix-truncated}, and the convolution formula is again the same because \begin{itemize}\item $\Perv_{G_\cO}(\Gr^{(N)};R)\simeq \Perv_{G_{\cO}^{(j)}}(\Gr^{(N)};R)$.
	\item $G_\cK^{(N)}\times^{G_\cO}\Gr^{(N)}\simeq G_{\cK}^{(N,j)}\times^{G_{\cO}^{(j)}}\Gr^{(N)}$ by definition.
	\end{itemize}

Notice that the same arguments work at the level of equivariant constructible sheaves (see e.g. \cite[2.4.3]{Achar-Riche}).

\subsection{Models for the spherical Hecke category}\label{Section-models-Sph}

The main result of this paper \cref{final-theorem} is about $\Sph(G; \cE)^\topp$, a category of equivariant constructible sheaves on $\Gr^\an$, the analytification of the affine Grassmannian, with very general coefficients: namely, any presentable stable $\infty$-category $\cE$ works. In the present subsection, we will remark how that category (or more precisely, its small version $\Sph(G,;\cE)^\locc$ appearing in \cref{corollary-small-subcategory}) specializes to familiar ones in two special cases: with coefficients in finite/profinite/$\ell$-adic rings, or with complex coefficients. In the first case, there is an interpretation of the mentioned category in terms of algebraic constructible sheaves over the affine Grassmannian. Hence, let us gather a couple observations about those.

Let $\sS$ be the stratification by Schubert cells of the affine Grassmannian (\cref{perverse-and-stratification}). For $N\in\nn$, let $\sS^{(N)}$ be its restriction to $\Gr^{(N)}$.

\begin{defin}\label{pro-constructibles}Let $G$ be a reductive group over $\C$, and $R$ a finite ring. We define $$\Cons^\fd_{G_\cO}(\Gr^{(N)},\mathscr S^{(N)}; R)=\Cons^\fd_{G_{\cO}^{(j_N)}}(\Gr^{(N)}, \sS^{(N)};R)$$ with the notation of \cref{equivariant-constructible-category-schemes}.
\end{defin}

By \cref{equivariant-GAGA}, this category is equivalent to its counterparts defined in the complex-analytic world $$\Cons^\fd_{(G_{\cO}^{(j_N)})^\an}((\Gr^{(N)})^\an,(\sS^{(N)})^\an;R)$$ (\cref{top-constructible-sheaves-with-ring-coefficients}).

 As a consequence, by the same proof of \cref{unipotent-equivalence}, the definition is independent of $j_N$. One could also reach the same conclusion directly on the algebro-geometric side (\cite[Proposition 10.2.8 with $K=(0)$]{Achar-Riche}).

We can thus define
	\begin{gather*}
	\Cons^\fd_{G_\cO}(\Gr_G,\sS; R)=\colim_{N\geq 0}\Cons_{G_\cO}^\fd(\Gr^{(N)},\mathscr S^{(N)}; R)\\
		\cD^\fd_{\textup{c},G_\cO}(\Gr^{(N)}; R)=\cD^\fd_{\textup{c},G_{\cO}^{(j_N)}}(\Gr^{(N)}; R)\\
		\cD^\fd_{\textup{c},G_\cO}(\Gr; R)=\colim_{N\geq 0}\cD_{\textup{c},G_\cO}^\fd(\Gr^{(N)};R).
	\end{gather*}
which again will agree with their complex-analytic counterparts (by definition):
\begin{equation}\label{Grassmannian-GAGA}\begin{gathered}\Cons^\fd_{G_\cO}(\Gr_G,\sS; R)\simeq\Cons^\fd_{G_\cO^\an}(\Gr^\an,\sS^\an;R)\\
\cD^\fd_{\textup{c},G_\cO}(\Gr; R)\simeq \cD^\fd_{\textup{c},G_\cO^\an}(\Gr^\an; R).
\end{gathered}\end{equation}

\begin{rem}Note that the forgetful functor $$\Perv_{G_\cO}(\Gr,\sS;R)\to \Perv(\Gr,\sS;R)$$ is an equivalence (see for example \cite[Section 4.4]{Baumann-Riche}), but $$\Cons^\fd_{G_\cO}(\Gr,\sS;R)\to \Cons^\fd(\Gr,\sS;R)$$ is not.
\end{rem}
On the other hand:

\begin{prop}\label{equivariant-constructibles-are-constructible-Schubert}The map $$\Cons_{G_\cO}^\fd(\Gr,\sS;R)\to \cD^\fd_{\textup{c},G_\cO}(\Gr;R)$$ is an equivalence.\end{prop} 
\begin{proof}By definition, the claim can be checked on each $\Gr^{(N)}$, where it is true by \cref{lemma-noetherian-induction}.\end{proof}

\begin{rem}\label{remark-l-adic-GAGA}All these results, with the right definitions, are true for profinite and $\ell$-adic coefficients (cf. \cref{remark-schemes-profinite-coefficients}, \cref{profinite-GAGA}) as well. We will only mention such coefficients in the following \cref{comparison-with-algebraic-l-adic}, which serves as a connection with other definitions of the small spherical Hecke category appearing in the literature and only uses the contents of this subsection. Hence, we will not delve into further details about profinite and $\ell$-adic coefficients. The interested reader is encouraged to look at the references mentioned in \cref{remark-schemes-profinite-coefficients}.
\end{rem}


\begin{defin}\label{spherical-categories}
	Let $G$ be a complex reductive group and $R$ a discrete, prodiscrete or $\ell$-adic ring. The \textit{topological spherical Hecke category} of $G$ with coefficients in $R$ is $$\Sph(G;R)^\topp=\Cons_{G_\cO^\an}(\Gr^\an;R).$$
	The \textit{small spherical Hecke category} of $G$ with coefficients in $R$ as $$\Sph(G;R)^\textup{loc.c.}=\Cons_{G_\cO^\an}^\fd(\Gr^\an;R).$$
\end{defin}

Usually, in the Geometric Langlands Program, a renormalization of $\Sph(G)$ is used:

\begin{defin}\label{renormalized-spherical-category}
	Let $G$ be a complex reductive group and $R$ be a discrete, prodiscrete ring or $\ell$-adic ring. The \textit{renormalized spherical Hecke category} of $G$ with coefficients in $R$ is $$\Sph(G;R)^\textup{ren}=\Ind(\Sph(G;R)^\textup{loc.c}).$$
\end{defin}

\begin{rem}\label{comparison-with-algebraic-l-adic}{}Let $R$ be finite, profinite or $\ell$-adic. By \eqref{Grassmannian-GAGA} and \cref{equivariant-constructibles-are-constructible-Schubert},
$$\Sph(G;R)^\locc\simeq \Cons_{G_\cO}^\fd(\Gr;R)\simeq \cD_{\textup{c},G_\cO}^\fd(\Gr;R)$$ (and same for its Ind-completion). In particular, with these coefficients, the small and the renormalized spherical Hecke category do not distinguish between the algebraic and the analytic setting. Therefore, the same is true for the main result of our paper (\cref{corollary-small-subcategory}), although the proof uses features of the analytic setting.\end{rem}

\begin{rem}\label{comparison-dmodules}On the other hand, we can consider discrete infinite rings such as $\C$. In this case, $\Sph(G;\C)^\locc$ and $\Sph(G;\C)^\ren$ admit an interpretation in terms of D-modules. Namely, by the Riemann-Hilbert correspondence we have that $$\dc^\fd(\Gr^\an;\C)\simeq \Dmod(\Gr)^\omega.$$  Let $\Dmod_{G_\cO}(\Gr)^\locc\subset \Dmod_{G_\cO}(\Gr)$ be the full subcategory spanned by objects which become compact after forgetting the equivariant structure (\cite[12.2.3]{Arinkin-Gaitsgory}). Then we have an equivalence $$\cD^\fd_{\textup{c},G_\cO}(\Gr^\an;\C)\simeq \Dmod_{G_\cO}(\Gr)^\locc$$
	Combined with \cref{noetherian-topological}, this provides an equivalence $$\Sph(G;\C)^\locc\simeq\Dmod_{G_\cO}(\Gr)^\locc.$$
\end{rem}


\section{Recollections and complements in stratified homotopy theory}\label{appendix-stratified-spaces}

\subsection{Stratified schemes and stacks} 

Let us denote by $\Sch_\C$ the category of complex schemes, and by $\Sch_\C^\lft$ the full subcategory of complex schemes, locally of finite type.
\begin{defin}\label{stratified-schemes}The category of \textit{stratified complex schemes} is defined as $\Str\Sch_\C=\Sch_\C\times_{\Top}\Str\Top$, where the map $\Sch_\C\to \Top$ sends a scheme $X$ to its underlying \textit{Zariski} topological space, and the other map is the evaluation at $[0]$.
	
	We also define the category $$\Str\Sch_\C^\lft$$ as the full subcategory of $\Str\Sch_\C$ spanned by those stratified schemes such that the stratification is finite constructible in the sense of \cite[Definition 1.2.1]{Exodromy}, and such that the underlying scheme is locally of finite type.
\end{defin}

\begin{rem}\label{stratified-algebraic-stacks}By looking at the \'etale topology, we can consider the site $(\Sch_\C,\ett)$ and the category
	
	$$\Stk_\C^\lft=\Shv_\ett(\Sch_\C^\lft; \Grpd)$$ of \'etale stacks locally of finite type.

	We can define $\str\ett$ as the topology whose coverings are \'etale coverings whose stratification is induced by the one on the base. This leads to defining the category 
	\begin{equation}\label{stratified-stacks}\Str\Stk^\lft_\C=\Sh_{\str\ett}(\Str\Sch^\lft_\C; \Grpd)\end{equation} of stratified \'etale stacks locally of finite type (stratified stacks for short).
	\end{rem}

	\begin{defin}\label{uni}Let $\uni\subset \Mor(\Str\Sch_\C^\lft)$ be the class of those morphisms $f:X\to Y$ which:
	\begin{itemize}\item are smooth morphisms on the underlying schemes
	\item there exists a smooth stratified scheme $S$, together with stratified maps $X\to S, Y\to S$ which make the triangle $$\begin{tikzcd}X\arrow[rr,"f"]\arrow[dr]&&Y\arrow[dl]\\
	&S&\end{tikzcd}$$ commute and such that $f$ is a relative quotient by the action of a smooth relative group scheme $H\to S$ whose restriction $H\times_SS_\alpha\to S_\alpha$ is a torsor by a smooth unipotent group $H_\alpha$. In particular, \'etale-locally in $S_\alpha$, $H\times_SS_\alpha\to S_\alpha$ splits as $H_\alpha\times S_\alpha\to S_\alpha$ for some unipotent group scheme $H_\alpha$.
	\end{itemize}
	By abuse of notation, we define $\uni\subset \Mor(\Str\TStk)$ to be the class of morphisms $\cX\to \cY$ which are representable and whose pullback to any map $Z\to \cY$, with $Z$ representable, belongs to $\tri$.

	We define $\Pro_\uni(\Str\Sch_\C^\lft)$ to be the full subcategory of $\Pro(\Str\Sch_\C^\lft)$ spanned by those pro-objects which can be presented as formal limits of cofiltered diagrams with transition maps belonging to $\uni$.
\end{defin}

\begin{rem}
The class $\uni$ (both at the level of schemes and stacks) is stable under pullbacks.
\end{rem}

\begin{defin}\label{stratified-prestacks}We define the $(2,1)$-category
	
	$$\PStrStk=\Fun((\Pro_\uni(\Str\Stk_\C^{\lft}))^\op, \Grpd).$$

\end{defin}

\subsection{Constructible sheaves on stratified schemes and stacks}\label{Section-Cons-schemes}

\begin{defin}[{\cite[Example 13.2.9]{Exodromy}}]\label{constructible-sheaves-on-schemes}
	Let $(Y,s)$ be a complex stratified scheme of finite dimension, and $R$ a finite ring. 
	The $\infty$-category of constructible sheaves $$\Cons^\fd(Y,s;R)$$ is defined as the full subcategory of $$\Shv_{\ett}(Y,\Mod_R)$$ spanned by those objects which, after restriction to each stratum, are locally constant and lisse (in the sense of \cite[Definition 13.2.6]{Exodromy}).
\end{defin}
There is also the notion of constructible sheaves with respect to {\it some} stratification (instead of a fixed one). 

\begin{recall}[{\cite[4.1]{Bernstein-Lunts}, \cite[Proposition 4.2.5]{WC}}]\label{varying-strat} Let $Y$ be a complex scheme locally of finite type, and $R$ a finite ring.
	We define  
	\begin{equation}\label{dc}\dc^\fd(Y;R)=\colim_{\sS\textup{ algebraic stratification of }Y}\Cons(Y,\sS; R)\subset \Shv_\ett(Y;R).\end{equation} since the colimit in the right-hand-side is filtered: indeed, one can refine two algebraic stratification by a common one\footnote{This is easy because we did not assume the strata to be smooth. For the smooth case, see \cite{MO-refining-stratifications}.}. When $Y$ is quasiprojective, this is the $\infty$-category of compact objects\footnote{See \cite[Proposition 4.2.5]{WC}. To see the equivalence between constructibility and condition (1) in \textit{loc. cit.}, one can apply noetherianity in order to find suitable finite stratifications of $Y$.} in $\Shv_\ett(Y;R)$.
\end{recall}

Constructible sheaves on stratified stacks can be defined by right Kan extension. However, we will only need the special case of quotient stacks (i.e. the case of equivariant constructible shaves on schemes), which we recall below.

\begin{defin}\label{equivariant-constructible-category-schemes}Let $(Y,s)$ be a stratified complex scheme of finite dimension, $H$ a group scheme of finite type acting on $Y$ in such a way that the action sends strata to strata, and let $R$ be a finite ring. 
	There is a simplicial diagram 
	
	\begin{equation}\label{simplicial-diagram-action}\begin{tikzcd}
			\dots & {(H\times H\times Y,s_2)} & {(H\times Y,s_1)} & {(Y,s)}
			\arrow[shift left=2, from=1-1, to=1-2]
			\arrow[shift right=2, from=1-1, to=1-2]
			\arrow[shift right=6, from=1-1, to=1-2]
			\arrow[shift left=6, from=1-1, to=1-2]
			\arrow[from=1-2, to=1-1]
			\arrow[shift left=4, from=1-2, to=1-1]
			\arrow[shift right=4, from=1-2, to=1-1]
			\arrow[from=1-2, to=1-3]
			\arrow[shift right=4, from=1-2, to=1-3]
			\arrow[shift left=4, from=1-2, to=1-3]
			\arrow[shift left=2, from=1-3, to=1-2]
			\arrow[shift right=2, from=1-3, to=1-2]
			\arrow[shift right=2, from=1-3, to=1-4]
			\arrow[shift left=2, from=1-3, to=1-4]
			\arrow[from=1-4, to=1-3]
	\end{tikzcd}\end{equation}
	
	where $s_i$ is the stratification on $\overbrace{H\times\dots\times H}^{i}\times Y$ which is trivial on the group factors and $s$ on the last factor. As usual, in the left direction are induced by the identity element of $G$ in various ways, and maps in the right direction are induced by combinations of the action and the projections.
	
	The $\infty$-category $\Cons^\fd_{H}(Y,s;R)$ of $H$-equivariant constructible sheaves on $Y$ with respect to the stratification $s$ is defined as the limit of the cosimplicial diagram (induced by pullback of sheaves from \eqref{simplicial-diagram-action})
	
	\begin{equation}\label{alg-cosimplicial-diagram-Cons}\begin{tikzcd}
			\dots & {\Cons^\fd(H\times H\times Y,s_2;R)} & {\Cons^\fd(H\times Y,s_1;R)} & {\Cons^\fd(Y,s;R)}
			\arrow[shift left=2, from=1-2, to=1-1]
			\arrow[shift right=2, from=1-2, to=1-1]
			\arrow[shift right=6, from=1-2, to=1-1]
			\arrow[shift left=6, from=1-2, to=1-1]
			\arrow[from=1-1, to=1-2]
			\arrow[shift left=4, from=1-1, to=1-2]
			\arrow[shift right=4, from=1-1, to=1-2]
			\arrow[from=1-3, to=1-2]
			\arrow[shift right=4, from=1-3, to=1-2]
			\arrow[shift left=4, from=1-3, to=1-2]
			\arrow[shift left=2, from=1-2, to=1-3]
			\arrow[shift right=2, from=1-2, to=1-3]
			\arrow[shift right=2, from=1-4, to=1-3]
			\arrow[shift left=2, from=1-4, to=1-3]
			\arrow[from=1-3, to=1-4]
	\end{tikzcd}\end{equation}
	
\end{defin}


\begin{recall}\label{equivariant-varying-strat}Let $Y$ be a finite-dimensional scheme, $H$ a group scheme of finite type acting on $Y$, $R$ a finite ring. The $\infty$-category $\cD_{\textup{c},H}^\fd(Y;R)$ of $H$-equivariant constructible sheaves on $Y$ is the limit of the diagram (induced by pullback of sheaves from \eqref{simplicial-diagram-action})
	\begin{equation}\label{cosimplicial-diagram-dc}\begin{tikzcd}
			\dots & {\dc^\fd(H\times H\times Y;R)} & {\dc^\fd(H\times Y;R)} & {\dc^\fd(Y;R)}
			\arrow[shift left=2, from=1-2, to=1-1]
			\arrow[shift right=2, from=1-2, to=1-1]
			\arrow[shift right=6, from=1-2, to=1-1]
			\arrow[shift left=6, from=1-2, to=1-1]
			\arrow[from=1-1, to=1-2]
			\arrow[shift left=4, from=1-1, to=1-2]
			\arrow[shift right=4, from=1-1, to=1-2]
			\arrow[from=1-3, to=1-2]
			\arrow[shift right=4, from=1-3, to=1-2]
			\arrow[shift left=4, from=1-3, to=1-2]
			\arrow[shift left=2, from=1-2, to=1-3]
			\arrow[shift right=2, from=1-2, to=1-3]
			\arrow[shift right=2, from=1-4, to=1-3]
			\arrow[shift left=2, from=1-4, to=1-3]
			\arrow[from=1-3, to=1-4]
	\end{tikzcd}\end{equation}
\end{recall}

\begin{rem}\label{remark-schemes-profinite-coefficients}One can define constructible sheaves and equivariant constructible sheaves on stratified schemes (both with respect to a fixed stratification or not) with coefficients in profinite or $\ell$-adic rings, cf. \cite[§6.1]{Behrend}, \cite[Recollection 13.7.7, 13.8.7]{Exodromy}. Unlike the case of constructible sheaves on stratified topological spaces (\cref{profinite-etc}), this definition is not formal. However, in this paper we are only interested in proving that with the correct definitions a GAGA principle holds (\cref{profinite-GAGA}). Since the profinite and $\ell$-adic case arise via limits and filtered colimits from the finite case, it will suffice to prove the finite case (\cref{algebraic-equals-topological-cons}).
\end{rem}

In the setting of \cref{equivariant-varying-strat}, if one assumes that there are finitely many orbits of $H$, these orbits form a stratification themselves (\cite{MO-Gr-is-Whitney}), which we denote by $s$. Let $R$ be a finite ring. We have a pullback square of stable $\infty$-categories
\begin{equation}\label{cons-square}\begin{tikzcd}  
		\Cons^\fd_{H}(Y,s;R)\arrow[r]\arrow[d,hook]& \Cons^\fd(Y,s;R)\arrow[d,hook]\\
		\cD_{\textup{c},H}^\fd(Y;R)\arrow[r]& \dc^\fd(Y;R).
\end{tikzcd}\end{equation}
Note that the vertical functors are fully faithful because the transition maps in the colimit \eqref{dc} are, and the colimit is filtered. 
Now, the horizontal arrows in \eqref{cons-square} are not equivalences.  On the contrary:

\begin{lem}\label{lemma-noetherian-induction}Let $R$ be a finite ring. Let $H$ be a group scheme acting on a finite-dimensional scheme $Y$, and suppose that there are finitely many orbits, forming a stratification $s$ of $Y$. Then the functor $\Cons_H^\fd(Y,s;R)\to\cD^\fd_{\textup c,H}(Y;R)$ is an equivalence.\end{lem}
\begin{proof}
	
	We already remarked that the functor is fully faithful. We now argue like in \cite{Sawin-answer}. Let us now consider an equivariant constructible sheaf $\cF$ with respect to some stratification, and let us prove that it is constructible with respect to the orbit stratification. Let us consider the maximal open subset $U$ of $Y$ where the sheaf is locally constant: this is nonempty since we know that $\cF$ is constructible with respect to some stratification, and any stratification of a finite-dimensional scheme has an open stratum. More subtly, there is an open dense stratum in every connected component of $Y$, and by taking unions of such over all the connected components of $Y$, we obtain an $U$ such that its complementary is closed of dimension strictly smaller than $\dim Y$. Also, $U$ is unique, since the union of two open subsets where $\cF$ is locally constant has again the property that $\cF$ is locally constant there. Now, $U$ is $H$-stable by equivariancy of $\cF$ and maximality of $U$ itself, and thus its complementary is $H$-stable as well and we can apply Noetherian induction.\end{proof}

\subsection{Stratified topological spaces and stacks}

The following definition is a particular case of \cite[8.2.1 and ff.]{Exodromy}.
\begin{defin}\label{strtspc}Let $\Top$ be the 1-category of topological spaces. The category of \textit{stratified topological spaces} is defined as
	$$\Str\Top_\C=\Fun(\Delta^1,\Top)\times_{\Top}\Poset,$$
	where the map $\Fun(\Delta^1,\Top)\to \Top$ is the evaluation at $1$, and $\Alex:\Poset\to \Top$ assigns to each poset $P$ its underlying set with the so-called Alexandrov topology (see \cite[Definition 1.1.1]{Exodromy}).
\end{defin}

Note that $\Str\Top$ is complete and cocomplete: see \cite[Remark 2.2]{WM}, and originally \cite[Proposition 6.1.4.1]{Nand-Lal-thesis} in a slightly different setting.

\begin{recall}Let $(X,s)$ be a stratified topological space. The notion of conical stratification as given in \cite[Definition A.5.5]{HA} amounts to asking that around each point of $X$ there exists a neighbourhood which is stratified homemorphic to $Z\times \tC(Y)$ where $Z$ is an unstratified space and $\tC(Y)$ is the stratified open cone of a stratified space $Y$. Being conical is the main condition required to a stratified space in order to make the so-called Exodromy Theorem (\cref{exodromy}) true.\end{recall}

For simplicity, in the present paper a stratified space is called \textit{conical} if satisfies several conditions altogether:

\begin{defin}\label{defin-con}The category $\Con$ is the full subcategory of $\Strat$ spanned by those stratified spaces $(X, s:X\to P)$ such that: \begin{itemize}
		\item $X$ is locally of singular shape in the sense of \cite[Definition A.4.15]{HA};
		\item the strata of $X$ are locally weakly contractible;
		\item $P$ satisfies the ascending chain condition;
		\item the stratification is conical in the sense of \cite[Definition A.5.5]{HA}.
	\end{itemize}
	This category admits finite products, essentially because the product of two cones is the cone of the join space. Therefore, there is a well-defined symmetric monoidal Cartesian structure $\Con^\times$.
\end{defin}

\begin{rem}[Condition of the frontier]\label{condition-of-the-frontier}
	One consequence of the conicality assumption is the condition of the frontier, i.e. the fact that given two strata $X_p,X_q$, such that $X_p$ is connected and such that $\overline{X_q}\cap X_p\neq\varnothing$, then $X_p\subset \overline{X_q}$. This is true locally by inspection of the conical model and is globalized by connectedness of $X_p$. Notice that, whenever this happens, we automatically get that $p\in s(\overline{X_q})\subset \overline{\{q\}}=P_{\leq q}$, i.e. $p\leq q$.
\end{rem}

\begin{notation}\label{stratified-topological-stacks} Let us consider the topology of local homeomorphisms on the topological side (which has however the same sheaves as the topology of open embeddings). We have thus the site $(\Top,\loc)$ and the category
$$\TStk=\Shv_\loc(\Top; \Grpd)$$ of topological stacks.

Let $\str\loc$ be the topology whose coverings are jointly surjective families of local homeomorphisms such that the stratification on the total space is induced by that on the base. This defines $(2,1)$-categories \begin{gather*}\Str\TStk=\Sh_{\str\loc}(\Str\Top; \Grpd)\\
\Str\TStk_\con=\Sh_{\str\loc}(\Con; \Grpd)\end{gather*}
 of (conically) stratified topological stacks.

\end{notation}

\begin{defin}Let $\alpha,\beta:(X,s)\to (Y,t)$ be two stratified maps between stratified topological spaces. Let $\widetilde s$ be the stratification of $[0,1]\times X$ induce by the projection $[0,1]\times X\to X$. A stratified homotopy between $\alpha$ and $\beta$ is a stratified map $$H:([0,1]\times X, \widetilde s)\to (Y,t)$$ such that $H(0,-)=f, H(1,-)=g$.
\end{defin}

\begin{defin}\label{stratified-homotopy-equivalence}A stratified homotopy equivalence between stratified topological spaces is a stratified map $f:(X,s)\to (Y,t)$ such that there exist a stratified map $g:(Y,t)\to (X,s)$ and stratified homotopies $fg\sim \id_Y, gf\sim \id_X$.
\end{defin}

\begin{rem}Note that the class of stratified homotopy equivalences is not closed under pullback (just like homotopy equivalences of topological spaces are closed under homotopy pullback but not under pullback).\end{rem}

\begin{rem}
	A stratified homotopy equivalence induces an isomorphism at the level of posets, and homotopy equivalences on each stratum.
\end{rem}

\begin{defin}\label{defin-smooth}Let $f:(X,s)\to (Y,t)$ be a morphism in $\Con$. We say that $f$ is \textit{stratified smooth} if, locally in the topology of stratified local homeomorphisms on $X$, it is of the form of a projection $(Y,t)\times \mathbb R^N\to (Y,t)$ for some $N$ (where $\rr^N$ is seen as a trivially stratified space).

We say that $f$ is a \textit{smooth stratified submersion} if it is stratified smooth in the above sense and, for each stratum $Y_\alpha$ of $(Y,t)$, the restriction $f_\alpha:f^{-1}(Y_\alpha)\to Y_\alpha$ satisfies the following property: locally in the topology of $Y_\alpha$, it splits as a product whose fiber is trivially stratified. In particular, $f_\alpha$ is a Serre fibration and its fiber is a topological manifold (by the smoothness condition above).

	We denote the class of smooth stratified submersions by $\subm$.
\end{defin}

\begin{rem}
	Smooth stratified submersions are closed under pullback.
\end{rem}

\begin{rem}\label{recall-smooth-base-change-sheaves}
If all stratifications are trivial, the notion of smooth stratified submersion recovers the notion of topological submersion of manifolds.

Also, any smooth stratified map is a \textit{shape submersion} in the sense of \cite[Definition 3.21]{Marco-six-functors}. In particular, if $f:(X,s)\to (Y,t)$ is stratified smooth and $\cE$ is a presentable stable category, \cite[Lemma 3.24]{Marco-six-functors} ensures the existence of a functor $f_\#:\Sh(X,s;\cE)\to \Sh(Y,t;\cE)$ which is left adjoint to $f^*$.

The adjunction $f_\#\dashv f^*$ satisfies a base change property \cite[Lemma, 3.25]{Marco-six-functors}.
\end{rem}

\begin{lem}\label{lower-sharp-preserves-constructibles}Let $\cE$ be a presentable category, and $p:(X,s)\to (Y,t)$ be in $\subm$. Then $p_\#:\Sh(X; \cE)\to \Sh(Y;\cE)$ restricts to a functor $\Cons(X,s;\cE)\to\Cons(Y,t;\cE)$, i.e. preserves constructible sheaves.\end{lem}
\begin{proof}Let $\cF$ be a constructible sheaf on $(X,s)$. Let $Y_\alpha$ be a stratum of $(Y,t), X_\alpha=p^{-1}(Y_\alpha), p_\alpha:X_\alpha\to Y_\alpha$ the restriction. The subspace $X_\alpha$ is in turn a stratum of $X$, because we assume that $p$ is a bundle on $Y_\alpha$ with trivially stratified fiber.
Note now that by the base change property recalled in \cref{recall-smooth-base-change-sheaves}, the restriction of $p_\#\cF$ to a stratum $Y_\alpha$ of $(Y,t)$ coincides with $p_{\alpha,\#}(\cF|_{X_\alpha}).$ Therefore, in order to prove that $p_\#$ preserves constructible sheaves, it suffices to prove that each $p_{\alpha,\#}$ preserves local systems. This condition is local in the topology of $Y_\alpha$, and therefore follows from \cite[Lemma 3.28]{HPT}.
\end{proof}
Since we have an adjunction $p_\#\dashv p^*$ at the level of categories of all sheaves, the following holds.
\begin{cor}\label{lower-sharp-is-constructible-left-adjoint}Under the hypotheses of \cref{lower-sharp-preserves-constructibles}, the induced functor $p_\#:\Cons(X,s;\cE)\to \Cons(Y,t;\cE)$ is left adjoint to $p^*:\Cons(Y,t;\cE)\to \Cons(X,s;\cE)$.\end{cor}

\begin{defin}\label{tri}Let $\tri\subset \Mor(\Str\Top)$ be the class of those morphisms $f:X\to Y$ which:
	\begin{itemize}\item are smooth stratified submersions in the sense of \cref{defin-smooth}
	\item when restricted to each stratum $Y_\alpha$ of $Y$, become trivial Serre fibrations (in particular, the stratification on the source becomes trivial).
	\end{itemize}
\end{defin}

\begin{defin}
	By abuse of notation, we denote by $\tri\subset \Mor(\Str\TStk)$ be the class of morphisms $\cX\to \cY$ which are representable and whose pullback to any map $Z\to \cY$, with $Z$ representable, belongs to $\tri$.

	We define $\Pro_\tri(\Str\TStk)$ to be the full subcategory of $\Pro(\Str\TStk)$ spanned by those pro-objects which can be presented as formal limits of small cofiltered diagrams with transition maps belonging to $\tri$; we also define $\Pro_\tri(\Str\TStk_\con)$ to be the full subcategory of $\Pro(\Str\TStk_\con)$ spanned by those pro-objects which can be presented as formal limits of small cofiltered diagrams with transition maps belonging to $\tri$.
\end{defin}

\begin{rem}The class $\tri$ (both at the level of spaces and stacks) is stable under pullback.\end{rem}

\begin{defin}\label{stratified-topological-prestacks}We define the $(2,1)$-categories
	\begin{gather*}\PStrat=\Fun((\Pro_\tri(\Str\TStk))^\op, \Grpd).\\
	\PCon=\Fun((\Pro_\tri(\Str\TStk_\con))^\op, \Grpd).\end{gather*}

\end{defin}

\begin{rem}\label{Con-to-StrTStk}Note that there is a canonical fully faithful embedding $$\PCon\hookrightarrow\PStrat.$$ This is defined as follows. First of all, there is a functor $$\Str\TStk_\con\to\Str\TStk$$ given by left Kan extension, which preserves colimits. 
Then, this can be extended to a functor $$\Pro_\tri(\Str\TStk_\con)\to\Pro_\tri(\Str\TStk)$$ by using right Kan extension, and the new functor preserves cofiltered limits. 
Finally, we can extend the latter functor to $$\PCon\to \PStrat$$ by left Kan extension, and the new functor preserves all colimits.\end{rem}

\begin{defin}\label{defin-she-prestacks}
	A morphism $f:\cX\to\cY$ in $\StrTStk$ belongs to $\she'$ if it is representable and it can be presented as a small colimit of maps in $\she$. 

	A morphism $f:\cX\to \cY$ in $\Pro_\tri(\StrTStk)$ belongs to $\she''$ if it can be presented as $$``\lim_{j\in J}" f_{j}$$ where $J$ is small cofiltered and each $f_{j}$ is equivalent to a map in $\she'$ in $\StrTStk$.

	A morphism $f:\cX\to \cY$ in $\PStrat$ belongs to $\pshe$ if it can be presented as a colimit of maps in $\she''$.

	These classes restrict to classes of maps in the correspondent categories formed from $\Con$ instead of $\Str\Top$, and will be denoted in the same way.
\end{defin}

\begin{defin}\label{defin-subm-prestacks}
	A morphism $f:X\to\cY$ in $\StrTStk$ belongs to $\widetilde\subm$ if it is representable and its pullback to any representable belongs to $\subm$. 

	A morphism $f:\cX\to\cY$ in $\StrTStk$ belongs to $\subm$ if there is a commutative square 
	$$\begin{tikzcd}
	X\arrow[r]\arrow[d]&Y\arrow[d]\\
	\cX\arrow[r]&\cY
	\end{tikzcd}$$ 
	where $X,Y$ are representable, the vertical arrows belong to $\widetilde\subm$ and the top horizontal arrow belongs to $\subm$.

	A morphism $f:\cX\to \cY$ in $\Pro_\tri(\StrTStk)$ belongs to $\subm''$ if it can be presented as $$``\lim_{j\in J}" f_{j}$$ where $J$ is cofiltered and each $f_{j}$ is equivalent to a map in $\subm'$ in $\StrTStk$.

	A morphism $f:\cX\to \cY$ in $\PStrat$ belongs to $\psubm$ if for every $\mathcal Z\in \Pro_\tri(\Str\TStk)\hookrightarrow\PStrat$ and a map $\mathcal Z\to \cY$, the pullback in $\PCon$ $$\cX\times_\cY\mathcal Z$$ belongs to the essential image of the Yoneda embedding $\Pro_\tri(\StrTStk)\hookrightarrow \PStrat$ and the canonical map $$\cX\times_\cY\mathcal Z\to \mathcal Z$$ is equivalent to a map in $\subm''$ in $\Pro_\tri(\StrTStk)$.

	These classes restrict to classes of maps in the correspondent categories formed from $\Con$ instead of $\Str\Top$, and will be denoted in the same way.
\end{defin}

\begin{rem}\label{subm-pullback}
The classes of maps $\psubm$ in $\PStrat$ and in $\PCon$ are closed under pullbacks.
\end{rem}

\subsection{Constructible sheaves on stratified topological spaces and stacks}\label{Section-Cons-spaces}

\begin{recall}\label{exodromy}Let $(Y,s)\in \Con$. The $\infty$-category $$\Cons(Y,s;\cS)$$ of space-valued constructible sheaves is defined in \cite[Definition A.5.2]{HA} and proven to be equivalent (by \cite[Theorem A.9.3]{HA} , nowadays known as the \textit{Exodromy Theorem in topology}), to the $\infty$-category $$\Fun(\Exit(Y,s),\cS).$$ Here $\Exit(Y,s)$ is a small $\infty$-category called the $\infty$-category of exit paths on $(Y,s)$ (see \cite[Definition A.6.2]{HA}, where it is denoted by $\Sing^A(Y)$, $A$ being the poset associated to the stratification). \end{recall}

\begin{rem}\label{hyperconstructibles-are-constructibles}Every space which is locally of singular shape in the sense of \cite[Definition A.4.15]{HA} has the property that every locally constant sheaf is automatically hypercomplete (\cite[Corollary A.1.17]{HA}).\end{rem}

One can consider coefficient categories different than $\cS$, namely any presentable stable $\infty$-category.

\begin{defin}\label{topological-constructibles}Let $(Y,s)\in \Con$ and $\cE$ a presentable stable $\infty$-category. We define $$\Cons(Y,s;\cE)$$ as the full subcategory of $\Shv(Y;\cE)$ spanned by constructible sheaves in the sense of \cite[Definition 2.27]{Exodromy-PT} (hypercompleteness in \textit{loc. cit.} can be ignored thanks to \cref{hyperconstructibles-are-constructibles}).
\end{defin}


\begin{thm}[\cite{Exodromy-PT}]\label{Exodromy-extended-coefficients}
	Let $(Y,s)\in \Con$, and $\cE$ a presentable stable $\infty$-category. Then $$\Cons(Y,s;\cE)\simeq \Fun(\Exit(Y,s),\cE).$$
\end{thm}
\begin{proof}
	This is the content of \cite[Theorem 5.17]{Exodromy-PT} together with \cite[Remark 5.18]{Exodromy-PT} and \cref{hyperconstructibles-are-constructibles}.
\end{proof}

\begin{rem}\label{change-of-coefficients-presheaves} Note that, for $\cc$ any small $\infty$-category and $\cE$ presentable, \begin{equation}\Fun(\cc,\cE)\simeq \Fun(\cc,\cS)\otimes \cE.\end{equation} 
	This follows from the fact that if $\cc$ is small then $\Fun(\cc,\cS)$ is presentable and from the formula $$\cA\otimes \cB\simeq \Fun^{\textup{R}}(\cB^\op, \cA)$$ \cite[Proposition 4.8.1.17]{HA} for $\cA,\cB$ presentable. As a consequence, under the hypotheses of \cref{Exodromy-extended-coefficients}, we have $$\Cons(Y,s;\cE)\simeq\Cons(Y,s;\cS)\otimes\cE.$$
\end{rem}
The following \cref{profinite-etc} and \cref{top-constructible-sheaves-with-ring-coefficients}, and related results in the rest of the paper, can more elegantly be formulated and proven in the pro-\'etale setting. Only by simplicity, we choose not to do so and give ad hoc definitions (which are easily recovered from that formalism).

\begin{defin}\label{profinite-etc}Let $R$ be a discrete ring. We denote \begin{gather*}\Mod_R^\cont=\Mod_R\\
		\Mod_R^\fd=\Mod_R^{\cont,\fd}=\Perf_R.\end{gather*} Let us also denote by $\Mod_R^\tors, \Mod_R^{\fd,\tors}$ the respective full subcategories of torsion modules.
	
	Let $R=\lim_{i\in \cI} R_i$ be a prodiscrete ring, i.e. $\cI$ is cofiltered, $R_i$ is discrete for each $i$ and $R$ has the limit topology.
	
	Then we define $$\Mod^\cont_R=\lim_{i\in \cI} \Mod_{R_i}$$ as a limit in $\Prl$ (transition maps are given by tensor product) and $$\Mod_R^{\cont,\fd}=\lim_{i\in \cI} \Perf_{R_i}.$$ In the same way we define $\Mod^{\cont, \tors}=\lim_i\Mod^\tors_{R_i}$ and $\Mod^{\cont,\fd,\tors}=\lim_i\Mod^{\fd, \tors}_{R_i}$.
	
	Let $R$ be a finite extension of $\Q_\ell$ (from now on an ``$\ell$-adic ring''). We define \begin{gather*}\Mod^\cont_R=\textup{cofib}[\Mod^{\cont, \textup{tors}}_{\cO_R}\hookrightarrow \Mod^{\cont}_{\cO_R}]\\
		\Mod^{\cont,\fd}_R=\textup{cofib}[\Mod^{\cont,\fd, \textup{tors}}_{\cO_R}\hookrightarrow \Mod^{\cont, \fd}_{\cO_R}]\end{gather*} where the cofibers are taken respectively in $\PrLst$ and in $\Catex$.
	Finally, let $R$ be an algebraic extension of $\Q_\ell$. 
	Then we define \begin{gather*}\Mod^\cont_R=\colim_{\Q_\ell\subset E\subset R\textup{ finite subextension}}\Mod^\cont_{E}\\
		\Mod_R^{\cont,\fd}=\colim_{\Q_\ell\subset E\subset R\textup{ finite subextension}}\Mod^{\cont,\fd}_{E}.\end{gather*}\end{defin}

Note that since the inclusion functors $\PrLst\hookrightarrow\Prl$ and $\Catex\hookrightarrow \Cat$ are closed under limits and filtered colimits, all of the above are stable $\infty$-categories.

\begin{rem}\label{top-constructible-sheaves-with-ring-coefficients}If $R$ is a discrete ring, we denote \begin{gather*}\Cons(Y,s;R)=\Cons(Y,s,\Mod_R)\\
\Cons^\fd(Y,s;R)=\Cons^\fd(Y,s;\Mod_R^\fd).\end{gather*}
Let $R=\lim_i R_i$ be a profinite ring. We can apply \cref{Exodromy-extended-coefficients} with $\cc=\Mod_R^\cont$ and obtain \begin{equation}\label{Exodromy-Mod-R}
		\Cons(Y,s;\Mod^\cont_R)\simeq \Fun(\Exit(Y,s),\Mod^\cont_R)\simeq\lim_i\Cons(Y,s; R_i)
	\end{equation} which we denote by $\Cons(Y,s;R)$.
	This equivalence restricts to an equivalence $$\Cons^\fd(Y,s;R)\simeq\Fun(\Exit(Y,s),\Mod^{\cont,\fd}_R)\simeq \lim_i\Cons(Y,s;R_i),$$ which we denote by $\Cons^\fd(Y,s;R)$. Note that, in general, this is \textit{not} the category of compact objects of $\Cons(Y,s;R)$.
	
	When $R$ is a finite extension of $\Q_\ell$, by \cref{change-of-coefficients-presheaves} we have \begin{gather*}\Cons(Y,s;\Mod^\cont_R)\simeq \Fun(\Exit(Y,s),\Mod_R^\cont)\simeq\\ \Fun(\Exit(Y,s),\cS)\otimes \textup{cofib}(\Mod^{\cont,\textup{tors}}_{\cO_R}\to \Mod^{\cont}_{\cO_R})\simeq\\
		\textup{cofib}(\Cons^{\textup{tors}}(Y,s;\cO_R)\to\Cons(Y,s;\cO_R)).\end{gather*} When $R$ is an algebraic extension of $\Q_\ell$, we have $$\Cons(Y,s;\Mod^\cont_R)\simeq \colim_{\Q_\ell\supset E\supset R\textup{ finite subextension}}\Cons(Y,s;\Mod^\cont_E)$$
	which we denote by $\Cons(Y,s;R)$. We also define, in all previous cases of $R$, $$\Cons^\fd(Y,s;R)=\Cons^\fd(Y,s;\Mod^{\cont, \fd}_R).$$
\end{rem}

\begin{recall}\label{topological-dc}Let $Y$ be a topological space, $\cE$ a stable presentable $\infty$-category. We define $$\dc(Y;\cE)=\colim_{(Y,s)\textup{ conical stratification}}\Cons(Y,s;\cE).$$\end{recall}

As we will see in \cref{subsection-correspondences}, the functor $\Cons(-;\cE):\Con^\op\to \Prl_\cE$ can be right Kan extended to $\ConStk$. We consider here the case of quotient stacks, i.e. the case of equivariant constructible sheaves over stratified topological spaces.

\begin{defin}\label{rem-alg-top-cons}Let $(Z,t)$ be a stratified topological space and $K$ a topological group acting on $Z$ compatibly with $t$. Let $\cE$ be a presentable stable $\infty$-category. We define the the category of $K$-equivariant $\cE$-valued constructible sheaves on $Z$ $$\Cons_{K}(Z,t;\cE)$$ as the limit of the diagram (induced by pullback of sheaves along \eqref{simplicial-diagram-action}) \begin{equation}\label{cosimplicial-diagram-Cons}\begin{tikzcd}
			\dots & {\Cons(K\times K\times Z,t_2;\cE)} & {\Cons(K\times Z,t_1;\cE)} & {\Cons(Z,t;\cE)}
			\arrow[shift left=2, from=1-2, to=1-1]
			\arrow[shift right=2, from=1-2, to=1-1]
			\arrow[shift right=6, from=1-2, to=1-1]
			\arrow[shift left=6, from=1-2, to=1-1]
			\arrow[from=1-1, to=1-2]
			\arrow[shift left=4, from=1-1, to=1-2]
			\arrow[shift right=4, from=1-1, to=1-2]
			\arrow[from=1-3, to=1-2]
			\arrow[shift right=4, from=1-3, to=1-2]
			\arrow[shift left=4, from=1-3, to=1-2]
			\arrow[shift left=2, from=1-2, to=1-3]
			\arrow[shift right=2, from=1-2, to=1-3]
			\arrow[shift right=2, from=1-4, to=1-3]
			\arrow[shift left=2, from=1-4, to=1-3]
			\arrow[from=1-3, to=1-4]
	\end{tikzcd}\end{equation}
	where $s_i$ is the stratification on $\overbrace{H\times\dots\times H\times Y}^{i}$ which is trivial on the group factors and $s$ on the last factor, and the diagram is the simplicial diagram encoding the action of $H$ on $Y$.
\end{defin}

\begin{defin}Let $(Y,s)$ be a stratified topological space, $H$ a topological group acting on $Y$ compatibly with $s$, and let a presentable stable $\infty$-category. We define $$\Cons_{H}^\fd(Y, s;\cE)=\Cons_H(Y,s;\cE)\times_{\Cons(Y,s;\cE)} \Cons^\fd(Y,s;\cE)\simeq \lim_{n\in\bDelta}\Cons^\fd(H^n\times Y, s_n;\cE)$$\end{defin}

\begin{lem}\label{noetherian-topological}Let $\cE$ be a symmetric monoidal presentable stable $\infty$-category. Let $K$ be a topological group acting on a topological space $Z$ locally of singular shape, and suppose that the orbits form a conical stratification $t$ of $Z$. Then the functor $\Cons_K^\fd(Z,t;\cE)\to\cD^\fd_{\textup c,K}(Z;\cE)$ is an equivalence.\end{lem}
\begin{proof}
	As in the proof of \cref{lemma-noetherian-induction}, we already know that the functor is fully faithful. Let us now consider an equivariant constructible sheaf $\cF$ with respect to some conical stratification $(Z,s:Z\to P)$, and let us prove that it is constructible with respect to the orbit stratification. By the ascending chain condition for $P$, there exists at least an open stratum in $s$. In fact, all minimal depth strata (i.e. strata associated to a maximal $p$ in $P$) are open: indeed, let $p{}$ be a maximal element of $P$, and consider $\overline{Y\setminus Y_p}$. Suppose for simplicity that $Y_p$ is connected. If this subspace intersects $Y_p$ nontrivially, in particular there exists another stratum $Y_q,q\neq p$, such that $\overline{Y_q}\cap Y_p\neq \varnothing$. By the condition of the frontier \cref{condition-of-the-frontier}, this implies that $Y_p\subset \overline{Y_q}$ and hence $q\geq p$, contradiction. 
	
	We define $U$ as the maximal open union of strata such that $\cF$ is locally constant over $U$. It is nonempty because $\cF$ is constructible with respect to $s$ and by the above remark about minimal depth strata. Moreover, $U$ is unique because the union of two open subsets where $\cF$ is locally constant has again the property that $\cF$ is locally constant there. Also, the complementary of $U$ is concentrated in strata of non-minimal depth, again by the above remark. Now, $U$ is $K$-stable by equivariancy of $\cF$ and maximality of $U$ itself\footnote{To see this, one can argue as follows: suppose $U$ is not $K$-stable, and let $g\in K$ such that $gU\neq U$. Pick a point $v\in gU$, and an open subset $V\subset U$ such that $gV\ni v$ and $\cF|_V$ is constant with value $E\in \cE$. Then, by equivariance, we have that $\cF(gV)\simeq \cF(V)\simeq E$. Hence, $\cF$ is locally constant on $U\cup gV$, which contradicts maximality of $U$.}, and thus its complementary is $K$-stable as well and we can apply Noetherian induction on the maximum lenght of ascending chains of $P$.\end{proof}

\subsection{Stratified analytification}\label{Stratifications}

\begin{thm}[{\cite[Th\'eor\`eme et d\'efinition 1.1]{SGA1-XII}}]\label{Raynaud}
	Let $X$ be a scheme locally of finite type over $\C$. Then there exists an associated complex-analytic space $X^\an$, whose underlying set is $X(\C)$ and which represents the functor $$\{\textup{complex-analytic spaces}\}\to \sets$$
	$$Y\mapsto \Hom_{\textup{ ringed spaces }}(Y,X).$$
\end{thm}

We can forget the structure sheaf of holomorphic functions and recover an underlying Hausdorff topological space (which corresponds to the operation denoted by $|-|$ in \cite{SGA1-XII}). We thus obtain a functor which by  abuse of notation we denote by $$(-)^\an:\Sch_\C^\lft\to \Top$$ (instead of $|(-)^\an|$).

\begin{rem}\label{an-left-exact}This functor preserves finite limits (\cite[1.2]{SGA1-XII}) and sends \'etale coverings to coverings in the local homeomorphism topology (\cite[Proposition 3.1]{SGA1-XII}).\end{rem}

\begin{construction}\label{stratified-analytification}
There is a natural stratified version of the functor $(-)^\an$, namely the one which accounts for the stratification induced by the map of ringed spaces $u:S^\an\to S$ coming with the universal property: $$\Str\Sch^\lft_\C\to \Str\Top$$
	$$(S,s)\mapsto (S^\an, s\circ u).$$
	
\end{construction}

\begin{lem}[Smooth algebraic maps]\label{analytification-of-smooth}
	Let $X,Y$ be complex stratified schemes, locally of finite type, and $f:(X, s)\to (Y,t)$ be a smooth morphism in the sense of algebraic geometry, such that the stratification $s$ is the pullback of $t$ along $f$. Then $f^\an$ is a smooth morphism in the sense of \cref{defin-smooth}. In particular, the analytification of a trivially stratified smooth complex scheme locally of finite type is a topological manifold.
\end{lem}
\begin{proof}
	The trivially stratified case is \cite[Exercise 12.6.A]{Vakil}; the stratified statement follows straightforwardly.
\end{proof}

\begin{eg}[Relative stratified torsors with smooth fiber]\label{analytification-of-smooth-torsors}
	Let $X,Y,S$ be complex stratified schemes, locally of finite type, and $X\to S, Y\to S$ maps. Suppose that $S$ is smooth, and let $f:X\to Y$ be a smooth torsor, relative over $S$. In particular, \'etale locally on $Y$, the map has the form $Y\times_SZ\to Y$ where $z:Z\to S$ is a smooth morphism. Suppose further that the stratification of $Z$ is the pullback of the stratification of $S$ along $z$. Then by \cref{analytification-of-smooth} $f$ is a smooth stratified morphism.

	In our applications, we will usually have the further condition that, when restricted to each stratum $S_\alpha$of $S$, the map splits as $$\begin{tikzcd}X'\times S_\alpha\arrow[rr, "f'\times \id"]\arrow[rd]&& Y'\times S_\alpha\arrow[ld]\\
	&S_\alpha & \end{tikzcd}$$ for some smooth map $f':X'\to Y'$ of trivially stratified smooth schemes. In particular, when restricted to every stratum $Y_\alpha$ of $Y$, locally in the topology of $Y_\alpha$ $f$ splits as a product $X''\times Y_\alpha\to Y_\alpha$, for $X''$ a smooth trivially stratified scheme.
\end{eg}

\begin{prop}\label{uni-tri}The stratified analytification functor $(-)^\an:\Str\Sch_\C^\lft\to \Str\Top$ sends stratified \'etale coverings to stratified coverings in the topology of local homeomorphisms. Also, it sends morphisms in $\uni$ (cf. \cref{uni}) to morphisms in $\tri$.
\end{prop}

\begin{proof}The first part follows from \cite[Proposition 3.1]{SGA1-XII}.

For the second part, let $f:(X,s)\to (Y,t)$ be in $\uni$. Note first of all that $f^\an$ is in $\subm$ by \cref{analytification-of-smooth-torsors}. Let $f_\alpha$ denote the restriction of $f$ to a stratum $Y_\alpha$ of $Y$. \'Etale-locally in the topology of $Y_\alpha$, $f_\alpha$ splits as a product $Y_\alpha\times H_\alpha\to Y_\alpha$, where $H_\alpha$ is a unipotent group scheme (by definition, it is true after restriction to strata of $S$, and \textit{a fortiori} of $Y$). Now, a unipotent group scheme is isomorphic (as a scheme) to some affine space $\mathbb A^N_\C$ (see \cite[Theorem 8.0]{Unipotent-groups}). Therefore, since $H_\alpha$ is unipotent, its analytification is homeomorphic to $(\mathbb A^N_\C)^\an$ for some $N$, hence contractible. It follows that $f_\alpha^\an$ is locally trivial with contractible fiber, hence a trivial Serre fibration. 
\end{proof}

\begin{lem}\label{LKE-are-left-exact}
Let $F:\cc\to \cD, G:\cc\to \cE$ be functors, where $\cc,\cD,\cE$ all have finite limits, the slice $\cc_{/d}$ is small for every $d\in \cD$, and $\cE$ is an $\infty$-topos. Suppose also that $G$ preserves finite limits. Then the left Kan extension of $G$ along $F$ exists and preserves finite limits.
\end{lem}
\begin{proof}
The existence of the left Kan extension comes from the fact that $\cE$ has all small colimits. Let us call $H:\cD\to \cE$ this left Kan extended functor. In the case when $\cE=\cP(\cE_0)$ is a presheaf category, we can reduce easily to $\cE=\cS$. In that case, $G$ is left exact if and only if it is equivalent to a colimit of representables when seen as an object of $\cP(\cc^\op)$. Since the left Kan extension procedure $$\cP(\cc)\to \cP(\cD)$$ is a left adjoint, it preserves such colimits, and it also preserves representables: hence $H:\cD\to \cS$ is left exact as well.

In the general case, by assumption there is a left exact localization $L$ from a category of presheaves $\PSh(\cE_0)$ to $\cE$. Then we have a diagram 

\[\begin{tikzcd}
	\cc & \cE & {\PSh(\cE_0)} & \cE \\
	\cD
	\arrow["G", from=1-1, to=1-2]
	\arrow["F"', from=1-1, to=2-1]
	\arrow["R", from=1-2, to=1-3]
	\arrow["L", from=1-3, to=1-4]
	\arrow["{H_0}", from=2-1, to=1-3]
	\arrow["H"', from=2-1, to=1-4]
\end{tikzcd}\]

where $R$ is the fully faithful right adjoint to $L$, $H_0$ is the left Kan extension of $RG$ along $F$, and $H$ is the left Kan extension of $LRG$ along $F$. Since $R$ is fully faithful, the counit $LR\to\id_\cE$ is an equivalence, hence the latter Kan extension coincides with $\textup{LKE}_FG$; also, the rightmost triangle commutes. We proved in the previous case that $H_0$ is left exact. Since $L$ is left exact by hypothesis, $H$ is left exact.
\end{proof}

\begin{cor}\label{big-stratified-analytification}
	We have an extended analytification functor \begin{equation}\label{prestack-analytification}(-)^\an:\PStrStk\to\PStrat\end{equation}

	which preserves finite limits.
\end{cor}
\begin{proof}
The extension from schemes to stacks follows from the first part of \cref{uni-tri}. The extension from stacks to pro-$\uni$-objects follows from the second part of \cref{uni-tri} and is defined as a right Kan extension. The extension from pro-$\uni$-objects is done by left Kan extension (i.e. covariant functoriality of $\Fun(-, \Grpd)$).

The fact that the functor preserves finite limits comes from the following argument. Recall first of all that $$(-)^\an:\Str\Sch^\lft_\C\to\Str\Top$$ preserves finite limits (\cref{an-left-exact}). Now, both left and right Kan extensions of functor which preserve finite limits do again preserve finite limits: for right Kan extensions this is straightforward, whereas for left Kan extensions it follows from \cref{LKE-are-left-exact}.
\end{proof}

\begin{prop}\label{algebraic-equals-topological-cons}Let $(Y,s)$ be a qcqs complex stratified scheme, locally of finite type, and $R$ a finite ring. Then there is an equivalence of $\infty$-categories $$\Cons^\fd(Y,s;R)\simeq \Cons^\fd(Y^\an,s^\an;R).$$\end{prop} 

\begin{proof}The claim follows from the proof \cite[Proposition 12.6.4]{Exodromy} (in turn building upon \cite[Th\'eor\`eme XVI. 4.1]{SGA4})\footnote{Although the statement is given for categories of constructible sheaves with respect to some algebraic stratification (and its analytic counterpart) the proof actually proceeds by establishing the equivalence when the stratification is fixed, and then passes to the colimit.}.
\end{proof}

\begin{rem}\label{profinite-GAGA}
This proof only relies on the Exodromy theorem both on the algebraic\footnote{By Exodromy on the algebraic side, we mean \cite[Corollary 13.2.12]{Exodromy} for finite coefficients, and the references below in this remark for profinite and $\ell$-adic coefficients. Note that a ``coherent scheme'' in the terminology of  \textit{loc.cit.} is a qcqs scheme (\cite[0.11.15]{Exodromy}).} and topological side. Hence, it extends to the profinite and $\ell$-adic case (cf. \cref{remark-schemes-profinite-coefficients}) by \cite[Theorem 13.7.8]{Exodromy} (for the finite/profinite case) and \cite[Theorem 13.8.8]{Exodromy} (for algebraic extensions of $\Q_\ell$).
\end{rem}

By forming limits from the statement of \cref{algebraic-equals-topological-cons}, we obtain, in the same setting, that:
 \begin{cor}\label{equivariant-GAGA}Let $(Y,s)$ be a finite dimensional complex stratified scheme, and $H$ a complex group scheme, locally of finite type, acting on it in a stratified way. There is an equivalence of $\infty$-categories $$\Cons^\fd_{H}(Y,s;R)\simeq \Cons^\fd_{H^\an}(Y^\an,s^\an;R).$$\end{cor}

\subsection{Monoidality, Kan extensions and correspondences}\label{subsection-correspondences}

\begin{notation}\label{recall-linear-categories}
	The (very large) $\infty$-category of presentable stable $\infty$-categories and all functors is denoted by $\Pr_\st$. The (very large) $\infty$-category of presentable stable $\infty$-categories and left adjoint functors is denoted by $\Prl_\st$. The $\infty$-category of presentable stable $\infty$-categories and right adjoint functors is denoted by $\Prr_\st$.

	Let $\cE$ be a presentable stable symmetric monoidal $\infty$-category. We denote:\begin{itemize}\item by $\Prl_\cE$ the $\infty$-sub-category of $\Mod_{\cE}(\Pr^\otimes_\st)$ spanned by the objects belonging to $\Mod_{\cE}(\Prlo_\st)$ and whose morphisms are $\cE$-linear functors admitting an $\cE$-linear right adjoint. This is, in particular, a non-full subcategory of $\Mod_{\cE}(\Prlo_\st)$;
		\item by $\Prr_\cE$ the $\infty$-sub-category of $\Mod_{\cE}(\Pr^\otimes_\st)$ spanned by the objects belonging to $\Mod_{\cE}(\Prlo_\st)$ and whose morphisms are $\cE$-linear functors admitting an $\cE$-linear left adjoint;
		\item by $\Prlr_\cE$ the $\infty$-sub-category of $\Mod_{\cE}(\Pr^\otimes_\st)$ spanned by the objects belonging to $\Mod_{\cE}(\Prlo_\st)$ and whose morphisms are $\cE$-linear functors admitting both an $\cE$-linear left adjoint and an $\cE$-linear right adjoint. 
	\end{itemize}

	For $\cE=\Mod_R$, $R$ being a discrete commutative ring, we denote this categories by $\Prl_R, \Prr_R, \Prlr_R$. For $R$ prodiscrete or $\ell$-adic, we use the same notation while taking $\cE=\Mod_R^\cont$.

\end{notation} 
\begin{rem}\label{Prro}		Let $\cc^\otimes$ be a symmetric monoidal $\infty$-category. We denote by $\cc^{\otimes\textup{-}\op}$ the ``pointwise-opposite'' symmetric monoidal structure on $\cc^\op$ (\cite{MO-Barwick-pointwise-op}, \cite{Barwick-pointwise-op}). For example, if $\cc$ has finite products, $\cc^{\times\textup{-}\op}\simeq( \cc^{\op})^\amalg$.

	$\Prl_\cE$ carries a natural symmetric monoidal structure inherited from $\Prlo_\st$, which we denote by $\Prlo_\cE$. The monoidal law is the relative tensor product $-\otimes_\cE-$.
	
	Let $\Prr$ be the $\infty$-category of presentable $\infty$-categories with right adjoint functors between them. We denote by $\Prro$ the symmetric monoidal structure induced by the equivalence $$(\Prl)^\op\simeq \Prr$$ i.e. $$\Prro=\Prlop.$$ 
	
	As usual, for any stable presentable symmetric monoidal $\infty$-category, we also have the $\cE$-linear variant $\Prro_\cE$, whose objects are the same as $\Mod_{\cE}(\Pr_\st)$ and whose morphisms are those functors admitting a left adjoint. Note that we also have an equivalence $$(\Prl_\cE)^\op\simeq \Prr_\cE.$$ 
	This relies on the fact that we required the right adjoints to be $\cE$-linear.
	
	Finally, $\Prlr_\cE$ also admits a symmetric monoidal structure, which is auto-dual and makes the inclusion functors $\Prlro_\cE\subset \Prlo_\cE$ and $\Prlro_\cE\subset \Prro_\cE$ both symmetric monoidal.\end{rem}

\begin{notation}\label{small-linear-categories}
	Let $\cE$ be a small stable symmetric monoidal $\infty$-category. We denote by $\cat_{\infty,\cE}=\Mod_{\cE}(\Cat^\ex)$ the $\infty$-category of small stable $\cE$-linear $\infty$-categories and exact $\cE$-linear functors. We denote the Cartesian monoidal structure on this $\infty$-category by $\cat_{\infty,\cE}^\times$. 
	
	For $R$ a commutative ring (discrete, prodiscrete or $\ell$-adic) we adopt the notation $\cat_{\infty, R}$ for $\cat_{\infty,\Perf(R)}$.
\end{notation}

\cref{exodromy} implies that

\begin{cor}Let $(X,s)\in\Con$, and $\cE$ be presentable stable $\infty$-category. Then $\Cons(X,s;\cE)$ is a presentable stable $\cE$-linear $\infty$-category and $\Cons^\fd(X,s;\cE)$ is a small $\cE^\omega$-linear $\infty$-category.\end{cor}

\begin{lem}\label{Exit-is-symmetric}The functor $$\Exit:\Con\to \Cat$$
	$$(Y,s:X\to P)\mapsto \Exit(Y,s)$$
	carries a symmetric monoidal structure when we endow both source and target with the Cartesian symmetric monoidal structure. In other words, $\Exit$ preserves finite products.\end{lem}
\begin{proof}
	Given two stratified topological spaces $Y,s:Y\to P, W, t:W\to Q$, in the notations of \cite[A.6]{HA}, consider the commutative diagram of simplicial sets $$
	\begin{tikzcd}
		\Sing^{P\times Q}(Y\times W) \arrow[d] \arrow[r] & \Sing(Y\times W) \arrow[d] \arrow[r, "\sim"] & \Sing(Y)\times \Sing(W) \arrow[dd] \\
		N(P\times Q) \arrow[r] \arrow[d, "\sim"']        & \Sing(P\times Q) \arrow[rd, "\sim"]         &                                    \\
		N(P)\times N(Q) \arrow[rr]                       &                                             & \Sing(P)\times \Sing(Q)   .       
	\end{tikzcd}$$
	The inner diagram is Cartesian by definition. Therefore the outer diagram is Cartesian, and we conclude that $\Sing^{P\times Q}(Y\times W)$ is canonically equivalent to $\Sing^P(Y)\times \Sing^Q(W)$. Since $\Sing^P(Y)$ models the $\infty$-category of exit paths of $Y$ with respect to $s$, and similarly for the other spaces, we conclude.
\end{proof}

\begin{lem}[{\cite[Remark 4.8.1.8 and Proposition 4.8.1.15]{HA}}]\label{lem:presheaf}
	Let $\cE$ be a presentable symmetric monoidal $\infty$-category. There exist symmetric monoidal functors $$\cP_\cE^{(*)}:\Cat^{\times\textup{-}\op}\to\Prlro_\cE$$
	$$\cP_{\cE,(\dashv)}:\Cat^\times\to \Prlo_\cE$$
	$$\cP_{\cE,(\vdash)}:\Cat^\times\to \Prro_\cE$$
	
	sending an $\infty$-category $\cc$ to the $\infty$-category of $\cE$-valued presheaves $\Fun(\cc^\op, \cE)$, and a functor $F:\cc\to \cD$ to the functor $$F^*:\Fun(\cD^\op, \cE)\to \Fun(\cc^\op, \cE)$$ induced by precomposition by $F$, resp. to the functor $$F_\dashv:\Fun(\cc^\op, \cE)\to \Fun(\cD^\op, \cE)$$ induced by left Kan extension along $F$, resp. to the functor $$F_\vdash:\Fun(\cc^\op, \cE)\to \Fun(\cD^\op, \cE)$$ induced by right Kan extension along $F$.
\end{lem}
\begin{proof}
	The functor $\cP_{\cE,(\dashv)}$ takes values in $\Prl_\cE$ since for each $F:\cc\to \cD$ the functor $F_{\dashv}:\Fun(\cc^\op, \cE)\to \Fun(\cD^\op, \cE)$ admits a right adjoint given by $F^*$, and both $F_{\dashv}$ and $F^*$ are $\cE$-linear.
	
	Let us start with the case $\cE=\cS$, the $\infty$-category of spaces. The existence of the (strong) monoidal structure on $\cP_{\cS,(\dashv)}$ follows from \cite[Proposition 4.8.1.8]{HA}. Indeed, if we take $\cK=\varnothing, \cK'=\{\textup{all simplicial sets}\}$ in \textit{loc.cit.}, the functor $\cc\mapsto \cP_{\cK}^{\cK'}(\cc)$ is what we call $\cP_{\cE,(\dashv)}$, since it is the functor sending $F:\cc\to \cD$ to the left Kan extension of the Yoneda embedding $\cc\to \cP(\cc)$ along $\cc\xrightarrow{F} \cD\to \cP(\cD)$ (see the proof of \cite[Proposition 5.3.6.2]{HTT}). 
	
	In the case of a general presentable $\infty$-category $\cE$, this follows from \cref{change-of-coefficients-presheaves}. 
	
	As for $\cP^{(*)}$ and $\cP_{(\vdash)}$, the claim follows from the fact that we have a symmetric monoidal equivalence
	$$\Prro\simeq \Prlop.$$\end{proof}

\begin{cor}\label{cons-monoidal}Let $\cE$ be a presentable stable $\infty$-category. There are well-defined symmetric monoidal functors 
	
	$$\Cons^{(*),\otimes}_\cE:\Con^{\times\textup{-}\op}\to \Prlro_\cE$$
	$$(Y,s)\mapsto \Cons(Y,s;\cE)$$
	$$f\mapsto f^*=- \circ \Exit(f).$$
	
	$$\Cons^\otimes_{(\dashv), \cE}:\Con^{\times}\to \Prlo_\cE$$
	$$(Y,s)\mapsto \Cons(Y,s;\cE)$$
	$$f\mapsto f_\pushcons:=\Lan_{\Exit(f)}.$$
	
	$$\Cons^\otimes_{(\vdash), \cE}:\Con^{\times}\to \Prro_\cE$$
	$$(Y,s)\mapsto \Cons(Y,s; \cE)$$
	$$f\mapsto f_\vdash:=\Ran_{\Exit(f)}.$$ 
\end{cor}
\begin{proof}The previous constructions provide us with a symmetric monoidal functor
	$$\Con^{\times\textup{-}\op}\xrightarrow{\Exit(-)} \Cat^{\times\textup{-}\op}\xrightarrow{(-)^\op} \Cat^{\times\textup{-}\op}\xrightarrow{\cP_\cE^{(*)}} \Prlro$$ sending $$(X,s)\mapsto \Fun(\Exit(X,s),\cE),$$$$f\mapsto f^*$$
and similarly for the other two cases.\end{proof}

\begin{prop}\label{homotopy-invariance-of-cons}
	The functors constructed in \cref{cons-monoidal} take stratified homotopy equivalences to equivalences of $\infty$-categories.
	\end{prop}

	\begin{proof} 
	Indeed, consider a stratified homotopy $H:[0,1]\times Y\to Y$. This map has the property that the compositions $\{0\}\times Y\to[0,1]\times Y\to Y$ is $f\circ g$ and $\{1\}\times Y\to[0,1]\times Y\to Y$ is $\id_Y$. Since the functor $\Exit(-)$ preserves products (\cref{Exit-is-symmetric}), one gets a map $\Exit([0,1])\times \Exit(Y)\to \Exit(Y)$ such that $$\Exit(\{0\})\times \Exit(Y)\to \Exit([0,1])\times \Exit(Y)\to \Exit(Y)$$ is $\Exit(f)\circ\Exit(g)$ and $$\Exit(\{1\})\times \Exit(Y)\to \Exit([0,1])\times \Exit(Y)\to \Exit(Y)$$ is $\Exit(\id_Y)$. But since $[0,1]$ is contractible and unstratified, $\Exit([0,1])$ is equivalent to the terminal $\infty$-category $*$, and the two compositions are equivalent as functors $\Exit(Y)\to\Exit(Y)$. Therefore $\Exit(f)$ is a left inverse to $\Exit(g)$. By repeating the argument on $K$ one obtains that $\Exit(g)$ is a left inverse to $\Exit(f)$.
\end{proof}

\begin{prop}\label{tri-equivalences}
The functors constructed in \cref{cons-monoidal} take maps in $\tri$ to equivalences of $\infty$-categories.
\end{prop}
\begin{proof}
Let $f:(X,s)\to (Y,t)$ be in $\tri$. Let us first tackle the case where $\cE=\cS$. We follow the same strategy of \cite[Lemma 4]{Stratified-Group-Actions}: one reduces to proving that for any $\cF\in \Cons(X,s;\cS), \cG\in \Cons(Y,t;\cS)$ the unit and counit maps $$\cF\to f^*f_\dashv\cF, \quad f_\dashv f^*\cG\to \cG$$ are equivalences. Both statements can be checked after pulling back to strata, e.g. by \cite[Lemma 8.2.10]{Exodromy} (see also \cite[Lemma A.9.18]{HA} and its application in the proof of \cite[Theorem A.9.3]{HA}): this is an instance of the recollement principle\footnote{By recollement principle we mean the fact that a category of constructible sheaves over a space $Y$ which is stratified by a closed subspace $Z$ and its complement $U$ may be reconstructed from the categories of local systems on the two strata plus some gluing data. In particular, the restriction functors $\Cons(Y)\to \Loc(Z), \Cons(Y)\to \Loc(U)$ are jointly conservative, which is precisely the content of the cited \cite[Lemma 8.2.10]{Exodromy}.}.

Let thus $i_\alpha:Y_\alpha\hookrightarrow Y$ be a stratum of $Y$, and $$\begin{tikzcd}X\times_YY_\alpha\arrow[r, "f_\alpha"]\arrow[d, "j_\alpha"]&Y_\alpha\arrow[d,"i_\alpha"]\\
X\arrow[r,"f"]&Y\end{tikzcd}$$ the resulting pullback diagram. Note that by hypothesis $X\times_YY_\alpha$ consists of a single stratum of $X$, and $f_\alpha$ is a trivial Serre fibration.

By smooth base change \eqref{smooth-base-change} we have that $i_\alpha f_\dashv f^*\cG\simeq (f_\alpha)_\dashv f_\alpha^*i_\alpha^*\cG$. Now $f_\alpha$ is a trivial Serre fibration by hypothesis, and hence the adjunction $(f_\alpha)_\dashv:\Loc(X\times_YY_\alpha)\to \Loc(Y_\alpha):f_\alpha^*$ between categories of local systems is an equivalence.

Again by smooth base change we have that $j_\alpha^*f^*f_\dashv\cF\simeq f_\alpha^*(f_\alpha)_\dashv j_\alpha^*\cF$, and we conclude as above by using that $f_\alpha$ is a trivial Serre fibration.

We recover the case of general coefficients either by a similar proof or by combining \cref{exodromy} and \cref{change-of-coefficients-presheaves}.\end{proof}

\begin{prop}\label{cons-hyperdescent}The functor $\Cons_{(\dashv),\cE}$ from \cref{cons-monoidal} satisfies hyperdescent.
\end{prop}
\begin{proof}This follows from the stratified Seifert-Van Kampen theorem \cite[Theorem A.7.1]{HA} and the fact that $$\Fun(-,\cE):\Cat\to \Prl_\cE$$ (with left Kan extension functoriality) preserves colimits.
\end{proof}

We now turn to some essential recalls of the theory of Gaitsgory and Rozenblyum's correspondences.

Let $\cc$ be an $(\infty,1)$-category, and $\vert,\horiz,\adm$ classes of morphisms with the properties \cite[Chapter 7, 1.1.1]{GRI}. The $(\infty,2)$-category of correspondences $\Corr(\cc)^\adm_{\vert,\horiz}$ is defined in \cite[Chapter 7, 1.2.5]{GRI}.\footnote{Note that we indeed need a definition that works at least for $\cc$ a $(2,1)$-category, since for example $\PCon$ is such.}
\begin{rem}[{\cite[Chapter 7, 1.3.3]{GRI}}]\label{remark-adm=isom}
	If $\adm=\isom$ is the class of equivalences in $\cc$, then $$\Corr(\cc)_{\vert,\horiz}^\isom$$ is an $(\infty,1)$-category, which we denote simply by $\Corr(\cc)_{\vert,\horiz}$. We will always be in this situation. However, some results from \cite{GRI} used in the proof of \cref{take-constructibles-final} rely on $(\infty,2)$-categorical constructions.
\end{rem}

\begin{rem}Let $\cc_\vert$ and $\cc_\horiz$ the subcategories of $\cc$ spanned by all objects and vertical or horizontal morphisms, respectively. There are embeddings $$\cc_\vert\to\Corr(\cc)_{\vert,\horiz}^\adm$$
	$$\cc_\horiz^\op\to\Corr(\cc)_{\vert,\horiz}^\adm.$$
\end{rem}

\begin{rem}[{\cite[Chapter 9, 2.1.3]{GRI}}]
	If $\cc^\otimes$ is a symmetric monoidal $\infty$-category and the classes $\vert,\horiz,\adm$ are closed under tensor product, then there is a symmetric monoidal structure on $\Corr(\cc)^\adm_{\vert,\horiz}.$ which we denote by $$\Corr(\cc)_{\vert,\horiz}^{\adm,\otimes}.$$ In our specific case, we will always consider Cartesian symmetric monoidal structures. Note that the class $\psubm$ from \cref{defin-subm-prestacks} is closed under cartesian product and pullback (\cref{subm-pullback}), hence it is a good class for the theory of correspondences.
\end{rem}

\begin{notation}
	If $\cc$ is an $(\infty,1)$-category, we denote by $\all$ the class of all morphisms in $\cc$.
\end{notation}

\begin{recall}When $\horiz=\all$ and $\adm=\isom$, the $(\infty,1)$-category $\Corr(\cc)^\otimes_{\vert,\horiz}$ is also treated in \cite{Mann} and \cite{Scholze}. In what follows, we will sometimes use some results formulated in this context.
\end{recall}

We will need the following lemma:

\begin{lem}\label{slices-in-cartesian-structures}
	Let $\cc^\times$ be a Cartesian symmetric monoidal $\infty$-category. Let $X,Y\in \cc$ be objects. Then the map $$\cc_{/X}\times\cc_{/Y}\to \cc_{/X\times Y}$$ is cofinal.
	
	If $F:\cc^\times\to \cD^\times$ is a symmetric monoidal functor between symmetric monoidal structures, and $X,Y\in \cD$ are objects, then the map $F_{/X}\times F_{/Y}\to F_{/X\times Y}$ is cofinal.
\end{lem}
\begin{proof}
	By Quillen's Theorem A, it suffices to prove that fibers are contractible. This can be straightforwardly checked by considering, for $Z\to \cX\times\cY$, $Z$ representable, the canonical factorization $$Z\xrightarrow{\Delta} Z\times Z\to \cX\times \cY$$ and proving that it is initial amongst all factorizations $$Z\xrightarrow{\Delta} X\times Y\to \cX\times \cY$$ induced by maps $X\to \cX, Y\to \cY, Z\to X, Z\to Y$, with $X,Y$ representable.
\end{proof}
	
\begin{thm}\label{take-constructibles-final}Let $\cE$ be a presentable stable symmetric monoidal $\infty$-category. There is a symmetric monoidal functor of $\infty$-categories $$\Cons_\cE^{\corr,\otimes}:\Corr(\PCon)_{ \all, \psubm}^\times\to \Prro_\cE$$ with the following properties:

\begin{enumerate}\item its restriction along $$(\ConStk)_\vert\hookrightarrow \Corr(\ConStk)_{\all,\subm'}\hookrightarrow \Corr(\PCon)_{\all, \psubm}$$ left Kan extends $\Cons_{\cE,(\vdash)}$ from \cref{cons-monoidal}, preserves colimits and sends maps in $\tri$ and $\she'$ to equivalences;
\item dually, its restriction along $$(\ConStk)_\horiz^\op\hookrightarrow \Corr(\ConStk)_{\all,\subm'}\hookrightarrow \Corr(\PCon)_{\all, \psubm}$$right Kan extends $\Cons_{\cE}^{(*)}$ from \cref{cons-monoidal}, preserves limits and sends maps in $\tri$ and $\she'$ to equivalences;
\item its restriction along $$\Pro_\tri(\ConStk)_\vert\hookrightarrow\Corr(\Pro_\tri(\ConStk))_{\all, \subm''}\hookrightarrow\Corr(\PCon)_{\all, \psubm},$$ resp. $$\Pro_\tri(\ConStk)_\horiz^\op\hookrightarrow\Corr(\Pro_\tri(\ConStk))_{\all, \subm''}\hookrightarrow\Corr(\PCon)_{\all, \psubm},$$
 extends the points (1), resp. (2), by $\tri$-invariance, and sends maps in $\she''$ to equivalences;
\item its restriction along $$(\PCon)_\vert\hookrightarrow\Corr(\PCon)_{\all, \psubm},$$ resp.
$$(\PCon)_\horiz^\op\hookrightarrow\Corr(\PCon)_{\all, \psubm},$$ left (resp. right) Kan extends the previous point, preserves colimits (resp. limits) and send morphisms in $\pshe$ to equivalences.

\end{enumerate}
\end{thm}

\begin{proof} 
	Our starting datum is the functor $\Cons^{(*)}:\Con^{\op}\to \Prl_\cE$ defined in \cref{cons-monoidal}. 
	Let us prove that this functor satisfies the right Beck-Chevelley condition \cite[Chapter 7, Definition 3.1.5]{GRI} with respect to the classes $\vert=\subm,\horiz=\all,\adm=\isom$. First of all, for each $\beta:(Y',s')\to (Y,s)$ in $\subm\subset\Mor(\Con)$, the functor $\beta^*:\Cons(Y,s)\to \Cons(Y',s')$ admits a left adjoint $\beta_\dashv$. Then, consider a cartesian diagram
	
	$$\begin{tikzcd}(X',t')\arrow[r,"\alpha_0"]\arrow[d,"\beta_1"]&(X,t)\arrow[d,"\beta_0"]\\(Y',s')\arrow[r,"\alpha_1"]& (Y,s)\end{tikzcd}$$ where $\alpha_0,\alpha_1 \in \all, \beta_0,\beta_1\in \subm$. Then we want to prove that the base-change map \begin{equation}\label{smooth-base-change}\beta_{1\dashv}\alpha_0^*\to \alpha_1^*\beta_{0\dashv}\end{equation} is an equivalence of functors $\Cons(X,t;\cE)\to \Cons(Y',s';\cE)$. But since $\beta_0,\beta_1$ belong to $\subm$, \cref{lower-sharp-is-constructible-left-adjoint} implies that $\beta_{0\dashv}\simeq \beta_{0\#}, \beta_{1\dashv}\simeq \beta_{1\#}$, and thus it suffices to invoke \cref{recall-smooth-base-change-sheaves}.
	
	
	We now want to obtain an extension at the level of $(\infty,2)$-categories
	
	\begin{equation}\label{corr-extension-first}\begin{tikzcd}
		{\Con^\op} & \Prl_\cE \\
		{\Corr(\Con)_{\subm,\all}}
		\arrow[dotted, from=2-1, to=1-2]
		\arrow[from=1-1, to=2-1]
		\arrow[from=1-1, to=1-2]
	\end{tikzcd}.\end{equation}
	
	
	which encodes $(-^*,-_\dashv)$-functoriality.
	As explained to us by Lucas Mann, one can proceed as in the proof of \cite[Proposition A.5.10]{Mann} with $I=\subm, P=\isom$, and without the assumption that $I$ satisfies property (d) in \textit{loc. cit.}: one does not need \cite[Theorem 5.4]{Liu-Zheng-a} and instead of $\delta_{3,\{...\}}^*$ one works with $\delta_{2,\{...\}}^*$. Then the main observation is that one can invert edges in dimension 1 using \cite[Lemma 1.4.4]{Liu-Zheng-b}.

	The extended functor \eqref{corr-extension-first} carries a canonical symmetric monoidal structure, because $\Cons_{(\dashv)}:\Str\Top_\con^\times\to \Prlo_\cE$ is symmetric monoidal (\cref{cons-monoidal}) and thus we can apply\footnote{Alternatively, from \cite[Proposition A.5.10]{Mann} we already get a lax-monoidal structure, and one can check that it is indeed strong symmetric monoidal.} \cite[Chapter 9, Proposition 3.1.5]{GRI} with $\vert=\subm, \horiz=\adm=\all, \coadm=\isom$.
We will now perform a series of extensions:

\begin{itemize}

	\item Let us consider the Yoneda embedding $F:\Con\to \Str\TStk_\con$. We want to obtain a symmetric monoidal extension 

	\begin{equation}\label{corr-extension-to-stacks-rep}\begin{tikzcd}
			\Corr(\Con)_{\subm,\all}\arrow[d]\arrow[r]& \Prl_\cE\\
			\Corr(\Str\TStk_\con)_{\widetilde\subm,\all}\arrow[ru, dotted]&
	\end{tikzcd}\end{equation}
	(see \cref{defin-subm-prestacks} for the notation) which preserves limits in both $\vert$ and $\horiz$.

	The functor $F$ satisfies the conditions of \cite[Chapter 8, 6.1.1 and Theorem 6.1.5]{GRI}, (with $\mathbf C=\Con, \mathbf D=\PCon,\vert=\subm,\widetilde\subm$ respectively, $\horiz=\all, \adm=\isom$ for both $\mathbf C$ and $\mathbf D$), because:\begin{itemize}
		\item it preserves finite limits;
		\item it sends $\subm$ to $\widetilde\subm$ by definition of $\widetilde\subm\subset\Mor(\ConStk)$;
		\item for each $(Y,s)\in \Con$, the functor $((\Con)_\subm)_{/(Y,s)}\to ((\ConStk)_{\subm'})_{/F(Y,s)}$ is an equivalence. Indeed, since $F$ is fully faithful, it suffices to prove that for any $\cY\in \ConStk, \phi:\cY\to F(Y,s)$ in $\ConStk$, then $\cY$ belongs to the essential image of $F$. But this is true since any morphism in $\widetilde\subm$ is assumed to be representable, cf. \cref{defin-subm-prestacks}.
	\end{itemize}

	Therefore, we may apply \cite[Theorem 6.1.5]{GRI} to \eqref{corr-extension-first} and obtain a horizontal right Kan extension (in the terminology of that theorem) as in \eqref{corr-extension-to-stacks-rep}.

	In particular, when restricted to $(\Con)_\horiz^\op\to (\ConStk)_\horiz^\op$, this is the right Kan extension. 
	Also, the extended functor preserves limits in both $\vert$ and $\horiz$.

	As for the symmetric monoidal structure, we proceed as follows\footnote{Incidentally, the same verifications show that \cite[Proposition A.5.16]{Mann} can be applied to obtain the same result.}. By \cite[Chapter 9, Proposition 3.2.4]{GRI}, we obtain a right-lax monoidal functor $$\Corr(\ConStk)_{\subm, \all}^\times\to \Prlo_\cE.$$

	Let us prove that the obtained functor is indeed strongly monoidal. It suffices to prove that, for every $\cX,\cY\in \Str\TStk_\con,$ the map 

	$$\Cons(\cX;\cE)\otimes_\cE \Cons(\cY;\cE)\to \Cons(\cX\times \cY;\cE)$$
	 is an equivalence. 
	 But this map can be presented as the chain of equivalences 
	 \begin{equation}\label{lax-is-strong}\begin{gathered}\Cons(\cX;\cE)\otimes_\cE\Cons(\cY;\cE)={\colim_{(\Con)_{/\cX}}}^{\hspace{-10pt}\Prl}\Cons(X;\cE)\otimes_\cE {\colim_{(\Con)_{/\cY}}}^{\hspace{-10pt}\Prl}\Cons(Y;\cE)\simeq \\
		{\colim_{(\Con)_{/\cX}\times(\Con)_{/\cY}}}^{\hspace{-35pt}\Prl}\Cons(X;\cE)\otimes_\cE\Cons(Y;\cE)\simeq\\ {\colim_{(\Con)_{/\cX}\times(\Con)_{/\cY}}}^{\hspace{-35pt}\Prl}\Cons(X\times Y;\cE)\simeq \\
		{\colim_{Z\in(\Con)_{/\cX\times \cY}}}^{\hspace{-20pt}\Prl}\Cons(Z;\cE)\simeq \Cons(\cX\times\cY;\cE)\end{gathered}\end{equation}
	where the second-to-last equivalence holds by \cref{slices-in-cartesian-structures}.

\item The idea for this point comes from a discussion with Michele Pernice. We want to extend the functor \eqref{corr-extension-to-stacks-rep} to $\Corr(\Str\TStk_\con)_{\subm',\all}$, that is, we want to enlarge the class of vertical morphisms from representable smooth submersions to all smooth submersions. To this end, we want to apply \cite[Proposition A.5.14]{Mann} to the functor constructed in \eqref{corr-extension-to-stacks-rep}. The target category can be upgraded from $\Cat$ to $\Prl_\cE$: indeed, the upgrade to $\Prl$ is already in the proof of \textit{loc.cit.}, and adding the linear structure is straightforward. Let us thus check the conditions of \cite[Proposition A.5.14]{Mann}. In the notations of \textit{loc.cit.}, let $E=\widetilde\subm,E'=\subm', S\subset E$ be the class of smooth covers, i.e. maps in $\widetilde\subm$ whose source is representable and which are surjective.

\begin{enumerate}[a]
	\item For every $\cX\in\Str\TStk_\con$, $\Cons(\cX;\cE)$ is presentable.
	\item $\Cons_{(\dashv),\cE}:\Str\TStk_\con\to \Prl_\cE$ satisfies $S$-descent by construction (it is left-Kan-extended from stratified topological spaces).
	\item By definition of $\subm'$, for every map $f:\cX\to \cY$ in $E'$ there is a square 
	$$\begin{tikzcd}X\arrow[r, "f'"]\arrow[d, "g"]&Y\arrow[d, "g'"]\\
	\cX\arrow[r, "f"]&\cY
	\end{tikzcd}$$ where the top map is in $E$ (actually in $\subm$) and the vertical arrows are in $S$. Hence, the composition $f\circ g$ belongs to $E$.
	\item We want to prove that pullbacks of edges in $S$ remain smooth covers, and they are computed in the same way in $(\Str\TStk_\con)_{\widetilde\subm}$ and in $(\Str\TStk_\con)$. Indeed, the take a square
	\[\begin{tikzcd}
	X & Y \\
	\cX & \cY
	\arrow["{f'}", from=1-1, to=1-2]
	\arrow["{g'}"', from=1-1, to=2-1]
	\arrow["g"', from=1-2, to=2-2]
	\arrow["f", from=2-1, to=2-2]
\end{tikzcd},\]
Cartesian in $\Str\TStk_\con$, with $f\in E, g\in S$. Then $g'\in S$ because smooth covers are stable under pullback, and $f'\in E$ because smooth and representable maps are as well. We are left to prove that for any $\cX'\in \Str\TStk_\con$, equipped with maps $\cX'\to Y, \cX'\to \cX$ making the diagram 

\[\begin{tikzcd}
	{\cX'} \\
	& X & Y \\
	& \cX & \cY
	\arrow[bend left, from=1-1, to=2-3]
	\arrow[bend right, from=1-1, to=3-2]
	\arrow["{f'}", from=2-2, to=2-3]
	\arrow["{g'}"', from=2-2, to=3-2]
	\arrow["g"', from=2-3, to=3-3]
	\arrow["f", from=3-2, to=3-3]
\end{tikzcd}\]

commute, the natural map $\cX'\to X$ is smooth and representable (i.e. it belongs to $E$). Representability amounts to representability of $\cX'$ (since $X$ is represetnable), which follows from the fact that the map $\cX'\to Y$ is representable with representable target. Smoothness follows from the fact that $g'$ is a smooth cover and smoothness passes to smooth covers by definition.
\end{enumerate}

We can thus apply the result and obtain an extension of \eqref{corr-extension-to-stacks-rep} to a lax-monoidal functor\begin{equation}\label{corr-extension-to-stacks}\Corr(\Str\TStk_\con)_{\subm', \all}\to \Prl_\cE.\end{equation}
Symmetric monoidality is established in a similar way as in the previous point.

\item Let us consider the localization functor $$G:\ConStk\to \ConStk[\tri^{-1}].$$ By restriction from \eqref{corr-extension-to-stacks}, we get a functor $$(\ConStk)_\horiz^\op\to \Prl_\cE$$ satisfying the right Beck-Chevalley condition for $(\vert=\subm,\horiz=\all)$. This functor is right-Kan-extended from $\Cons(-,\cE)^{(*)}:\Con^\op\to \Prl_\cE$, hence it sends $\tri$ to equivalences by \cref{tri-equivalences}, and therefore it factors as the composition of $G$ and a functor $$\ConStk[\tri^{-1}]^\op\to \Prl_\cE$$ which also satisfies the right Beck-Chevalley condition with respect to $(\vert=\subm,\horiz=\all)$ (note that $\tri\subset \subm'$). Hence, again by \cite[Chapter 7, Theorem 3.2.2.(b)]{GRI}, it induces a functor $$\Corr(\ConStk[\tri^{-1}])_{\subm',\all}\to \Prl_\cE$$ (where we also call $\subm'$ the class generated by $\subm'$ in the localization), compatible with all of the above. This functor carries a symmetric monoidal structure by the same arguments as above.

\item Note that the functor $\ConStk[\tri^{-1}]\to \Pro_\tri(\ConStk[\tri^{-1}])$ is an equivalence (since $\tri\subset\isom$ in $\ConStk[\tri^{-1}]$), symmetric monoidal with respect to the Cartesian structures. On the other hand, there is a functor $\Pro(G):\Pro_\tri(\ConStk)\to \Pro_\tri(\ConStk[\tri^{-1}]),$ also symmetric monoidal. This induces a functor $$\Corr(\Pro_\tri(\ConStk))_{\subm'', \all}\to \Corr(\ConStk[\tri^{-1}])_{\subm',\all}$$ which is symmetric monoidal because pro-objects are compatible with Cartesian monoidal structures. Note that we cannot argue anymore that the restriction to horizontal morphisms $$\Pro_\tri(\ConStk)^\op\to \ConStk[\tri^{-1}]^\op$$ preserves limits, since $\Pro(G)$ need not preserve colimits.

\item Consider the Yoneda embedding $H:\Pro_\tri(\ConStk)\to \PCon.$ By the exact same arguments used in the case of $F$ above, this induces a horizontal right Kan extension

	\begin{equation}\label{corr-extension-to-prestacks}\begin{tikzcd}
			\Corr(\Pro_\tri(\ConStk))_{\subm'',\all}\arrow[d]\arrow[r]& \Prl_\cE\\
			\Corr(\PCon)_{\psubm,\all}\arrow[ru, dotted]&
	\end{tikzcd}\end{equation}
which is again symmetric monoidal.
\end{itemize}

	From this, by passing to opposite categories and opposite monoidal structures as in \cref{Prro}, we obtain the sought-after symmetric monoidal functor of $(\infty,1)$-categories \begin{equation}\Cons^{\corr,\otimes}_\cE:\Corr(\PCon)_{\all,\subm}^\times\to \Prro_\cE\end{equation}
	which encodes $(-^*,-_\vdash)$-functoriality and satisfies the conditions of the statement.
\end{proof}

We conclude this Appendix with a Lemma regarding the existence of exceptional inverse image functoriality for constructible sheaves.

\begin{lem}\label{upper-shriek-constructibles}Let $(Y,t)$ be a conically stratified space and $f:(X,s)\to (Y,t)$ an inclusion of a union of strata. Let $\cE$ be a presentable stable category. Then the functor $$f^!:\Sh(Y;\cE)\to S\Sh(X\cE)$$ restricts to a functor $$f^!:\Cons(Y,t;\cE)\to \Cons(X,s;\cE).$$
\end{lem}
\begin{proof}
If $f$ is an open embedding, then $f^!\simeq f^*$ and the result is already known. Thanks to this, and since all strata are locally closed, one can reduce to the case of a closed embedding of a union of strata. Let $j:(U,u)\to (Y,t)$ be the embedding of the open complement of $(X,s)$. Note that $(U,u)$ is in turn a union of strata of $(Y,t)$. Let $\cF$ be a constructible sheaf on $(Y,t)$. There is an exact sequence \begin{equation}\label{open-closed-exact-sequence}f_!f^!\cF\to \cF\to j_!j^*\cF.\end{equation}
Since $f$ is a closed embedding, $f_!\simeq f_*$ and furthermore $f_*$ is fully faithful. Hence, the counit transformation $f^*f_*\to \id$ is an equivalence, and by applying $f^*$ to \eqref{open-closed-exact-sequence} one obtains that $f^!\cF\simeq \textup{fib}(f^*\cF\to f^*j_*j^*\cF)$. In particular, $f^!$ preserves constructible sheaves if and only if $j_*$ does, and this is true by \cite[Proposition 6.35]{Exodromy-PT}.
\end{proof}

\end{appendices}

\bibliographystyle{alpha}
\bibliography{E3.bib}

\begin{thebibliography}{CvdHS22}

\bibitem[AG15]{Arinkin-Gaitsgory}
Dimitri Arinkin and Dennis Gaitsgory.
\newblock {Singular support of coherent sheaves, and the geometric Langlands
  conjecture}.
\newblock {\em Selecta Mathematica 21, 1-199}, 2015.

\bibitem[AGV72]{SGA4}
Michael Artin, Alexander Grothendieck, and Jean-Louis Verdier.
\newblock {\em {Th\'eorie des topos et cohomologie \'etale des sch\'emas.
  S\'eminaire de G\'eom\'etrie Alg\'ebrique du Bois-Marie 1963–1964 (SGA
  4)}}.
\newblock Lecture Notes in Mathematics 269, Springer, 1972.
\newblock Avec la collaboration de N. Bourbaki, P. Deligne et B. Saint-Donat.

\bibitem[AR23]{Achar-Riche}
Pramod Achar and Simon Riche.
\newblock {Central sheaves on affine flag varieties}.
\newblock \url{https://lmbp.uca.fr/~riche/central.pdf}, 2023.

\bibitem[BBDG18]{BBDG}
Alexander Beilinson, Joseph Bernstein, Pierre Deligne, and Ofer Gabber.
\newblock {\em {Faisceaux pervers}}.
\newblock Societ\'e Math\'ematique de France, Ast\'erisque 100, 2018.

\bibitem[BD05]{BD-Hitchin}
Alexander Beilinson and Vladimir Drinfeld.
\newblock {Quantization of Hitchin's integrable system and Hecke eigensheaves}.
\newblock
  \url{http://www.math.uchicago.edu/~drinfeld/langlands/QuantizationHitchin.pdf},
  2005.

\bibitem[Beh04]{Behrend}
Kai~A. Behrend.
\newblock {Derived $\ell$-adic categories for algebraic stacks}.
\newblock {\em Memoirs of the American Mathematical Society 774}, 2004.

\bibitem[BF07]{Bezrukavnikov-Finkelberg}
Roman Bezrukavnikov and Mikhail Finkelberg.
\newblock {Equivariant Satake category and Kostant-Whittaker reduction}.
\newblock \url{https://arxiv.org/abs/0707.3799}, 2007.

\bibitem[BGH20]{Exodromy}
Clark Barwick, Saul Glasman, and Peter Haine.
\newblock Exodromy.
\newblock \url{https://www.maths.ed.ac.uk/~cbarwick/papers/exodromy_book.pdf},
  2020.

\bibitem[BGN18]{Barwick-pointwise-op}
Clark Barwick, Saul Glasman, and Denis Nardin.
\newblock Dualizing cartesian and cocartesian fibrations.
\newblock {\em Theory and Applications of Categories, Vol. 33, No. 4, pp.
  67–94}, 2018.

\bibitem[BL94]{Bernstein-Lunts}
Joseph Bernstein and Valery Lunts.
\newblock {\em {Equivariant Functors and Sheaves}}.
\newblock Springer, 1994.

\bibitem[BL95]{BL}
Arnaud Beauville and Yves Laszlo.
\newblock {Un lemme de descente}.
\newblock {\em C. R. Acad. Sci. Paris S\'er. I Math.}, 320(3):335--340, 1995.

\bibitem[BR18]{Baumann-Riche}
Pierre Baumann and Simon Riche.
\newblock {Notes on the geometric Satake equivalence}.
\newblock \url{https://arxiv.org/abs/1703.07288}, 2018.

\bibitem[BZFN10]{BZFN}
David Ben-Zvi, John Francis, and David Nadler.
\newblock {Integral transforms and Drinfeld centers in derived algebraic
  geometry}.
\newblock {\em J. Amer. Math. Soc. 23, 909-966}, 2010.

\bibitem[BZNP17]{BenZvi-Nadler-Preygel}
David Ben-Zvi, David Nadler, and Anatoly Preygel.
\newblock Integral transforms for coherent sheaves.
\newblock {\em Journal of the European Mathematical Society, 19 (12),
  3763–3812}, 2017.

\bibitem[BZSV23]{Relative-Langlands}
David Ben-Zvi, Yiannis Sakellaridis, and Akshay Venkatesh.
\newblock {Relative Langlands duality}.
\newblock \url{https://www.math.ias.edu/~akshay/research/BZSVpaperV1.pdf},
  2023.

\bibitem[CL21]{Cepek-Lejay}
Anna Cepek and Damien Lejay.
\newblock On the topologies of the exponential.
\newblock \url{https://arxiv.org/abs/2107.11243}, 2021.

\bibitem[CN18]{Chen-Nadler-Quasi-Maps}
Tsao-Hsien Chen and David Nadler.
\newblock {Real and symmetric quasi-maps}.
\newblock \url{https://arxiv.org/abs/1805.06564}, 2018.

\bibitem[CN24]{Chen-Nadler-Real-Groups}
Tsao-Hsien Chen and David Nadler.
\newblock {Real groups, symmetric varieties and Langlands duality}.
\newblock \url{https://arxiv.org/abs/2403.13995v1}, 2024.

\bibitem[CR23]{Campbell-Raskin-paper}
Justin Campbell and Sam Raskin.
\newblock {Langlands Duality for the Beilinson-Drinfeld Grassmannian}.
\newblock \url{https://arxiv.org/abs/2310.19734}, 2023.

\bibitem[CvdHS22]{Cass-BDGrassmannian}
Robert Cass, Thibaud van~den Hove, and Jakob Scholbach.
\newblock {The Geometric Satake Equivalence for integral motives}.
\newblock \url{https://arxiv.org/pdf/2211.04832.pdf}, 2022.

\bibitem[DK19]{DK1}
Tobias Dyckerhoff and Mikhail Kapranov.
\newblock {\em Higher Segal Spaces I}.
\newblock Lecture Notes in Mathematics. Springer, 2019.

\bibitem[DL23]{Sylvain-Ran}
Sylvain Douteau and Marie Labeye.
\newblock {Topologies et distances sur les espace de Ran}.
\newblock Draft available in French at \url{https://sdouteau.cygale.net/},
  2023.

\bibitem[Gin95]{Ginzburg}
Victor Ginzburg.
\newblock {Perverse sheaves on a Loop group and Langlands' duality}.
\newblock \url{https://arxiv.org/abs/alg-geom/9511007}, 1995.

\bibitem[GL]{Gaitsgory-Tamagawa}
Dennis Gaitsgory and Jacob Lurie.
\newblock {Weil's Conjecture for Function Fields, draft of the complete
  version}.
\newblock \url{https://people.math.harvard.edu/~lurie/papers/tamagawa.pdf}.

\bibitem[GL18]{WC}
Dennis Gaitsgory and Jacob Lurie.
\newblock {\em Weil's Conjecture for Function Fields}.
\newblock Princeton University Press, 2018.

\bibitem[GR17]{GRI}
Dennis Gaitsgory and Nick Rozenblyum.
\newblock {\em {A Study in Derived Algebraic Geometry Vol. I}}.
\newblock AMS, 2017.

\bibitem[HN24]{Stratified-Group-Actions}
Peter Haine and Guglielmo Nocera.
\newblock {A note on stratified group actions}.
\newblock
  \url{https://www.ihes.fr/~nocera/Notes/Stratified%20group%20actions.pdf},
  2024.

\bibitem[HPT22]{HPT}
Peter~J. Haine, Mauro Porta, and Jean-Baptiste Teyssier.
\newblock {The homotopy-invariance of constructible sheaves}.
\newblock \url{https://arxiv.org/abs/2010.06473}, 2022.

\bibitem[HY19]{HY}
Jeremy Hahn and Allen Yuan.
\newblock {Multiplicative structure in the stable splitting of $\Omega
  \textup{SL}_n(\C)$}.
\newblock {\em Advances in Mathematics 348 (4)}, 2019.

\bibitem[KMT74]{Unipotent-groups}
Tatsuji Kambayashi, Masayoshi Miyanishi, and Mitsuhiro Takeuchi.
\newblock {\em {Unipotent Algebraic Groups}}.
\newblock Lecture Notes in Mathematics 414, Springer-Verlag, 1974.

\bibitem[KMW18]{Kamnitzer-Mutiah-Weekes}
Joel Kamnitzer, Dinakar Mutiah, and Alex Weekes.
\newblock {On a reducedness conjecture for spherical Schubert varieties and
  slices in the affine Grassmannian}.
\newblock {\em Transformation Groups 23, 707-772}, 2018.

\bibitem[KS94]{Kashiwara-Schapira}
Masaki Kashiwara and Pierre Schapira.
\newblock {\em {Sheaves on Manifolds}}.
\newblock Springer-Verlag, 1994.

\bibitem[KW01]{KW}
Reinhardt Kiehl and Rainer Weissauer.
\newblock {\em {Weil Conjectures, Perverse Sheaves and l'adic Fourier
  Transform}}.
\newblock Springer, 2001.

\bibitem[LO06]{Laszlo-Olsson-perverse}
Yves Laszlo and Martin Olsson.
\newblock {Perverse sheaves on Artin stacks}.
\newblock \url{https://arxiv.org/abs/math/0606175}, 2006.

\bibitem[Lur09]{HTT}
Jacob Lurie.
\newblock {\em {Higher Topos Theory}}.
\newblock Princeton University Press, 2009.

\bibitem[Lur11]{DAG-VII}
Jacob Lurie.
\newblock {Derived Algebraic Geometry VII: Spectral Schemes}.
\newblock \url{https://www.math.ias.edu/~lurie/papers/DAG-VII.pdf}, 2011.

\bibitem[Lur17]{HA}
Jacob Lurie.
\newblock {Higher Algebra}.
\newblock \url{http://people.math.harvard.edu/~lurie/papers/HA.pdf}, 2017.

\bibitem[LZ12a]{Liu-Zheng-b}
Yifei Liu and Wheizhe Zheng.
\newblock {Enhanced six operations and base change theorem for higher Artin
  stacks}.
\newblock \url{https://arxiv.org/abs/1211.5948}, 2012.

\bibitem[LZ12b]{Liu-Zheng-a}
Yifei Liu and Wheizhe Zheng.
\newblock {Gluing restricted nerves of $\infty$-categories}.
\newblock \url{https://arxiv.org/abs/1211.5294}, 2012.

\bibitem[Man22]{Mann}
Lucas Mann.
\newblock {A p-Adic 6-Functor Formalism in Rigid-Analytic Geometry}.
\newblock \url{https://arxiv.org/pdf/2206.02022}, 2022.

\bibitem[Mat70]{TopStab}
John Mather.
\newblock Notes on topological stability.
\newblock {\em Published as Bull. AMS, Volume 49, Number 4, October 2012, Pages
  475-506, Originally published as notes of lectures at the Harvard
  University}, 1970.

\bibitem[MO14]{MO-Gr-is-Whitney}
Olivier Straser~on Math~Overflow.
\newblock {Whitney stratification and affine Grassmanian}.
\newblock \url{https://mathoverflow.net/q/154594}, 2014.

\bibitem[MO15]{MO-Barwick-pointwise-op}
Johnatan Beardsley~on Math~Overflow.
\newblock Opposite symmetric monoidal structure on an infinity category.
\newblock \url{https://mathoverflow.net/q/199007}, 2015.

\bibitem[MO21]{Sawin-answer}
Will Sawin~on Math~Overflow.
\newblock {Are equivariant perverse sheaves constructible with respect to the
  orbit stratification?}
\newblock \url{https://mathoverflow.net/q/384551}, 2021.

\bibitem[MSE20]{MO-refining-stratifications}
KReiser~on Math Stacks~Exchange.
\newblock Refinement of two stratifications of an algebraic variety.
\newblock \url{https://math.stackexchange.com/q/3679894}, 2020.

\bibitem[MV07]{MV}
Ivan Mirkovic and Kari Vilonen.
\newblock {Geometric Langlands duality and representations of algebraic groups
  over commutative rings}.
\newblock {\em Annals of Mathematics, 166, 95–143}, 2007.

\bibitem[Nad05]{Nadler-Perverse-real}
David Nadler.
\newblock {Perverse sheaves on real loop Grassmannians}.
\newblock {\em Invent. math. 159, 1–73}, 2005.

\bibitem[NL19]{Nand-Lal-thesis}
Stephen Nand-Lal.
\newblock {\em {A simplicial approach to stratified homotopy theory}}.
\newblock PhD thesis, University of Liverpool, 2019.

\bibitem[NP24a]{NPnote}
Guglielmo Nocera and Michele Pernice.
\newblock {A note on convolution over double quotients of group schemes}.
\newblock \url{https://www.ihes.fr/~nocera/Notes/Convolution.pdf}, 2024.

\bibitem[NP24b]{WM}
Guglielmo Nocera and Morena Porzio.
\newblock {On the homotopy type of the Beilinson-Drinfeld Grassmannian}.
\newblock \url{https://www.ihes.fr/~nocera/Papers/BDGrassmannian.pdf}, 2024.

\bibitem[NV23]{Whitney}
Guglielmo Nocera and Marco Volpe.
\newblock {Whitney stratifications are conically smooth}.
\newblock {\em Selecta Mathematica New Ser. 29, n. 68}, 2023.

\bibitem[PT22]{Exodromy-PT}
Mauro Porta and Jean-Baptiste Teyssier.
\newblock {Topological exodromy with coefficients}.
\newblock \url{https://arxiv.org/abs/2211.05004}, 2022.

\bibitem[Ras18]{Raskin-Principal-II}
Sam Raskin.
\newblock {Chiral principal series categories II: The factorizable Whittaker
  category}.
\newblock \url{https://gauss.math.yale.edu/~sr2532/cpsii.pdf}, 2018.

\bibitem[Ray71]{SGA1-XII}
Michelle Raynaud.
\newblock {G\'eometrie alg\'ebrique et g\'eometrie analytique}.
\newblock {\em SGA1, Expos\'e XII}, 1971.

\bibitem[Rei12]{Reich}
Ryan~Cohen Reich.
\newblock {Twisted factorizable Satake equivalence via gerbes on the
  factorizable Grassmannian}.
\newblock {\em Represent. Theory 16, 345-449}, 2012.

\bibitem[Ric14]{Richarz}
Timo Richarz.
\newblock {A new approach to the Geometric Satake Equivalence}.
\newblock {\em Doc. Math. 19, pp. 209–246}, 2014.

\bibitem[Sat63]{Satake}
Ichiro Satake.
\newblock {Theory of spherical functions on reductive algebraic groups over
  $p$-adic fields}.
\newblock {\em Publications Mathématiques de l'IH\'ES, Volume 18, pp. 5-69},
  1963.

\bibitem[Sch23]{Scholze}
Peter Scholze.
\newblock {Six-functor formalisms}.
\newblock \url{https://people.mpim-bonn.mpg.de/scholze/SixFunctors.pdf}, 2023.

\bibitem[SPA]{Stacks-Project}
The Stacks Project~Authors.
\newblock {The Stacks Project}.
\newblock \url{https://stacks.math.columbia.edu}.

\bibitem[Tao20]{James-GrRan}
James Tao.
\newblock {$\textup{Gr}_{\textup{Ran}}$ is reduced}.
\newblock \url{https://arxiv.org/pdf/2011.01553}, 2020.

\bibitem[Vak22]{Vakil}
Ravi Vakil.
\newblock {The Rising Sea. Foundations of Algebraic Geometry}.
\newblock \url{https://math.stanford.edu/~vakil/216blog/FOAGaug2922public.pdf},
  2022.

\bibitem[Vol21]{Marco-six-functors}
Marco Volpe.
\newblock The six operations in topology.
\newblock \url{https://arxiv.org/abs/2110.10212}, 2021.

\bibitem[Wed22]{Wedhorn}
Thorsten Wedhorn.
\newblock {Definition of the Satake category and convolution, Talk 9 of the
  Clermont-Darmstadt workshop on the Geometric Satake Equivalence}.
\newblock \url{https://lmbp.uca.fr/~riche/Notes-workshop/workshop-Talk9.pdf},
  2022.

\bibitem[Zhu16]{Zhu}
Xinwen Zhu.
\newblock {An introduction to affine Grassmannians and to the geometric Satake
  equivalence}.
\newblock \url{http://arxiv.org/abs/1603.05593v2}, 2016.

\bibitem[Zhu17]{Zhu-mixed}
Xinwen Zhu.
\newblock {Affine Grassmannians and the Geometric Satake in mixed
  characteristic}.
\newblock {\em Annals of Mathematics Second Series, Vol. 185, No. 2 pp.
  403-492}, 2017.

\end{thebibliography}

\end{document}